\documentclass{article}
\usepackage[nottoc]{tocbibind}
\usepackage[titletoc,title]{appendix}
\usepackage[utf8]{inputenc}
\usepackage{times}

\usepackage{authblk}

\usepackage[pdftex]{graphicx}
\usepackage[pdftex,dvipsnames,usenames]{color}

\usepackage{enumitem}
\usepackage{mathtools}

\usepackage{cancel}

\usepackage[english]{babel}

\usepackage[sc]{mathpazo}
\usepackage[hyperindex,breaklinks]{hyperref}

\hypersetup{
  colorlinks   = true, 
  urlcolor     = blue, 
  linkcolor    = blue, 
  citecolor   = red 
}

\usepackage{amsmath,amssymb,amsfonts,mathrsfs, amsthm}

\usepackage{graphicx}
\usepackage[margin=.94in]{geometry}

\usepackage{soul}

\usepackage{pdfpages}

\usepackage{lipsum}
\usepackage{csquotes}
\usepackage{thmtools}

\let\OLDthebibliography\thebibliography
\renewcommand\thebibliography[1]{
  \OLDthebibliography{#1}
  \setlength{\parskip}{0pt}
  \setlength{\itemsep}{4pt plus 0.3ex}
}

\usepackage{amsfonts,amsmath,amstext,amssymb,amsthm}
\usepackage{enumitem}
\usepackage[english]{babel}
\usepackage{fancybox}
\usepackage{xspace}
\usepackage{amscd}
\usepackage[all]{xy}
\usepackage[svgnames]{xcolor}
\usepackage{calc}

\usepackage{esint}
\usepackage{authblk}

\usepackage{cases}
\usepackage{eqnarray}

\usepackage{accents}

\newcommand{\B}{\mathcal{B}}

\def\XXint#1#2#3{{\setbox0=\hbox{$#1{#2#3}{\int}$ }
\vcenter{\hbox{$#2#3$ }}\kern-.56\wd0}}

\newcommand{\Qb}{\mathcal{Q}}

\newcommand{\N}{\mathbb{N}}

\newcommand{\R}{\mathbb{R}}

\newcommand{\Sp}{\mathbb{S}}

\newcommand{\ep}{\varepsilon}
\newcommand{\conv}[1]{\xrightarrow{\,#1\,}}
\newcommand{\wconv}[1]{\xrightharpoonup{\,#1\,}} 
\newcommand{\vol}{\text{Vol}}

\newcommand{\sing}{\textnormal{sing}}

\newcommand{\dive}{\text{div}}
\newcommand{\s}{\hspace{7pt}}
\newcommand{\vsp}{\vspace{4.5pt}}
\newcommand{\ov}[1]{\widetilde{#1}}
\newcommand{\lp}{\langle}
\newcommand{\rp}{\rangle}

\newcommand{\p}{\mathfrak p}
\newcommand{\inj}{\textnormal{inj}}

\newtheorem{theorem}{Theorem}[section]
\newtheorem{proposition}[theorem]{Proposition}
\newtheorem{lemma}[theorem]{Lemma}
\newtheorem{corollary}[theorem]{Corollary}

\theoremstyle{definition}

\newtheorem{definition}[theorem]{Definition}

\newtheorem{remark}[theorem]{Remark}

\usepackage{url}
\usepackage{authblk}

\makeatletter
\newcommand{\addressa}[1]{\gdef\@addressa{#1}}
\newcommand{\emaila}[1]{\gdef\@emaila{\url{#1}}}

\newcommand{\addressb}[1]{\gdef\@addressb{#1}}
\newcommand{\emailb}[1]{\gdef\@emailb{\url{#1}}}

\newcommand{\addressc}[1]{\gdef\@addressc{#1}}
\newcommand{\emailc}[1]{\gdef\@emailc{\url{#1}}}

\newcommand{\@endstuff}{\par\vspace{\baselineskip}\noindent
\begin{tabular}{@{}l}\scshape\@addressa\\\textit{E-mail address:} \@emaila\end{tabular} 

\vspace{12pt} \noindent
\begin{tabular}{@{}l}\scshape\@addressb\\ \textit{E-mail address:} \@emailb\end{tabular}

\vspace{12pt} \noindent
\begin{tabular}{@{}l}\scshape\@addressc\\ \textit{E-mail address:} \@emailc\end{tabular}
}

\AtEndDocument{\@endstuff}
\makeatother

\begin{document}

\title{Yau's conjecture for nonlocal minimal surfaces}
\date{\today}


\author{Michele Caselli, Enric Florit-Simon and Joaquim Serra}


\addressa{Michele Caselli \\ Scuola Normale Superiore\\ Piazza dei Cavalieri 7, 56126 Pisa, Italy }
\emaila{michele.caselli@sns.it}

\addressb{Enric Florit-Simon \\ Department of Mathematics, ETH Z\"{u}rich \\ Rämistrasse 101, 8092 Zürich, Switzerland}
\emailb{enric.florit@math.ethz.ch}

\addressc{Joaquim Serra \\ Department of Mathematics, ETH Z\"{u}rich \\ Rämistrasse 101, 8092 Zürich, Switzerland }
\emailc{joaquim.serra@math.ethz.ch}


\setlength\parindent{12pt}

\maketitle

\begin{abstract}
We introduce nonlocal minimal surfaces on closed manifolds and establish a far-reaching Yau-type result: in every closed, $n$-dimensional Riemannian manifold (without any genericity assumption on the metric), we construct infinitely many nonlocal $s$-minimal surfaces. We prove that when $s\in (0,1)$ is sufficiently close to $1$, the constructed surfaces are smooth for $n=3$ and $n=4$, while for $n\ge 5$, they are smooth outside of a closed set of dimension $n-5$. 

 \vsp
Moreover, we prove surprisingly strong regularity and rigidity properties of finite Morse index $s$-minimal surfaces, such as a ``finite Morse index Bernstein-type result'' and the compactness of the class of finite index $s$-minimal surfaces in the strongest geometric sense (that is, they are shown to subsequentially converge smoothly and with multiplicity one).

 \vsp
These properties make nonlocal minimal surfaces ideal objects on which to  apply min-max variational methods as well as to approximate classical minimal surfaces. Combined with the recent results in \cite{CDSV}, which proves uniform curvature estimates and optimal sheet separation for stable $s$-minimal surfaces as $s\nearrow 1$ (i.e. as they converge to classical minimal surfaces), this work sets a new powerful method for the study of old and new questions on the existence of classical minimal surfaces.
\end{abstract} 

\tableofcontents

\section{Introduction}

\subsection{The ``classical'' Yau conjecture}

The existence and regularity of minimal hypersurfaces in closed manifolds is one of the central questions in Riemannian geometry.
Yau's conjecture (raised in 1982 by S.-T. Yau~\cite{Yau82}) is a particularly famous and archetypal problem.  It states that every closed three-dimensional manifold must contain infinitely many smooth minimal surfaces. This problem exposes the enormous difficulties in applying variational methods to the area functional defined on the class of ``surfaces''.

\vspace{2pt}
Yau's conjecture was recently established  by K. Irie,  F. C. Marques, and A. Neves \cite{IMN} (in the case of generic metrics) and by A. Song \cite{Song} (in full generality):

\begin{theorem}[\cite{IMN,Song}] \label{Yaucon}
    Let $(M^n,g)$ be a closed Riemannian manifold of dimension $3\le n\le 7$. Then, there exists an infinite number of smooth, closed, minimal hypersurfaces in $M$. 
\end{theorem}

To construct an infinite number of minimal surfaces, one must consider non-stable critical points of the area functional (since in general closed 3-manifolds contain only a finite number of stable minimal surfaces). These surfaces are naturally constructed using min-max (i.e., mountain-pass type) methods. 
The use of min-max methods for the area functional goes back to Almgren \cite{Alm2, AlmNotes} and afterward, Pitts \cite{Pitts} in the 1960s.

\vspace{2pt}
Several essential difficulties arise when trying to construct minimal surfaces employing a min-max scheme. The principal underlying issue is that, in the case of the area functional for surfaces,  
min-max sequences can be extremely noncompact, unless we work with very weak notions of convergence. 

\vspace{2pt}
Indeed, suppose we are given a sequence of minimal surfaces with uniformly bounded Morse index.
This is only a best case scenario, as the accumulation points of min-max sequences, if they exist, will be finite index minimal surfaces.
A concrete instance of this best case scenario would be a sequence of shrinking catenoids that converge to a hyperplane (with multiplicity two). In this example, the sequence does converge in some weak sense (namely, in the sense of varifolds). However, the limiting object, a ``double hyperplane'', has arguably very few things in common with the catenoids that approximate it (e.g.  their topologies or their total curvatures are completely different!).
More generally, it is possible to guarantee that every sequence of minimal surfaces with uniformly bounded index and area will have convergent subsequences, provided one chooses a weak enough notion of convergence. However, as in the example of the catenoid, weak convergence has undesired side effects: different sequences of minimal surfaces with ``interesting topologies'' may yield, in the limit, the same object with a loss of topology (with an integer multiplicity).

\vspace{2pt}
We emphasize that, as the example of the catenoid shows, there is no way to avoid the side effects: they must come with any convergence weak enough to guarantee compactness. Hence, a crucial difficulty in the proof of Theorem \ref{Yaucon} is the extraordinary difficulty in ensuring that multiple minimal surfaces, with bigger and bigger areas, constructed via a min-max method are distinct (and not the same one counted multiple times). Another very delicate question is the control of the topology of the minimal surfaces constructed via min-max (e.g. to construct minimal spheres in a given manifold). Such issues have only been solved, in particular cases, through a huge effort in several outstanding works, including Almgren \cite{Alm, Alm2, AlmNotes}, Pitts \cite{Pitts}, Schoen-Simon \cite{SchoenSimon}, Marques-Neves \cite{MN, MN2, MN3, Willmore}, Irie-Marques-Neves \cite{IMN} and Song \cite{Song}; or Simon-Smith \cite{FSmith}, Haslhofer-Ketover \cite{minspheres} and Wang-Zhou \cite{WangZhou}, to cite a few.

\subsection{Purpose of the paper}
This work introduces nonlocal minimal (hyper)surfaces ---in the spirit of Caffarelli-Roquejoffre-Savin \cite{CRS}--- on closed Riemannian manifolds and develops their existence and regularity theory. A main purpose of this paper is to demonstrate that they are an ideal class of objects on which to apply min-max methods (as they seem to prevent almost every pathology that arises for classical minimal surfaces, such as multiplicity and loss of topology), as well as to approximate classical minimal surfaces.

\vsp
On this second point, let us emphasize that nonlocal minimal surfaces approximate classical minimal surfaces as the fractional parameter $s\in (0,1)$ converges to $1$.
The recent results in \cite{CDSV} show, among other things, uniform curvature estimates and optimal sheet separation (of order $\sqrt{1-s}$) for stable nonlocal $s$-minimal surfaces in a three-dimensional Euclidean setting as $s\nearrow 1$, which implies their multisheeted convergence towards smooth classical minimal surfaces. A forthcoming work \cite{Florit} will prove a Weyl Law for the $s$-minimal surfaces obtained in the present paper, as well as show their convergence (building on \cite{CDSV}), as $s\to 1$ and for $n=3$, to classical minimal surfaces, obtaining a new proof of their density (and in particular of Theorem \ref{Yaucon}) for three dimensional manifolds with a generic metric.\\
Therefore, in combination with the present paper, nonlocal minimal surfaces provide a powerful new method to construct classical minimal surfaces. This method resembles in some ways the Allen-Cahn approximation in \cite{G, GG1, GG2, CM}, but presents several advantages (some of which are discussed in \cite{CDSV}, and some of which will become evident in this work).

\vsp
We obtain surprisingly strong estimates for finite Morse index nonlocal minimal surfaces that do not hold for classical minimal surfaces. 
These estimates confer finite Morse index nonlocal minimal surfaces exceptional compactness and regularity properties, thanks to which we establish far-reaching existence and regularity results, including a nonlocal analog of Theorem \ref{Yaucon}. Let us give a quick selection/highlights of our results here:

\begin{itemize}
    \item[(i)] Any closed manifold of dimension $n\ge 3$ contains infinitely many nonlocal minimal (hyper)surfaces (i.e., the nonlocal analog of Yau's conjecture holds). More precisely, given $s\in (0,1)$, for every $\p\in \N$ there exists an $s$-minimal surface with Morse index $\le \p$ and fractional perimeter comparable to $\p^{s/n}$. These surfaces are smooth in low dimensions and smooth away from a closed lower-dimensional set in every dimension. We stress that our result holds for \textit{every} metric and not just generic ones.
    \item[(ii)]  For $n\in \{3,4\}$ and $s\in (0,1)$ sufficiently close to $1$ (the limit case $s=1$ formally corresponds to classical minimal surfaces), the following holds:
\begin{itemize}
\item  Any smooth (embedded) $s$-minimal hypersurface of finite Morse index in $\R^n$ must be a hyperplane.
\item In a closed $n$-dimensional manifold $M^n$, any sequence of smooth $s$-minimal (hyper)surfaces  with uniformly bounded Morse index automatically satisfies uniform curvature and sheet separation estimates. As a consequence, any such sequence has a subsequence that converges smoothly and with multiplicity one to a (smooth) submanifold. In particular, if all the elements of the sequence are homeomorphic to the same topological space $\mathbb X$ then the limit is also homeomorphic to $\mathbb X$.
\end{itemize}
\end{itemize}

Thanks to their exceptional compactness and regularity properties, Morse theory for nonlocal minimal surfaces is in some sense as ``flawless'' as finite-dimensional Morse theory, at least from the functional analysis (i.e. compactness) perspective. It goes without saying that this is in striking contrast with the situation for classical minimal surfaces for the area functional.

\vsp
Since we think that the paper can be of interest to readers who do not necessarily have any previous knowledge on nonlocal elliptic equations, we try to give an accessible and mostly self-contained presentation. For the reader's convenience, the present paper is complemented by the companion article \cite{FracSobPaper}, which provides greater detail about the definitions introduced in Section \ref{DefsSection} and includes the proofs of the estimates for the kernel $K_s(p,q)$ (here stated in Section \ref{kernsection}) and the monotonicity formula given in Section \ref{GenMonFormula}.
Moreover, we have spared no efforts in trying to make our proofs as efficient as possible.

\subsection{Nonlocal minimal (hyper)surfaces on a closed Riemannian manifold} 

Nonlocal minimal (hyper)surfaces in $\R^n$ were first introduced and studied in \cite{CRS}. In this section, we define  nonlocal minimal (hyper)surfaces on a closed Riemannian manifold, emphasizing the ``canonical nature'' of these new geometric objects. 

\vsp
Let $(M^n, g)$ be an $n$-dimensional, closed Riemannian manifold, with $n\ge2$.
Let us start by giving a canonical definition of the fractional Sobolev seminorm $H^{s/2}(M)$. This can be done in at least three equivalent ways:
 \begin{itemize}
 \item[(i)] Using the {\em heat kernel}\footnote{As customary, by heat kernel here we mean
the fundamental solution of the heat equation $\partial_t u = \Delta u$ on $M$, where $\Delta$ denotes the Laplace-Beltrami operator on $M$.} $H_M(t,p,q)$ of  $M$, we can put
 \begin{equation*}
K_s(p,q) :=  
 \int_{0}^\infty  H_M(p,q,t)\,\frac{dt}{t^{1+s/2}}.
 \end{equation*}
We then define
 \begin{equation}\label{wethiowhoihw}
 [u]^2_{H^{s/2}(M)} := \iint_{M\times M}(u(p)-u(q))^2 K_s(p,q) \,dV_p \,d V_q.
 \end{equation}
The kernel $K_s(p,q)$ will be shown to be comparable to $\frac{1}{d(p,q)^{n+s}}$, and they coincide in the case $M=\R^n$ (up to a constant factor).
\item[(ii)]  Following a  {\em spectral approach}, we can set
\begin{equation}\label{ghfghfg}
    [u]^2_{H^{s/2}(M)} = \sum_{k\ge 1} \lambda_k^{s/2} \langle u,\varphi_k \rangle^2_{L^2(M)}
\end{equation}
 where $\{\varphi_k\}_k$ is an orthonormal basis of eigenfunctions of the Laplace-Beltrami operator $(-\Delta_g)$ and $\{\lambda_k\}_k$ are the corresponding eigenvalues. For $s=2$ this gives the usual $[u]^2_{H^{1}(M)}$ seminorm.

\item[(iii)] Considering a {\em Caffarelli-Silvestre type extension (cf. \cite{CafSi,BGS})}, namely, a degenerate-harmonic extension problem in one extra dimension, we can set
$$ [u]^2_{H^{s/2}(M)} = \inf\left\{ \int_{M\times \R_+} z^{1-s} |\ov{\nabla}  U(p,z)|^2 \, dV_p dz \s {\mbox{s.t.}} \s U(x, 0)=u(x) \right\}.$$
\end{itemize}
Here $\widetilde \nabla$ denotes the Riemannian gradient of the manifold $\widetilde M = M\times \R_+$, with respect to the natural product metric $\widetilde g= g + dz\otimes dz$, and the infimum is taken over all  $U$ belonging to the weighted Hilbert space $\widetilde{H}^1(\widetilde{M})$ (see Definition \ref{weighSobspacemfd} for the precise definition of this space, and we refer to Section \ref{DefsSection} in general for all the basic properties of this extension characterization).

\vsp
It is proved by the three authors in \cite{FracSobPaper} that (i)-(iii) define the same norm (not merely equivalent norms), up to explicit multiplicative constants. We emphasize that this gives a canonical definition of the $H^{s/2}(M)$ seminorm on a closed manifold. For the reader's convenience, we recall here some definitions and results from \cite{FracSobPaper}. 

\vsp
Here and onwards $M$ denotes a closed $n$-dimensional Riemannian manifold. 
\begin{definition}
    Given $s \in (0, 1)$
and a (measurable) set $E \subset M$, we define the \textit{$s$-perimeter of $E$} as
\begin{equation}\label{wbhtouw1}
\textnormal{Per}_s(E): = [\chi_E]^2_{H^{{s/2}}(M)} = \frac{1}{4}[\chi_E-\chi_{E^c}]^2_{H^{{s/2}}(M)}  = 2\int_{E}\int_{E^c} K_s(p,q)dV_pdV_q\,, 
\end{equation}
where $\chi_E$ is the characteristic function of $E$, $E^c: = M\setminus E$ and $[\cdot]^2_{H^{s/2}(M)}$ is defined by \eqref{wethiowhoihw}. 
\end{definition}

From the estimates in \cite{FracSobPaper} for the kernel $K_s(p,q)$, one can see that for every set $E\subset M$ with smooth boundary, one has that
$(1-s)\textnormal{Per}_s(E) \to \textnormal{Per}(E)$  as $s\uparrow 1$ (up to a multiplicative dimensional constant, see \cite{BBM01}
and also~\cite{Dav02, CV11, ADPM11}
for further details on the computation in the case of $\R^n$). 

\vsp
Moreover, it is convenient to define localized or relative versions of the fractional perimeter, somewhat analogous to the classical relative perimeter.

\begin{definition}
Given a bounded open set $\Omega \subset M$ (with Lipschitz boundary), a \textit{relative $s$-perimeter} in $\Omega$ is a functional denoted
by $\textnormal{Per}_s(\, \cdot \,, \Omega)$ and satisfying the following two properties:
\begin{itemize}
    \item[\rm (I)] $ \textnormal{Per}_s(E , \Omega) - \textnormal{Per}_s(F, \Omega) = \textnormal{Per}_s(E)-\textnormal{Per}_s(F)$ for all (measurable) sets $E$ and $F$ that coincide outside $\Omega$ and $\textnormal{Per}_s(F) < \infty$.
    \item[\rm (II)] $\textnormal{Per}_s(E, \Omega) <\infty $ if $\partial E$ is a smooth submanifold in a neighborhood of the compact set $\overline{\Omega}$.
\end{itemize}
\end{definition}

Throughout the paper, we fix a relative $s$-perimeter defined similarly as in \cite{CRS} (there for the case of the Euclidean space $\R^n$). We define the \textit{relative $s$-perimeter of $E$ in $\Omega$} as
\begin{equation*}
    \textnormal{Per}_s(E, \Omega) :=  \iint_{(M \times M) \setminus (\Omega^c \times \Omega^c)} (\chi_E(p)-\chi_E(q))^2K_s(p,q) dV_p dV_q \,,
\end{equation*}
where $\Omega^c := M \setminus \Omega$ is the complement of $\Omega$. With this definition, one can easily check that properties \rm{(I)} and \rm{(II)} above hold. Moreover, it follows directly from its definition that the previous notion of relative $s$-perimeter satisfies the following properties. 
\begin{itemize}
    \item
    $\textnormal{Per}_s(E, \Omega) = \textnormal{Per}_s(E^c, \Omega) $ for every (measurable) $E\subset M$.
    \item
    If $E\subset \Omega$ or $E^c \subset \Omega$ then $\textnormal{Per}_s(E, \Omega) = \textnormal{Per}_s(E)$, where $\textnormal{Per}_s(E)$ is the $s$ -perimeter on the entire manifold $M$ defined in \eqref{wbhtouw1}.
    \item
    Let $\Omega_1, \Omega_2 \subset M$ with $|\Omega_1 \cap \Omega_2|=0$. Then $\textnormal{Per}_s(E, \Omega_1 \cup \Omega_2) \ge \textnormal{Per}_s(E, \Omega_1)+\textnormal{Per}_s(E, \Omega_2)$.
    \item
    Let $E_1, E_2 \subset M$ with $|E_1 \cap E_2|=0$. Then $ \textnormal{Per}_s(E_1 \cup E_2, \Omega) \le \textnormal{Per}_s(E_1, \Omega)+\textnormal{Per}_s(E_2, \Omega)$.
\end{itemize} 

\begin{remark}
Notice that there would be other possibilities to define a relative $s$-perimeter. For example, in view of the ``spectral definition" \eqref{ghfghfg} of the fractional perimeter, we could have defined a different relative perimeter by $\textnormal{Per}_s(E, \Omega) = \sum_{k\ge 1} \lambda_k^{s/2} \langle \chi_E ,\varphi_k \rangle^2_{L^2(\Omega)}$. It is easy to check that this satisfies properties \rm{(I)}-\rm{(II)} above as well.
\end{remark}

In this work we denote by $\mathfrak X(\mathcal{U})$ the space of smooth vector fields in $\mathcal{U} \subseteq M$, by ${\rm spt}(X)$ the support of $X$ and $\mathfrak X_c(\mathcal{U})$ the space of smooth vector fields with compact support in $\mathcal{U}$. 

\begin{definition}\label{def-fracminsurface}
Let $(M,g)$ be a closed Riemannian manifold. Given $s\in (0,1)$, the boundary $\partial E$ of a set $E\subset M$ is said to be an $s$-\textit{minimal surface} if  $\textnormal{Per}_s(E)<\infty$  and, for every $X\in \mathfrak X(M)$, we have 
\begin{equation*}
     \frac{d}{dt}\Big|_{t=0} \textnormal{Per}_s(\psi^t_X(E)) = 0 \, ,
\end{equation*}
where $\psi^t_X : M \times \R \to M$ denotes the flow of $X$ at time $t$.
\end{definition}

The previous definition admits a natural local version.
\begin{definition}\label{def-fracminsurface2}
Let $(M,g)$ be a closed Riemannian manifold. Given $\mathcal U\subset M$ open, the boundary $\partial E$ of a set $E\subset M$ is said to be an $s$-\textit{minimal surface} in $\mathcal U$ if for every Lipschitz domain $\Omega$ with compact closure such that $\overline \Omega\subset \mathcal U$ we have $\textnormal{Per}_s(E,\Omega)<\infty$, and for every smooth and compactly supported vector field $X\in \mathfrak X_c(\mathcal U)$ with ${\rm spt}(X) \subset \overline\Omega$  we have
\begin{equation*}
     \frac{d}{dt}\Big|_{t=0} \textnormal{Per}_s(\psi^t_X(E), \Omega) = 0 \, .
\end{equation*}
\end{definition}

\begin{definition}[Morse index and stability]\label{WeakMorseDef} Let $(M,g)$ be a closed Riemannian manifold and $\partial E$ be an $s$-minimal surface in $\mathcal U\subset M$ open (as in Definition \ref{def-fracminsurface2}). Then,  $\partial E$ is said to have {\em Morse index} at most $m$ in $\mathcal U$ if for every 
 Lipschitz domain $\Omega$ with compact closure such that $\overline \Omega\subset \mathcal U$, for every $(m+1)$ vector fields $X_0,\dotsc ,X_m \in \mathfrak X_c(\mathcal U)$ with $\cup_{i=0}^{m}{\rm spt}(X_i) \subset \overline\Omega\subset \mathcal U$ there exists some linear combination $X=a_0X_0+\dotsc+a_mX_m$ with $a_0^2 + a_1^2+\dotsc+a_m^2=1$ such that 
\begin{equation*}
    \frac{d^2}{dt^2}\Big|_{t=0}  \textnormal{Per}_s(\psi^t_X(E), \Omega) \ge 0.
\end{equation*}
In the particular case $m=0$, we say that $\partial E$ is \textit{stable} in $\mathcal U$.
\end{definition}
\begin{remark}
    This should be read as: there are at most $m$ independent ``directions'' which decrease the fractional perimeter. This general formulation of Morse index is well suited for passage to the limit.
\end{remark}

\begin{remark}
It follows\footnote{Notice that for functions taking values in $\{\pm 1\}$ the potential part of the energy vanishes and the Sobolev part of the energy gives the fractional perimeter.} from Lemma \ref{enboundslemma} and Lemma \ref{enboundslemma2} (which are results from \cite{FracSobPaper}) that, if $\textnormal{Per}_s(E,\Omega)<\infty$ and $X\in \mathfrak X_c(\mathcal U)$ is such that  ${\rm spt}(X) \subset \overline\Omega$ then the map $t\mapsto \textnormal{Per}_s(\psi^t_X(E), \Omega)$ is well-defined for all $t$ and of class $C^{\infty}$. Thus, the previous definitions are meaningful. 
\end{remark}

\subsection{Main results} 

One of the main goals of this paper is to establish the existence of infinitely many $s$-minimal surfaces on every closed manifold:

\begin{theorem}[\textbf{Fractional Yau-type result}]\label{FracYau3}
    Let $(M^n, g)$ be an $n$-dimensional, closed Riemannian manifold, with $n\ge2$. Fix $s_0 \in (0,1)$ and let $s\in (s_0, 1)$. Then, for every natural number $\p \ge 1$, there exists an $s$-minimal surface $\Sigma^\p = \partial E^\p$ with Morse index at most $\p$ ---in the sense of Definition \ref{WeakMorseDef}--- and fractional perimeter
    \begin{equation*}
    C^{-1} \p^{s/n}  \leq (1-s)\textnormal{Per}_s(E^\p) \leq C \p^{s/n} ,
\end{equation*}
for some $C=C(M, s_0)>1$. In particular, $M$ contains infinitely many $s$-minimal surfaces. Moreover, these surfaces are viscosity solutions to the NMS (i.e. Nonlocal Minimal Surface) equation (see Proposition \ref{prop:viscosity}), and satisfy the structural properties \eqref{cdcd}-\eqref{ababab} in Proposition \ref{StrongConv}.
\end{theorem}

The regularity of the constructed surfaces depends on the classification of stable $s$-minimal cones (An open subset $E \subset \R^n $ is said to be a cone if $E$ is an open set and $\lambda E = E$ for all $\lambda>0$).

\begin{definition}\label{critdim}
    Given $s\in(0,1)$, we define the critical dimension $n_s^*$ as the minimum dimension $n\geq 3$ such that there exists a smooth and stable $s$-minimal cone in $\R^n\setminus\{0\}$ which is not a hyperplane.
\end{definition}
By \cite{CCS} and \cite{CDSV}, $n_s^*\geq 5$ for all $s\in (s_0,1]$, where $s_0\in(0,1)$ is a universal constant. It is conjectured that, in fact, $n_s^*=8$ for all $s$ sufficiently close to $1$. For $n=8$, the Simons cone
$E =\{x_1^2 + x_2^2 + x_3^2 + x_4^2 < x_5^2 + x_6^2 + x_7^2 + x_8^2\} \subset \R^8$, which is a minimizer in the classical case $s=1$, is easily shown to be stable for all $s\in (s_0,1)$, for some $s_0<1$ sufficiently close to 1, so that $n_s^*\leq 8$ in this case\footnote{More generally, the natural generalization of the Simons cone to higher (even) $n$ has been shown in \cite{FelSanz} to be stable in dimension $n\geq 14$ for any $s\in(0,1)$. In particular, $n_s^*\leq 14$ for any $s\in(0,1)$, and the definition of $n_s^*$ as a minimum is justified.}.

\vsp
We now state a regularity result for the constructed surfaces, which will be proved in Section \ref{dimredsection}.
\begin{theorem}[\textbf{Size of the singular set}]\label{FracYau2}
    For $n\geq 3$, the surfaces $\{ \Sigma^{\p} \}_{\p \in \N}$ of Theorem \ref{FracYau3} are smooth submanifolds outside of a closed set $\sing(\Sigma^\p)$ of Hausdorff dimension at most $n-n_s^*$. In particular, $\sing(\Sigma^\p)=\varnothing $ if $n<n_s^*$ (and this holds for $n=3,4$ and $s$ close to $1$, since $n_s^* \ge 5$). Moreover, in the case $n=n_s^*$ the set $\sing(\Sigma^\p)$ is discrete.
\end{theorem}

The surfaces in Theorem \ref{FracYau3} will be constructed as limits as $\ep\to 0^+$ of solutions to the fractional Allen-Cahn equation on $M$. We emphasize that ---in sharp contrast to the case of classical minimal surfaces--- the Allen-Cahn approximation does not really play a crucial role in our construction. We just use it so that we are able to apply standard min-max existence results of critical points (like those in the book by Ghossoub \cite{Gho1}). What really makes our construction easier, in comparison with the classical case $s=1$, are the very strong a priori estimates satisfied by finite Morse index $s$-minimal surfaces for $s<1$, see Remark \ref{EstSurf}. Corresponding analog estimates are satisfied by Allen-Cahn solutions with bounded index, which allows us to send $\ep\to 0$ without problems. In contrast, in the classical case, this passage to the limit is really delicate: one is forced to use varifold convergence, and then multiplicity and neck-pinching situations need to be ruled out. This requires generic metric assumptions and has only been done for $n=3$ in \cite{CM}.

\begin{definition}[Fractional Allen-Cahn energy]
    Let $s\in (0,2)$ and $\ep>0$. Given $v:M\to\R$, we define the \textit{fractional Allen-Cahn (abbr. A-C) energy} of $v$ on the  open set $\Omega\subseteq M$ as
\begin{equation}\label{ACenergy}
 \mathcal E_{\Omega} (v) :=  \mathcal E_{\Omega}^{\rm Sob} (v)   + \mathcal E_{\Omega}^{\rm Pot} (v) ,
 \end{equation}
 where
 \[
 \mathcal E_{\Omega}^{\rm Sob} (v) := \frac 1 4 \iint_{M\times M\setminus \Omega^c\times \Omega^c} (v(p)-v(q))^2 K_s(p,q)dV_p dV_q,   \quad \mathcal E_{\Omega}^{\rm Pot} (v) :=\ep^{-s}\int_{\Omega} W(v)\,dx\, ,
  \]
and $W(v)=\frac{1}{4}(1-v^2)^2$ is the standard quartic double-well potential with wells at $\pm 1$.  We will sometimes denote $\mathcal E_{\Omega}$ by $\mathcal E_{\Omega}^{\ep,s}$ or $\mathcal E_{\Omega}^{\ep}$ if we want to stress the dependence of the energy from $\ep$ and/or $s$. 
\end{definition}
Note that, with this definition of the Allen-Cahn energy, we have
\begin{equation*}
    \mathcal E_{\Omega} (\chi_E-\chi_{E^c}) = \mathcal E_{\Omega}^{\rm Sob} (\chi_E-\chi_{E^c}) = \textnormal{Per}_s(E, \Omega) \,,
\end{equation*}
and 
\begin{equation*}
    \mathcal E_{\Omega_1 \cup \, \Omega_2} (v) \le \mathcal E_{\Omega_1} (v) + \mathcal E_{\Omega_2} (v) \,.
\end{equation*}

The double-well potential penalizes functions that are not identical to $\pm 1$, and that is why one expects to find nonlocal $s$-minimal surfaces as the limits of critical points of this energy when $\ep\to 0$. 

\vsp
A function $u:M\to \R $ is a critical point of $\mathcal E_{\Omega}$ if and only if it solves the fractional Allen-Cahn equation
\begin{equation}\label{restrictedeq}
 (-\Delta)^{s/2}u + \ep^{-s}W'(u)=0  \quad  \mbox{in } \Omega \, .
\end{equation}
Here $(-\Delta)^{{s/2}}$ is the fractional Laplacian on $(M,g)$, and it can be represented as (see Section \ref{DefsSection} for details)
\begin{equation*}
    (-\Delta)^{{s/2}} u(p) = \int_{M} (u(p)-u(q))K_s(p,q) \, dV_q \, .
\end{equation*}

We also have a definition of Morse index, related to the second variation of the energy.
\begin{proposition}[Second variation]
Let $\Omega\subset M$ be an open set. Let $ u \in H^{s/2}(M) $ be a critical point of $\mathcal E_{\Omega}$. Then, given $\xi\in C_c^1(\Omega)$, the second variation of $\mathcal E_{\Omega}$ at $u$ is given by
\begin{align}\label{2ndvar}
\mathcal E_{\Omega}''(u)[\xi,\xi] =
\frac{1}{4}\iint_{(M\times M) \setminus(\Omega^c\times\Omega^c)} |\xi(p)-\xi(q)|^2 K_s(p,q)\,dV_p\,dV_q + \ep^{-s}\int_{\Omega} W''(u)\xi^2\,dV .
\end{align}
\end{proposition} 

\begin{definition}[Morse index]\label{MorseDef}
Let $\Omega\subset M$ an open set, and let $ u \in H^{s/2}(M) $ be a critical point of $\mathcal E_{\Omega}$. The \textit{Morse index} of $u$ in $\Omega$, denoted by $m_{\Omega}(u)$, is defined as the maximum dimension among all linear subspaces $\mathcal L \subset C^1_c(\Omega) \subset H^{s/2}(M)$ such that $\mathcal{E}_{\Omega}''(u) $ is negative definite on $\mathcal L$. Moreover, we say that $u$ is \textit{stable} in $\Omega$ if $m_{\Omega}(u)=0$. 
\end{definition}

For $3\le n<n_s^*$, we prove a strong regularity and separation result for $s$-minimal surfaces which are limits of Allen-Cahn solutions with bounded index (as in our case), and which will be proved in Section \ref{UnifRegSection}.
\begin{definition}[Family of Allen-Cahn limits]\label{AClimits}
    A surface $ \Sigma \subset M$ is said to belong to the class $\mathcal{A}_{m}(M)$ if $\Sigma=\partial E$ and there exists a sequence of functions $u_j:M\to(-1,1)$ which are solutions to the Allen-Cahn equation (\ref{restrictedeq}) on $M$, with Morse index $ m(u_j) \leq m$ for all $j$, and parameters $\ep_j\to 0$, such that $u_j\to u_0:
    =\chi_E-\chi_{E^c}$ in $L^1(M)$. 
\end{definition}

\begin{theorem}[\textbf{Uniform regularity and separation}]\label{UnifReg} 
Let $s\in (0,1)$ and $3\leq n < n_s^*$. Let $(M^n, g)$ be an $n$-dimensional, closed Riemannian manifold satisfying the flatness assumption ${\rm FA}_3(M, g, p, 1, \varphi)$ around $p$ (see Definition \ref{flatnessassup}). Assume that $\partial E\in\mathcal A_m(M)$ is an Allen-Cahn limit (see Definition \ref{AClimits}).
\\
Then $\partial E$ is a $C^{1,\alpha}$ hypersurface for some $\alpha\in(0,1)$, with uniform regularity and separation estimates around $p$. That is, there exists a radius $R=R(n,s,m)>0$ such that, after a rotation, $\varphi^{-1}(\partial E)\cap \big( \B_{R}^{n-1}(0)\times[-R,R]\big)$ is the graph of a \textbf{single} function $f: \B_{R}^{n-1}(0)\times\{0\}\to[-R,R]$ inside the chart, and 
\begin{equation*}
    \|f\|_{C^{1,\alpha}( B_{R}^{n-1}\times\{0\})}\leq C(n,s,m) \,.
\end{equation*}
\end{theorem}

As an immediate application of Theorem \eqref{UnifReg} and Arzel\`a-Ascoli we obtain: 
\begin{corollary}
    Let $(M^n, g)$ be a closed Riemannian manifold, and let $s\in (0,1)$ and $3\leq n < n_s^*$. Then, every sequence $\Sigma_k  = \partial E_k \in \mathcal{A}_m(M)$ admits a subsequence converging to some $\Sigma_\infty$ in the strongest possible sense of convergence for submanifolds. In particular, if all elements $\Sigma_k$ of the sequence are homeomorphic to the same topological space $\mathbb X$, then the limit $\Sigma_\infty$ is also homeomorphic to $\mathbb X$.
\end{corollary}

We now make an important remark.
\begin{remark}\label{EstSurf}
    Define the class $\mathcal{A}_m'(M)$ consisting of surfaces $ \Sigma=\partial E \subset M$ such that there exists a sequence of $s$-minimal surfaces $\Sigma_j = \partial E_j$ of class $C^2$, with Morse index at most $m$ for all $j$, such that $E_j\to E$ in $L^1(M)$. Then, the result of Theorem \ref{UnifReg} would also hold for surfaces in $\mathcal{A}_m'(M)$ (and in particular for surfaces which are a priori known to be $C^2$), with a similar proof but with several technical modifications.
\end{remark}

We conclude this section with a technical remark about dimension $n=2$, which is excluded from our statements.

\begin{remark}
    In the case $n=2$, we cannot expect Theorem \ref{UnifReg} to hold in general anymore, since Lemma \ref{2Dcone} (which is where the assumption $n\geq 3$ is used) cannot be in general replicated. Indeed, for $s$ close to $1$, we expect a cross $\{xy>0\}\subset \R^2$, which is always an $s$-minimal surface, to be a limit of index one fractional Allen--Cahn solutions, just like in the classical case $s=1$. Nevertheless, the proof of Lemma \ref{2Dcone} shows that the only $s$-minimal cone in $\R^2$ which is a limit of stable (or almost-stable) Allen--Cahn solutions is a straight line: the finite index and $n\geq 3$ assumptions are used just to reduce to the almost-stable (see Definition \ref{almoststab}) case via a translation. Our analysis would easily show then that, for $n=2$ and any $s\in(0,1)$, the elements in $\mathcal A_m(M)$ are smooth curves outside of possibly up to $m$ points (where the Morse index of the approximating Allen--Cahn solutions concentrates).
    
    The former generally matches what happens in the local case. There is, however, a striking improvement for $s$ close to $0$: the recent article \cite{CasCones} shows that the only stable (outside of the origin) $s$-minimal cones in $\R^2$ are straight lines for this range of $s$. Therefore, \ref{2Dcone} (which is where $n\geq 3$ is used), and thus Theorem \ref{UnifReg}, hold for $n=2$ and $s$ close to $0$ as well. See \cite[Section 1.1]{CasCones} for more details.
\end{remark}

\subsection{Other highlighted results and overview of the paper}

\subsubsection{Existence of min-max solutions to Allen-Cahn} \label{ExistenceIntro}
In Section \ref{ExistenceSection}, we exhibit in a simple manner the existence of critical points of the Allen-Cahn energy \eqref{ACenergy} on $M$, employing a min-max theorem as in \cite{GG1}. Then, we prove lower and upper bounds for the energies of the constructed solutions. The complete statement of our result is the following.
\begin{theorem}[Existence of min-max Allen-Cahn solutions]
\label{Existence1}
Let $(M^n, g)$ be an $n$-dimensional, closed Riemannian manifold, and fix $s_0\in(0,1)$. Let $ \p \ge 1 $ be a natural number (the number of min-max parameters) and  $s\in(s_0,1)$. Then, there exists $\ep_\p>0$ (depending on $M$, $s$ and $\p$) such that for all $\ep \in (0,\ep_\p)$, there exists a solution $u_{\ep,\p}$ to the Allen-Cahn equation \eqref{restrictedeq} on $M$ with Morse index $m(u_{\ep,\p})\leq \p$. Moreover, there exists $C>1$ depending only on $M$ and $s_0$ such that
\begin{equation}
\label{minmaxboundeq}
    C^{-1}\p^{s/n} \leq (1-s) \,\mathcal{E}_M^{\ep,s} (u_{\ep,\p}) \leq C \p^{s/n}\, .
\end{equation}
\end{theorem}

After proving this result, our main goal will be to show that, for fixed $\p$, as $\ep\to 0$ a subsequence of the $u_{\ep,\p}$ converges in a strong sense to a fractional minimal surface $ \Sigma^\p = \partial E^\p \subset M$, meaning in particular that
\begin{equation*}
\textnormal{Per}_s(E^\p)=\lim_{\ep\to 0}\mathcal{E}_M^{\ep,s}(u_{\ep, \p}) \,.   
\end{equation*}

Together with the bound given by \eqref{minmaxboundeq}, we get for every $\p \in \N $ a fractional minimal surface $\Sigma^\p = \partial E^\p$ with fractional perimeter $ \textnormal{Per}_s(E^\p)\sim \p^{s/n}$. This perimeter growth shows that the family of surfaces $\{\Sigma^\p\}_{\p\in \N}$ necessarily forms an infinite set, thus proving the fractional Yau's conjecture.

\vsp
For this reason, a  large portion of the article is devoted to studying the properties of solutions to the Allen-Cahn equation with a uniform upper bound on their Morse index. We now state and explain the main results of Sections \ref{EstimatesSection} and \ref{ConvSection}.

\subsubsection{Estimates for finite Morse index solutions to Allen-Cahn}
In Section \ref{EstimatesSection}, we prove several estimates for finite Morse index solutions to the Allen-Cahn equation. 

\vsp
In order to quantify the dependence of the constants in the estimates on the geometry of the ambient manifold precisely, the notion of ``local flatness assumption'' will be very useful (this quantification will be important when we perform blow-up arguments). Let us introduce it below.

\vsp 
Here, as in the rest of the paper, $\B_R(0)$ denotes the Euclidean ball of radius $R$ centered at $0$ of  $\R^n$, and $B_R(p) $ denotes the metric ball on $M$ of radius $R$ and center $p$. 

\begin{definition}[Local flatness assumption]\label{flatnessassup}
Let $(M^n,g)$ be an $n$-dimensional Riemannian manifold and $p\in M$. For $R>0$, we say that $ (M,g)$ satisfies the $\ell$-th order \textit{flatness assumption at scale $R$ around the point $p$, with parametrization $\varphi$}, abbreviated as ${\rm FA}_\ell(M,g, R,p,\varphi)$,  whenever there exists an open neighborhood $V$ of $p$ and a diffeomorphism
\begin{equation*}
\varphi:  \B_R(0) \to V, \quad \mbox{with }\varphi(0)=p \,,  
\end{equation*}
such that, letting $g_{ij}=g\left(\varphi_*\left(\frac{\partial}{\partial x^i}\right), \varphi_* \left(\frac{\partial}{\partial x^j} \right) \right)$  
be the representation of the metric $g$ in the coordinates $\varphi^{-1}$, we have
\begin{equation}\label{hsohoh1}
\big( 1-\tfrac{1}{100}\big) |v|^2 \le  g_{ij}(x)v^i v^j \le \big( 1+\tfrac{1}{100}\big)|v|^2 \s\s \forall \, v\in \R^n  \mbox{ and } \forall \, x\in \B_R(0) \,,
\end{equation}
and
\begin{equation*}
R^{|\alpha|}\bigg|\frac{\partial^{|\alpha|} g_{ij} (x)}{\partial x^{\alpha}}\bigg|\le \tfrac{1}{100} \s\s\forall \alpha \mbox{ multi-index  with }1\le|\alpha|\le \ell, \mbox{ and }\forall   x\in \B_R(0).
\end{equation*}
\end{definition}

\begin{remark}\label{fbsvdg}
Notice that for any smooth closed Riemannian manifold $(M,g)$, given $\ell \ge 0$, there exists $R_0>0$ depending on $M$ for which ${\rm FA}_\ell(M,g,R_0,p,\varphi_p)$ is satisfied for all $p\in M$,  where $\varphi_p$ can be chosen to be the restriction of the exponential map\footnote{That is $\varphi_p = (\exp_p\circ \, i)|_{\B_{R_0}(0)}$ for any isometric identification of $i : \R^n \to TM_p$}(of $M$) at $p$ to the (normal) ball $\B_{R_0}(0)\subset T_p M \cong \R^n$.
\end{remark}

\begin{remark}
    The notion of local flatness introduced in Definition~\ref{flatnessassup} is fundamental to our analysis. It guarantees that, once a geodesic ball on $M$ (diffeomorphic to an Euclidean ball, with quantitative control) is fixed, our estimates are uniform independently on the geometry of $M$ outside this ball. Consequently, although the equation we consider is nonlocal---meaning it always depends on the global geometry of $M$---the estimates we derive depend only on the local geometry of $M$.
\end{remark}

\begin{remark}\label{flatscalingrmk}
Throughout the paper, the following scaling properties will be used several times.
\begin{itemize}
\item[(a)] Given $M = (M,g)$ and $r>0$, we can consider the "rescaled manifold" $\widehat M = (M,  r^2g)$. When performing this rescaling, the new heat kernel $H_{\widehat {M}}$ satisfies 
\begin{equation*}
    H_{\widehat{M}}(p,q,t) = r^{-n} H_M(p,q,t/r^2)  \,.
\end{equation*}
As a consequence,  the "rescaled kernel" $\widehat  K_s$ defining the $s$-perimeter on $\widehat M$ satisfies 
\[
\widehat  K_s(p,q) = r^{-(n+s)} K_s(p,q).
\]

\item[(b)] Concerning the flatness assumption, it is  easy to show that ${\rm FA}_\ell(M,g, R, p,\varphi) \Rightarrow {\rm FA}_\ell(M,g, R',p,\varphi)$ for all $R'<R$ and  ${\rm FA}_\ell(M,g, R, p,\varphi)\Leftrightarrow {\rm FA}_\ell(M, r^2g, R/r, p,\varphi(r\,\cdot\,))$.

\item[(c)] Similarly, if ${\rm FA}_\ell(M,g, R, p,\varphi)$ holds,  and $q\in \varphi(\B_R(0))$ is such that $\B_\varrho(\varphi^{-1}(q))\subset \B_R(0)$, then 
${\rm FA}_\ell(M,r^2g, \varrho/r, q,\varphi_{\varphi^{-1}(q),r})$ holds, where $\varphi_{x,\,\rho} := \varphi(x+\rho \, \cdot \,)$. 
\end{itemize}

\end{remark}

One of the main results in the present work is the following estimate, to be proved in Section \ref{BVSection}.

\begin{theorem}[\textbf{BV estimate}]\label{BVEst}
Let $M$ be a closed $n$-dimensional Riemannian manifold for which   ${\rm FA}_2(M,g,R,p,\varphi)$ holds ---see  Definition \ref{flatnessassup}. Let $s\in (0,1)$ and $u: B_R(p)\to (-1,1)$ be a solution of the Allen-Cahn equation (\ref{restrictedeq}) in $B_{R}(p)\subset M$ with parameter $\ep$, and with Morse index $m_{B_{R}(p)}(u)\leq m$. Then 
\begin{equation*}
\int_{B_{R/2}(p)}|\nabla u|dx \leq C R^{n-1},
\end{equation*}
for some $ C= C(n, s ,m )$. 
\end{theorem}
Note that the $BV$ estimate above holds uniformly in the Allen-Cahn parameter $\ep$. The nomenclature of ``$BV$" is to be read as ``\textit{bounded variation}".

\begin{remark}
    Our proof of Theorem \ref{BVEst} gives a control on the behavior of the constant $C(n,s,m)$ as $s\uparrow 1$. More precisely, for fixed $s_\circ \in (0,1)$ we have $C(n,s,m) \le C(n,s_\circ,m)/(1-s)$ for all $s\in (s_\circ,1)$. In view of the results from \cite{CDSV}, the sharp asymptotic for $s$ close to $1$ is expected to be $C(n,s,m) \le C(n,s_\circ,m)/(1-s)^{1/2}$.
\end{remark}

Another important result is a bound on the Sobolev and Potential parts of the energies, obtained in Section \ref{PotDecSection}:
\begin{theorem}[\textbf{Energy estimate}]\label{A-C-energy-est}
Let $u : M\to (-1,1)$ be a solution of (\ref{restrictedeq}) in $B_{R}(p)\subset M$ with parameter $\ep$ and Morse index $m_{B_{R}(p)}(u)\leq m$. Suppose that  ${\rm FA}_2(M,g,R,p,\varphi)$ holds ---see Definition \ref{flatnessassup}. Then
\begin{equation*}
\mathcal E^{\rm Sob}_{B_{R/2}(p)}(u) \le CR^{n-s} ,
\end{equation*}
and there exists $\ep_0 = \ep_0(n,s,m)$ such that for $\ep <\ep_0$
\begin{equation*}
\mathcal E^{\rm Pot}_{B_{R/2}(p)}(u) \le C \Big(\frac{\ep}{R}\Big)^\beta R^{n-s} ,
\end{equation*}
where $C = C(n,s,m)$ and $\beta := \min\big(\frac{1-s}{2}, s \big) >0$.
\end{theorem}

Section \ref{DensEstSection} will prove the following result, and which will give, among other things, that the level sets of Allen-Cahn solutions converge to the limit (hyper)surfaces in the Hausdorff distance of sets.
\begin{proposition}[\textbf{Density estimates}]\label{DensEst}
Let $u:M\to (-1,1)$ be a solution of (\ref{restrictedeq}) in $B_{R}(p)\subset M$ with Morse index $m_{B_{R}(p)}(u)\leq m$, and suppose that ${\rm FA}_2(M,g,R,p,\varphi)$ holds ---see Definition \ref{flatnessassup}. Then, there exist positive constants $\omega_0$, $C_0$ and $\ep_0$, depending only on $n$, $s$, and $m$, such that the following holds: whenever $\ep \le \ep_0$, $R \ge C_0\ep$ and
\begin{equation*}
R^{-n}\int_{B_R(p)} |1+u_\ep|  \leq \omega_0\qquad \left(\mbox{respectively,}\quad R^{-n}\int_{B_R(p)} |1-u_\ep|  \leq \omega_0\right), 
\end{equation*}
then
\begin{equation*}
\left\{ u_\ep\ge -\tfrac{9}{10}\right\}\cap B_{R/2}(p) = \varnothing\qquad  \bigg(\mbox{respectively,}\quad \left\{ u_\ep\le \tfrac{9}{10}\right\}\cap B_{R/2}(p) = \varnothing\bigg).
\end{equation*}
\end{proposition}

\subsubsection{Convergence results}
In Section \ref{ConvSection}, the estimates we have just stated are used to show the convergence, as $\ep\to 0$, of solutions of \eqref{restrictedeq} to a limit interface. 

The complete statement of our convergence result is the following. 
\begin{theorem}
\label{StrongConv}(\textbf{Convergence as $\ep\to 0^+$}). Fix $s \in (0,1)$. Let $u_{\ep_j}$ be a sequence of solutions of (\ref{restrictedeq}) on $M$ with parameters $\ep_j\to 0$ and Morse index $m(u_{\ep_j}) \le m$. Then, there exist a subsequence, still denoted by $u_{\ep_j}$, and an $s$-minimal surface $\Sigma= \partial E$ with Morse index at most $m$, such that
\begin{equation*}
    u_{\ep_j} \conv{H^{s/2}} u_0=\chi_E-\chi_{E^c}\, .
\end{equation*}
In particular $\mathcal{E}_{M}^{\rm Sob}(u_{\ep_j}) \to \textnormal{Per}_s(E) = \mathcal{E}_{M}^{ \rm Sob}(u_0)$. Moreover, $\mathcal{E}_{M}^{\rm Pot}(u_{\ep_j})\to 0=\mathcal{E}_{M}^{\rm Pot}(u_0)$.

\vsp
In addition, up to changing $E$ on a set of measure zero, we have
\begin{eqnarray}
\textnormal{int}(E) & \supseteq & \Big\{ p \in M \, : \liminf_{r\downarrow 0} \tfrac{|E \cap B_r(p) |}{|B_r(p)|} =1 \Big\} , \label{cdcd} \\ \nonumber 
M\setminus \overline{E} &\supseteq & \Big\{  p \in M \, : \limsup_{r\downarrow 0} \tfrac{|E \cap B_r(p) |}{|B_r(p)|} =0 \Big\} , \\ \label{ababab}
\Sigma  &= & \Big\{ p \in M \, : \tfrac{|E \cap B_r(p) |}{|B_r(p)|} \in [\delta, 1-\delta]\quad \forall \,  r\in (0,r_p), \s \textnormal{for some } r_p>0 \Big\} ,
\end{eqnarray}
where $\delta=\delta(n,s,m) \ll 1$ and $\Sigma=\partial E$ represents the topological boundary of $E$. Moreover, for all given $c\in(-1,1)$ 
\begin{equation*}
  d_{\rm H} ( \{u_{\ep_j} \ge   c\} ,  E ) \rightarrow 0 \,, \s \text{as } j \to \infty \,,
\end{equation*}
where $d_{\rm H}(X,Y) = \inf\{ \rho>0 \,: \, X \subseteq \bigcup_{y \in Y} B_\rho(y) \ \mbox{and} \ Y\subseteq \bigcup_{x \in X} B_\rho(x) \}$  denotes the standard Hausdorff distance between subsets of $M$.
\end{theorem}

As explained in Section \ref{ExistenceIntro}, this result combined with Theorem \ref{Existence1} gives Theorem \ref{FracYau3}.

\subsubsection{Regularity in low dimensions}
Sections \ref{BlowUpSection}--\ref{dimredsection} are devoted to proving the uniform regularity and separation estimate in low dimensions of Theorem \ref{UnifReg}, as well as the result of Theorem \ref{FracYau2} on the size of the singular set in higher dimensions.

First, Sections \ref{BlowUpSection} and \ref{ACpropsect} define and describe the properties of blow-ups of $s$-minimal surfaces, in particular when they are the limits of Allen-Cahn solutions with bounded index.

Then, in Section \ref{BlowupLimSection} it is shown that such blow-ups converge to a single hyperplane in $\R^n$, under the assumption that stable $s$-minimal cones in $\R^n$ are flat; that is, when $n<n_s^*$ is less than the critical dimension of Definition \ref{critdim}. This classification result for blow-ups is used in Section \ref{UnifRegSection} to prove Theorem \ref{UnifReg}. The proof is done by a blow-up and contradiction strategy to show that the surfaces are flat at some fixed scale, and an improvement of flatness theorem\footnote{This improvement of flatness theorem was proved on $\R^n$ in the seminal article \cite{CRS} which first defined nonlocal minimal surfaces, and the version of it on manifolds has been recently proved in \cite{Moy}.} which holds for all nonlocal minimal surfaces which are viscosity solutions of the zero nonlocal mean curvature equation, a criticality condition much weaker than minimality.

Finally, a dimension-reduction argument combined with the previous strategy allows to prove Theorem \ref{FracYau2} for all $n$.

\subsubsection{Bernstein and De Giorgi type results}
Section \ref{DGBSection} establishes the validity of the ``finite Morse index versions'' of the nonlocal De Giorgi and Bernstein conjectures, once again under the assumption of the classification of stable cones. This represents a remarkable departure from the behavior of classical minimal surfaces and of solutions to the classical (local) Allen-Cahn equation with a bounded index. 

\vsp
The proof of both results uses the same strategy as the proof, in Section \ref{BlowupLimSection}, of the fact that blow-up limits of $s$-minimal surfaces satisfying a certain list of properties, which are in particular satisfied by limits of Allen-Cahn need to be half-spaces.

\vsp
The Bernstein conjecture (today theorem) states that graphical complete minimal hypersurfaces must be hyperplanes in low dimensions. See \cite{Alm,Bern,CL,dCP,FishS,Pog,Simons} for related generalizations to the classes of minimizing and stable hypersurfaces.

\vsp
In Section \ref{DGBSection} we establish:

\begin{theorem}[Finite index nonlocal Bernstein-type result]\label{BernsteinIntro}
Let $s \in (0,1)$ and $3 \le n<n_s^*$, where $n_s^*$ is the critical dimension (see Definition \ref{critdim}). 

\vsp
Then, any finite Morse index $s$-minimal surface in $\R^n$ of class $C^2$ is a half-space.
\end{theorem}
Under the assumption of stability (Morse index zero) the previous theorem was established in \cite{Stable} and in the case of minimizers if follows from \cite{CRS}.

\vsp
The De Giorgi conjecture is a famous related statement about certain entire solutions to the Allen-Cahn equation being one-dimensional or equivalently about their level sets being hyperplanes in low dimensions. See \cite{AAC,AmbrC,FS,GG,Savin,Stable} for related previous results in the minimizing and stable cases. In Section \ref{DGBSection}, we show:

\begin{theorem}[Finite index nonlocal De Giorgi-type result]\label{DeGiorgi}
Let $s\in (0,1)$ and $3 \le n<n_s^*$, where $n_s^*$ is the critical dimension (see Definition \ref{critdim}).

\vsp
Then, every finite Morse index solution $u$ of $(-\Delta)^{s/2}u + W'(u)=0$ in $\R^n$ is a 1D layer solution, namely, $u(x)= \phi(e\cdot x)$ for some $ e \in \Sp^{n-1}$ and increasing function $\phi: \R \rightarrow (-1,1)$. 
\end{theorem}
Under the assumption of stability (Morse index zero), the previous theorem was established in \cite{Stable}, and for minimizers, it followed from \cite{CRS,dPSV}.

\begin{remark}
    Recall that $n_s^*\geq 5$ for $s\in (s_0,1)$ for some universal constant $s_0 \in (0,1)$. In particular, Theorems \ref{BernsteinIntro} and \ref{DeGiorgi} hold for $n=3,4$ and $s\in (s_0,1)$.
\end{remark}

\section{The fractional setting on manifolds}\label{DefsSection}

Throughout the paper $(M,g)$ will denote a closed (i.e. compact and without boundary) Riemannian manifold of dimension $n$, unless otherwise stated. We refer to \cite{FracSobPaper} for a more detailed introduction to fractional Sobolev spaces on Riemannian manifolds, including the proofs of all of the results in this section.

\subsection{The fractional Laplacian on \texorpdfstring{$(M,g)$}{}}
Taking inspiration from the case of $\R^n$ (see \cite{Stinga} for instance), in this section we give several equivalent definitions for the fractional Laplacian on a closed Riemannian manifold $(M,g)$. 
It is natural here to define the fractional powers $(-\Delta)^{s/2}$ for any $s \in (0,2)$. On the other hand, in the rest of the paper we will always restrict to $s\in (0,1)$: the norm $H^{s/2}(M)$ used to define the fractional perimeter is only considered with $s\in(0,1)$ since the $H^{1/2}$ energy of a characteristic function is infinite.

\subsubsection{Spectral and singular integral definitions}

The fractional Laplacian $(-\Delta)^{s/2}$ can be defined as the $s/2$-th power (in the sense of spectral theory) of the usual Laplace-Beltrami operator on a Riemannian manifold, through Bochner's subordination.

\vsp
Given real numbers $\lambda >0$ and $s \in (0,2)$, the following numerical formula holds
\begin{equation}\label{numeric1}
    \lambda^{s/2} = \frac{1}{\Gamma(-s/2)} \int_{0}^{\infty} (e^{-\lambda t}-1)\frac{dt}{t^{1+s/2}} \, ,
\end{equation}
which can be proved by a simple substitution in the integral on the right-hand side. Formally applying the above relation to the operator $L=(-\Delta)$ in place of $\lambda$, one obtains the following definition for the fractional Laplacian.
\begin{definition}[\textbf{Spectral definition}]
Let $s\in(0,2)$. The fractional Laplacian $(-\Delta)^{s/2}$ is the operator that acts on regular functions $u$ by
\begin{equation}\label{boclap}
   (-\Delta)^{s/2} \, u = \frac{1}{\Gamma(-s/2)} \int_{0}^{\infty} (e^{t\Delta}u-u)\frac{dt}{t^{1+s/2}}\, .
\end{equation}
Here, the expression $e^{t\Delta}u$ is to be understood as the solution of the heat equation on $M$ at time $t$ and with initial datum $u$.
\end{definition}
\begin{remark}
On a closed Riemannian manifold, a closely related definition of the fractional Laplacian is available: if $\{\phi_k\}_{k=1}^{\infty} $ is an $L^2(M)$ orthonormal basis of eigenfunctions for $(-\Delta) $ with eigenvalues
\begin{equation*}
    0 = \lambda_1 < \lambda_2 \le \dotsc \le \lambda_k \conv{k \to \infty} +\infty
\end{equation*}
and $u \in L^2(M)$ then
\begin{equation*}
    (-\Delta)^{s/2} u = \sum_{k=1}^\infty \lambda_k^{s/2} \lp u, \phi_k\rp_{L^2(M)} \phi_k \,. 
\end{equation*}
Since the solution to the heat equation on $M$ with initial datum an eigenfunction $\phi_k$ is given by $e^{t\Delta}\phi_k = e^{-\lambda_k t}\phi_k$, the above definition is easily shown to be identical (for $u$ regular) to \eqref{boclap} by first observing that they coincide for eigenfunctions (thanks to \eqref{numeric1}), and then extending the result by approximation.
\end{remark}

The second definition for the fractional Laplacian, closely related to the spectral one, expresses it as a singular integral. It will be our working definition in a substantial portion of the article.
\begin{definition}[\textbf{Singular integral definition}]
The fractional Laplacian $(-\Delta)^{s/2}$ of order (of differentiation) $s \in (0,2)$ is the operator that acts on a regular function $u$ by
\begin{equation}\label{singintlap}
    (-\Delta)^{s/2} u(p) = \int_{M} (u(p)-u(q))K_s(p,q) \, dV_q \,,
\end{equation}
where $K_s(p,q) : M \times M \to \R$ is given by\footnote{Note that $\frac{1}{|\Gamma(-s/2)|} = \frac{s/2}{\Gamma(1-s/2)}$. }
\begin{equation*}
    K_s(p,q)=\frac{s/2}{\Gamma(1-s/2)} \int_{0}^{\infty} H_M(p,q,t) \frac{dt}{t^{1+s/2}} \,,
\end{equation*}
and where $H_M:M\times M \times (0,\infty)\to \R$ denotes the usual heat kernel on $M$.
\end{definition}
\begin{remark}
If the manifold $M$ is replaced by the Euclidean space $\R^n$, then (we refer also to \cite[Section 9.1]{AlmLieb} for the very same computation) we have
\begin{equation*}
    K_{s}(x,y)=\frac{s/2}{\Gamma(1-s/2)} \int_{0}^{\infty} H_{\R^n}(x,y,t) \frac{dt}{t^{1+s/2}} = \frac{s/2}{\Gamma(1-s/2)} \int_{0}^{\infty} \left(  \frac{1}{(4\pi t)^{\frac{n}{2}}} e^{-\frac{|x-y|^2}{4t}}\right) \frac{dt}{t^{1+s/2}} = \frac{\alpha_{n,s}}{|x-y|^{n+s}} ,
\end{equation*}
where 
\begin{equation*}
    \alpha_{n,s}= \frac{2^s \Gamma\Big(\tfrac{n+s}{2}\Big)}{\pi^{ n/2} |\Gamma(-s/2)| } = \frac{s 2^{s-1} \Gamma\Big(\tfrac{n+s}{2}\Big)}{\pi^{n/2}\Gamma(1-s/2)}.
\end{equation*}
Hence, we recover the usual form of the fractional Laplacian on $\R^n$. See also 
\end{remark}

The equivalence\footnote{With our choice of constants, this is an equality and not only an equivalence.} between the two definitions \eqref{boclap} and \eqref{singintlap} is seen by expressing the solution $e^{t\Delta}u$ to the heat equation in terms of the initial datum $u$ as 
\begin{equation*}
    (e^{t\Delta}u)(p)=\int_M u(q) H_M(p,q,t) \,dV_q \,,
\end{equation*}
using $\int_M H_M(p,q,t) \,dV_q=1$ and changing the order of integration.

\subsubsection{Properties of the kernel}\label{kernsection}

In this section, we recall some important estimates on the singular kernel $K_s(p,q)$ taken from \cite{FracSobPaper}, which are proved by means of analogous local estimates for the heat kernel.

\vsp

In all the sections, we will use the (standard) multi-index notation for derivatives. A multi-index $\alpha=(\alpha_1,\alpha_2,\dots, \alpha_n)$ will be an $n$-tuple of nonnegative integers (in other words $\alpha \in \N^n$). We define 
\[
|\alpha| : = \alpha_1+ \alpha_2 +\cdots +\alpha_n.
\]
For a function  $f: \R^n \to \R$ of class $C^\ell$ we shall use the notation
\[
\frac{\partial^{|\alpha|}}{\partial x^{\alpha}} f : = \frac{\partial^{\alpha_1 + \alpha_2 + \cdots \alpha_n}f}{(\partial x^1)^{\alpha_1} (\partial x^2)^{\alpha_2}\cdots  (\partial x^n)^{\alpha_n} }\, .
\]
For $\alpha=0$, we put $\frac{\partial^{|0|}}{\partial x^{0}} f :=f$.

\begin{proposition}[\cite{FracSobPaper}]\label{prop:kern1}
Let $(M,g)$ be a Riemannian $n$-manifold, not necessarily closed, $s\in(0,2)$ and let $p\in M$. Assume ${\rm FA_\ell}(M,g,R,p, \varphi)$ holds and denote $K(x,y): = K_s(\varphi(x), \varphi(y))$.

\vsp
Given $x\in \B_R(0)$, let
$A(x)$ denote the positive symmetric square root of the matrix $(g_{ij}(x))$ ---$g_{ij}$ being the metric in coordinates $\varphi^{-1}$--- and, for $x,z\in \B_{R/2}(0)$, define
\begin{equation*}
    k(x, z): = K(x,x+z) \quad \mbox{and}\quad  \widehat k(x,z) := k(x,z) - \frac{\alpha_{n,s}}{|A(x)z|^{n+s}}.
\end{equation*}
Then
\begin{equation}\label{remaining0}
\big|\widehat k(x,z)\big| \le  R^{-1}\frac{C(n,s)}{|z|^{n+s -1}}\quad \mbox{for all } x,z \in \B_{R/4}(0) \setminus \{0\} ,
\end{equation}
and, for every multi-indices $\alpha,\beta$  with $|\alpha|+|\beta| \le \ell$, we have
\begin{equation*}
  \bigg|\frac{\partial^{|\alpha|}}{\partial x^{\alpha}} \frac{\partial^{|\beta|}}{\partial z^{\beta}}  k(x,z)\bigg| \le \frac{C(n,s, \ell)}{|z|^{n+s+ |\beta|}}\quad \mbox{for all } x,z \in \B_{R/4}(0) \setminus \{0\} .  
\end{equation*}
The constants $C(n,s)$ and $C(n,s,l)$ stay bounded for $s$ away from $0$ and $2$.\\

Moreover, for all $x\in \B_{R/4}(0)$ and for all $q\in M\setminus \varphi(\B_R(0))$ we have
\begin{equation*}
\bigg|\frac{\partial^{|\alpha|}}{\partial x^{\alpha}}  K_s(\varphi(x),q)\bigg| \le \frac{C(n,\ell)}{R^{n+s}} \,,
\end{equation*}
and 
\begin{equation}\label{remaining3}
\int_{M\setminus \varphi(\B_R(0))}\bigg|\frac{\partial^{|\alpha|}}{\partial x^{\alpha}}  K_s(\varphi(x),q)\bigg| dV_q \le \frac{C(n,\ell)}{R^{s}}, 
\end{equation}
for every multi-index  $\alpha$ with $|\alpha|\le \ell$.

\end{proposition}

\begin{lemma}[\cite{FracSobPaper}]\label{loccomparability}
    Let $s_0 \in (0,2)$ and $s \in (s_0,2)$. Let $(M,g)$ be a Riemannian $n$-manifold and $ p \in M$. Assume that ${\rm FA}_1(M,g,p,1,\varphi)$ holds. Then 
\begin{equation*}
   c_7 \frac{\alpha_{n,s}}{|x-y|^{n+s}} \le  K_s(\varphi(x),\varphi(y)) \le  c_8 \frac{\alpha_{n,s}}{|x-y|^{n+s}} \,, 
\end{equation*}
for all $x,y \in \B_{1/2}(0)$, where $c_7,c_8>0$ depends on $n$ and $s_0$.
\end{lemma}

The next property concerns the behavior of the kernel when the two points $p$ and $q$ are separated from each other.

\begin{proposition}[\cite{FracSobPaper}]
Let $(M,g)$ be a  Riemannian $n$-manifold and $s\in (0,2)$. Assume that for some $p,q\in M$  both ${\rm FA_\ell}(M, g, 1, p, \varphi_p)$  and ${\rm FA_\ell}(M, g, 1, q, \varphi_q)$  hold, and suppose that $\varphi_p(\B_1(0))\cap \varphi_q(\B_1(0))= \varnothing$.  Put $K_{pq}(x,y): = K_s(\varphi_p(x), \varphi_q(y))$. Then
\[
\bigg|\frac{\partial^{|\alpha|}}{\partial x^{\alpha}} \frac{\partial^{|\beta|}}{\partial y^{\beta}}   K_{pq}(x,y)\bigg| \le C(n,\ell)\quad 
\mbox{for all }|x|<\tfrac 1 2 \mbox{ and }|y|< \tfrac 1 2,
\]
whenever $|\alpha|+|\beta| \le \ell$.
\end{proposition}

The next results regards the behavior of the singular kernel $K_s$ when translating its arguments under the flow of a vector field.
\begin{proposition}[\cite{FracSobPaper}]\label{Ker2est} 
Let $(M,g)$ be a closed $n$-dimensional Riemannian manifold and $s\in(0,2)$. Consider any smooth vector field $X \in \mathfrak{X}( M)$, and fix points $p,q\in M$. Writing $\psi^t$ for the flow of $X$ at time $t$, then the kernel satisfies
\begin{equation*}
    \bigg| \frac{d^\ell}{dt^\ell}\bigg|_{t=0} K_s(\psi^t(p),\psi^t(q)) \bigg|\leq C(1 +K_s(p,q)).
\end{equation*}
for some constant $C=C(M,s,\ell, \max_{0\le k\le \ell} \|\nabla^k X\|_{L^\infty(M)} )$. Moreover, given $T>0$ we have that, for all $0\le t \le T$,
\begin{equation}\label{rderest}
    \left|\frac{d^\ell}{d t^\ell} K_s(\psi^{ t}(p),\psi^{ t}(q)) \right|  \leq  C_T (1+K_s(p,q))\,,
\end{equation}\
where $C_T=C_T(M,s,\ell, T,\max_{0\le k\le \ell} \|\nabla^k X\|_{L^\infty(M)} )$. Both $C$ and $C_T$ stay bounded for $s$ away from $0$ and $2$.
\end{proposition}

We also record a version of Proposition \ref{Ker2est} which depends only on local quantities:

    \begin{proposition}[\cite{FracSobPaper}]\label{sdfgsdrgs}
        Let $(M,g)$ be a closed $n$-dimensional Riemannian manifold and $s\in(0,2)$. Assume that the flatness assumption ${\rm FA}_\ell (M,g,R,p,\varphi)$ holds, and let $X\in \mathfrak{X}( M)$ be a smooth vector field supported on $\varphi(\B_{R/4})$. Writing $\psi^t$ for the flow of $X$ at time $t$, then for every $x,y \in \B_{R/4}(0)$ with $x\neq y$ we have
        \begin{equation*}
    \bigg| \frac{d^\ell}{dt^\ell}\bigg|_{t=0} K_s(\psi^t(\varphi(x)),\psi^t(\varphi(y))) \bigg|\leq C K_s(\varphi(x), \varphi(y)) \le  C\frac{\alpha_{n,s}}{|x-y|^{n+s}}\,,
\end{equation*}
for some constant $C=C(n,s, \ell,\|X \|_{C^\ell(\varphi(\B_{R/4}))})$. Moreover, given $T>0$ we have that, for all $0\le t \le T$,
\begin{equation}\label{cdsacsdc2}
    \bigg| \frac{d^\ell}{dt^\ell}  K_s(\psi^t(\varphi(x)),\psi^t(\varphi(y))) \bigg|\leq C_T K_s(\varphi(x), \varphi(y)) \le  C_T\frac{\alpha_{n,s}}{|x-y|^{n+s}}\,,
\end{equation}
where $C_T=C_T(n,s, \ell, T, \max_{0\le k\le \ell} \|\nabla^k X\|_{L^\infty(\varphi(\B_{R/4}))})$.
\end{proposition}

Proposition  \ref{Ker2est} is used to bound time derivatives of the energy of ``flown objects'' by their energy at time zero.
\begin{lemma}[\cite{FracSobPaper}]\label{enboundslemma}
Let $s \in (0,2)$ and $ v\in H^{s/2}(M)$ be a function with $|v|\leq 1$. Let $X \in \mathfrak{X} (M)$ be a smooth vector field and $v_t := v \circ \psi^{-t}$, where $\psi^t$ is the flow of $X$ at time $t$. Then, for all $T>0$ there holds
\begin{equation*}
    \sup_{0<t<T} \bigg| \frac{d^\ell}{d  t^\ell}\mathcal{E}_{M}(v_t) \bigg|\leq C\big(1+ \mathcal{E}_{M}(v)\big)\, ,
\end{equation*}
for some constant $C=C(M,s,\ell,T, \max_{0\le k\le \ell} \|\nabla^k X\|_{L^\infty(M)} )$.
\end{lemma}

Lemma \ref{enboundslemma} has a local version, which comes from applying local estimates for the kernel instead.
\begin{lemma}[\cite{FracSobPaper}]\label{enboundslemma2}
Let $M$ satisfy the flatness assumptions ${\rm FA}_\ell(M,g,R,p,\varphi)$. Let $s \in (0,1)$ and $ v\in H^{s/2}(M)$ be a function with $|v|\leq 1$. Let $X \in \mathfrak{X} (M)$ be a smooth vector field supported on $\varphi(\B_{R/2})$, and put $v_t := v \circ \psi_X^{-t}$, where $\psi_X^t$ is the flow of $X$ at time $t$. Then, for all $T>0$ there holds
\begin{equation*}
    \sup_{0< t<T} \left| \frac{d^\ell}{d  t^\ell}\mathcal{E}_{\varphi(\B_{R/2})}(v_{t}) \right|\leq C(1+\mathcal{E}_{\varphi(\B_{R/2})}(v))\, ,
\end{equation*}
for some constant $C=C(s,\ell,T, \max_{0\le k\le \ell} \|\nabla^k X\|_{L^\infty(\varphi(\B_{R/2}))} )$.
\end{lemma}

\subsubsection{Caffarelli-Silvestre type extension}

\begin{definition}\label{weighSobspacemfd}
    We define the weighted Sobolev space
   \begin{equation*}
       \widetilde{H}^1 (\widetilde M) = \widetilde{H}^1( M \times (0,\infty)) 
   \end{equation*}
   as the completion of $C_c^\infty( M \times [0,\infty))$ with the norm
   \begin{equation}\label{eq: norm def sob space}
       \| U\|^2_{\widetilde H^1} := \| TU \|^2_{L^2(M)} + \| \widetilde \nabla U\|^2_{L^2(\widetilde M, z^{1-s}dVdz)} ,
   \end{equation}
   where $TU = U(\cdot, 0)$ is the trace of $U$ and $\widetilde \nabla U =(\nabla U, U_z)$ denotes the gradient in $\widetilde{M}=M\times (0,+\infty)$ endowed with the natural product metric. This is a Hilbert space with the natural inner product that induces the norm above. Moreover, basically by definition, any $U \in \widetilde{H}^1 (\widetilde M) $ leaves a trace in $L^2(M\times \{0\})$.  
\end{definition}
\begin{remark}
The space $\widetilde{H}^1 (\widetilde M)$ can be concretely realized as a space of functions $U$ in $L^2_{\rm loc}(\widetilde M )$ having weak derivatives $\widetilde \nabla U$ in the same weighted space. Indeed, by the fundamental theorem of calculus and H\"older's inequality (see, for example, the proof of \cite[Lemma 3.3]{FracSobPaper}) we have
\begin{equation*}
    \int_{M\times (0,R)} |U|^2 \, dVdz \le CR \int_M |TU|^2 \, dV + CR^s \int_{M\times (0,R)} |\widetilde \nabla U|^2 z^{1-s} \, dVdz , 
\end{equation*}
for every $U \in C_c^\infty( M \times [0,\infty))$ and $R>0$. This inequality easily implies that every Cauchy sequence (with respect to \eqref{eq: norm def sob space}) of smooth functions converges to an actual function in $L^2_{\rm loc}(\widetilde M )$, and the limit is well-defined pointwise almost everywhere.
\end{remark}

\vsp 

The following essential result from \cite{FracSobPaper}, analogous to the classical one for $\R^n$ in \cite{CafSi}, shows that the fractional power of the Laplacian on $M$ can be realized as a Dirichlet-to-Neumann map via an extension problem.

\begin{theorem}[\cite{FracSobPaper}]\label{extmfd}
Let $(M^n,g)$ be a closed Riemannian manifold, and let $s\in (0,2)$ and $u\in H^{s/2}(M)$. Consider the product manifold $\widetilde{M}=M\times (0,+\infty)$ endowed with the natural product metric. Then, there is a unique solution $U:M \times (0, \infty) \to \R $ among functions in $\widetilde H^1(\widetilde M )$ to\footnote{Here $\widetilde{ {\rm div}}$ denotes the divergence operator on $\widetilde M$. }
\begin{equation}\label{caffextMfd}
    \begin{cases} \widetilde{ {\rm div}}(z^{1-s} \widetilde \nabla U) = 0  &   \mbox{in} \s \widetilde{M} \,,  \\ U(p,0) = u(p) &  \mbox{for} \s p \in \partial \widetilde{M} = M\, . \end{cases}
\end{equation}
Moreover, if $u$ is smooth then
\begin{equation*}
    \lim_{z \to 0^+} z^{1-s} \frac{\partial U}{\partial z}(p,z) =  -\beta_s^{-1} (-\Delta)^{s/2} u(p) \,, 
\end{equation*}
where the fractional Laplacian on the right-hand side is defined by either \eqref{boclap} or \eqref{singintlap} and
\begin{equation}\label{betadef}
    \beta_s = \frac{2^{s-1} \Gamma(s/2)}{\Gamma(1-s/2)}\, .
\end{equation}
\end{theorem}

\subsubsection{Modifications of definitions for the Euclidean space and other noncompact manifolds}

We observe that  \eqref{wbhtouw1} can also be used to define the fractional perimeter for any (possibly noncompact) complete manifold $N$ for which a fractional Hilbert norm $H^{s/2}(N)$ is defined. It will be clear from our proofs that, in the case of noncompact manifolds for which the equivalence of (i) and (iii) can be established,  $s$-minimal surfaces will enjoy the same (local) properties as the ones established here in the case of compact manifolds (e.g. the monotonicity formula), with almost identical proofs. See \cite{FracSobPaper} and \cite{BGS} for more details about this equivalence.

\subsection{The fractional Sobolev energy}
\begin{definition}
    We define the fractional Sobolev seminorm $[u]_{H^{s/2}(M)}$ for $s \in (0,2)$ as  
 \begin{equation}\label{sobint}
     [u]^2_{H^{{s/2}}(M)} =\iint_{M\times M} (u(p)-u(q))^2 K_s(p,q) \, dV_p dV_q \, .
\end{equation}
\end{definition}
Then, the associated functional space $H^{s/2}(M)$ is
\begin{equation*}
H^{s/2}(M)=\{u\in L^2(M)  \text{ : } [u]^2_{H^{s/2}(M)}<\infty\} \,,
\end{equation*}
and is called the \textit{fractional Sobolev space of order $s/2$}. This is a Hilbert space with norm given by 
\begin{equation*}
    \|u\|_{H^{{s/2}}(M)}^2=\|u\|_{L^2(M)}^2+[u]^2_{H^{{s/2}}(M)}\, .
\end{equation*}
A one-line computation using \eqref{singintlap} shows:
\begin{proposition}
    For a smooth function, one has that \begin{equation*}
    [u]_{H^{s/2}(M)}^2 = 2\int_{M} u(-\Delta)^{s/2} u \, dV \, .
\end{equation*}
\end{proposition}

In general, the fractional Sobolev seminorm can also be expressed using spectral or extension approaches:
\begin{proposition}[\cite{FracSobPaper}]\label{afdsgsdgh}
Let $u\in H^{s/2}(M)$. Then, with the notation in the previous sections, the fractional Sobolev seminorm \eqref{sobint} is equal to
\begin{equation*}
    [u]^2_{H^{{s/2}}(M)} = 2\sum_{k=1}^\infty \lambda_k^{s/2} \lp u, \phi_k\rp_{L^2(M)} ^2
\end{equation*}
and
\begin{equation}\label{sobext}
   [u]^2_{H^{{s/2}}(M)} = \inf_{v \in \widetilde H^1(\widetilde M)} \left\{ 2\beta_s \int_{\widetilde{M}} |\widetilde \nabla v|^2 z^{1-s} \, dVdz \, : \, v(\cdot,0)=u(\cdot) \mbox{ in } L^2(M) \right\}.
\end{equation}
Moreover, the infimum in \eqref{sobext} is attained by the unique $U \in \widetilde H^1(\widetilde M) $ given by Theorem \ref{extmfd}. In particular, we also have that
\begin{equation*}
    [u]^2_{H^{s/2}(M)}= 2\beta_s \int_{\widetilde{M}} |\widetilde{\nabla} U|^2 z^{1-s} \, dVdz \,,
\end{equation*}
where $\beta_s$ is the constant defined in \eqref{betadef}. 
\end{proposition}

We end this section by recalling two related interpolation results. The first one, after a finite covering argument, it implies in particular that the characteristic function $\chi_E$ of any set of finite perimeter $E\subset M$ is in $H^{s/2}(M)$, as long as $s\in(0,1)$.

\begin{proposition} \label{interprop} Let $s\in(0,1)$, and let $u:\B_1 \subset \R^n\to \R $ be a function of bounded variation. Then,
\begin{equation*}
    \iint_{\B_1\times \B_1} \frac{|u(x)-u(y)|^2}{|x-y|^{n+s}}\,dx\,dy \leq \frac{C(n)}{(1-s)s}[u]_{BV(\B_1)}^s\|u\|_{L^1(\B_1)}^{1-s}\, .
\end{equation*}
    
\end{proposition}
\begin{proof}
See, for instance, Proposition 4.2 in \cite{Brasc}. 
\end{proof}
The second interpolation result comes from relating the first one with the extension problem.
\begin{lemma}[\cite{FracSobPaper}]\label{lem:whtorwohh}
Let $s_0\in (0,1)$ and $s \in (s_0,1)$. Let $M$ satisfy flatness assumptions ${\rm FA}_1(M,g,1,p,\varphi)$. Using the ball notation in \eqref{notationballs}, let also $U : \ov{B}^+_1(p,0) \to (-1,1)$ be any function solving
\begin{equation*}
    \widetilde{\textnormal{div}}(z^{1-s} \widetilde \nabla U) =0\, ,
\end{equation*} 
and let $u$ be its trace on $B_1(p)$. Then for all $\varrho>0$, $R\ge1$, $k \in \R$ and  $q\in B_{1/2}(p)$ such that  $B_{R\varrho}(q)\subset B_{3/4}(p)$,  
\[
 \varrho^{s-n} \beta_s \int_{\widetilde B^+_{\varrho}(q,0)} z^{1-s} |\widetilde \nabla U|^2\,dVdz \leq \frac{C}{R^s} + \frac{C}{1-s}\bigg(\varrho^{-n}\int_{B_{R \varrho}(q)}|u + k|\,dV\bigg)^{1-s}  \bigg(\varrho^{1-n}\int_{B_{R \varrho}(q)} |\nabla u|\,dV\bigg)^s,
\]
where the constant $C$ depends only on $n$ and $s_0$, and $\beta_s$ is the constant defined in \eqref{betadef}.
\end{lemma}

\subsection{Monotonicity formula for stationary points of semilinear elliptic functionals and \texorpdfstring{$s$}{}-minimal surfaces}\label{GenMonFormula}

The following monotonicity formula from \cite{FracSobPaper} applies to any critical point of a semilinear elliptic functional with a nonnegative potential term, hence including the fractional Allen-Cahn energy and $s$-minimal surfaces. For $r>0$ and $p \in M $ denote
\begin{equation}\label{notationballs}
\begin{aligned}
     B_r(p) & = \big\{ q \in M \, : \, d_g(q, p) < r \big\} \,, \\[0.5mm]
     \ov{B}_{r}^+(p, 0) & = \big\{ (q,z) \in \ov{M} \, : \, d_{\ov{g}}((q,z), (p, 0) ) < r \big\} \,,  \\
    \partial \ov{B}_{r}^+(p, 0) & = \partial \left( \ov{B}_{r}^+(p, 0) \right) \, \\ \partial^+ \ov{B}_{r}^+(p, 0) & = \partial \ov{B}_{r}^+(p, 0) \cap \{ z>0 \} \,.
\end{aligned} 
\end{equation}
In the present section, we use $\nabla$ instead of $\widetilde \nabla$ to denote the gradient in  $\widetilde M$ with respect to the product metric.

\begin{theorem}[\cite{FracSobPaper}]\label{monfor}
Let $(M^n, g)$ be an $n$-dimensional, closed Riemannian manifold. Let $s \in (0,2)$ and
\begin{equation*}
    \mathcal E(v)=[v]^2_{H^{s/2}(M)}+\int_M F(v) \, dV ,
\end{equation*}
 where $F$ is any smooth nonnegative function. Let $u:M\to \R $ be stationary for $\mathcal{E}$ under inner variations, meaning that  $\mathcal E(u)<\infty$ and for any smooth vector field $X$ on $M$ there holds $\frac{d}{dt}\big|_{t=0} \mathcal{E}(u\circ\psi_X^t)=0$, where $\psi_X^t$ is the flow of $X$ at time $t$. For $(p_\circ,0) \in \ov{M} $ define 
\begin{equation*}
    \Phi(R) := \frac{1}{R^{n-s}} \left( \beta_s \int_{\widetilde{B}^+_R(p_\circ,0)} z^{1-s}| \nabla  U(p,z)|^2 \, dV_p dz +  \int_{B_R(p_\circ)} F(u) \, dV \right) , 
    \end{equation*}
where $U$ is the unique solution given by Theorem \ref{extmfd}. Then, there exist constants $C=C(n)$ and $R_{\rm max}=R_{\rm max}(M, p_\circ)>0$ with the following property: whenever $R_\circ \le R_{\rm max}$ and $K$ is an upper bound for all the sectional curvatures of $M$ in $B_{R_\circ}(p_\circ)$, then
\begin{equation*}
    R \mapsto \Phi(R)e^{C \sqrt{K} R } \s \textit{is nondecreasing for} \s R < R_\circ \,,
\end{equation*}
and the inequality
\begin{align*}
    \Phi'(R) \ge - C \sqrt{K} \Phi(R)+\frac{ s }{R^{n-s+1}} \int_{B_R(p_\circ)} F(u) \, dV + \frac{2\beta_s}{R^{n-s}} \int_{\partial^+ \ov{B}_R^+(p_\circ,0)} z^{1-s} \lp \nabla U , \nabla d \rp ^2 \, d\,\widetilde{\sigma}
\end{align*}
holds for all $R < R_0 $, with $d(\cdot) = d_{\ov{g}}((p_\circ,0), \,\cdot\, )$ the distance function on $\ov{M}$ from the point $(p_\circ, 0)$.

\vsp
Moreover, in the particular case where $M=\R^n$, $F\equiv 0$, $s\in(0,1)$, and $u=\chi_E-\chi_{E^c}$ where $E$ is an $s$-minimal surface, there holds
\begin{align*}
    \Phi'(R) = \frac{2\beta_{s}}{R^{n-s}} \int_{\partial^+ \ov{\B}_R^+(p_\circ,0)} z^{1-s} \lp \nabla U , \nabla d \rp ^2 \, dxdz \ge 0 \,,
\end{align*}
which shows that $\Phi$ is nondecreasing and that it is constant if and only if $E$ is a cone. 
\end{theorem}
\begin{remark}\label{injradrmk}
    It follows from the proof of Theorem \ref{monfor} in \cite{FracSobPaper} that the radius $R_{\rm max}$ in Theorem \ref{monfor} can be taken to be $R_{\rm max} = \inj_{M}(p_\circ)/4$. Moreover, since $M$ is compact $R_{\rm max}$ is uniformly bounded below as $ R_{\rm max} (M,p_\circ) \ge \inj_M/4$, for all $p_\circ \in M$.
\end{remark}

\section{Existence of min-max solutions to Allen-Cahn and convergence to a limit nonlocal minimal surface}

\subsection{Existence results -- Proof of Theorem \ref{Existence1}}
\label{ExistenceSection}

\vsp
In what follows, $(M^n,g)$ will be an arbitrary closed, $n$-dimensional Riemannian manifold, and $s_0 \in(0,1)$ will be fixed. Moreover, we will use the notation $\mathcal{E}_M^\ep(v)$ for the Allen-Cahn energy from (\ref{ACenergy}) to make the parameter $\ep$ explicit in the notation.

\subsubsection{Min-max procedure}
The solutions in Theorem \ref{Existence1} are obtained using an equivariant min-max procedure, based on the construction in \cite{GG1} and the min-max theorems of \cite{Gho1}, \cite{Gho2} and \cite{LS}. Since the topology of $H^{s/2}(M)$ is trivial, this is done by exploiting the $\mathbb{Z}_2$-symmetry of the functional $\mathcal{E}^\ep_M$. Indeed, we consider the family $\mathcal{F}_\p$ of all sets $A \subset H^{s/2}(M)\backslash \{0\}$ which are continuous odd images of $\p$-spheres:
\begin{equation*}
    \mathcal{F}_\p:= \big{\{} A = f(\Sp^\p) \ : \ f \in C^0 (\Sp^p;  H^{s/2}(M)\backslash \{0\} ) \s \text{and} \s f(-x)=-f(x) \  \forall \ x\in\Sp^\p \big{\}}.
\end{equation*}
\begin{remark}
    This min-max family has been chosen for simplicity and is well known in critical point theory (see, for example, \cite{DR99}), but other min-max families can be considered; see the seminal article \cite{LS} by Lazer-Solimini, as well as the discussion in Remark 3.7 of \cite{GG1}. In particular, one can obtain solutions in Theorem \ref{Existence1}, which also satisfy lower bounds for their (extended) Morse indices, and such that the corresponding min-max families come from a topological index. We nevertheless remark that a growth for the (proper!) index of the solutions is already implied in our case, by combining the lower energy bound in Theorem \ref{Existence1} with the upper energy bounds in Theorem \ref{A-C-energy-est} (which will be proved later).
\end{remark}
For fixed $\ep$, the min-max value of the family $\mathcal{F}_\p$ is defined as
\begin{equation}\label{minmaxval}
     c_{\ep,\p}:= \inf_{A\in\mathcal{F}_\p} \sup_{u\in A}\mathcal{E}_M^{\ep}(u).
\end{equation}
Note that, defining $T(u):=\max\{-1,\min\{u,+1\}\}$ the truncation of $u$ between the values $\pm 1$, we have that $|T(u)|(x)\le 1$ for all $x\in M$ and  $\mathcal{E}_M^\ep (T(u)) \le \mathcal{E}_M^\ep (u)$. Hence \begin{equation*}
     c_{\ep,\p} = \inf_{A\in\mathcal{F}_\p} \sup_{u\in A}\mathcal{E}_M^{\ep}(u) =  \inf_{A\in \widetilde{\mathcal{F}}_\p} \sup_{u\in A}\mathcal{E}_M^{\ep}(u) \,,
\end{equation*}
where 
\begin{equation*}
    \widetilde{\mathcal{F}}_\p = \{A \in \mathcal{F}_\p : |u|\le 1 \s \textnormal{for all} \s u \in A \} .
\end{equation*}
This shows that we can consider, in the arguments that follow, that the functions in the min-max sets have absolute values pointwise bounded by one. The proof of Theorem \ref{Existence1} relies on the existence result given by the min-max scheme and the following bound on the min-max values.

\begin{theorem}\label{mmbounds}
Let $(M^n,g)$ be a compact Riemannian manifold, $s_0 \in (0,1)$ and $s\in (s_0,1)$. Then, for every $\p \in \N$ there exists $\ep_\p>0$, depending on $M$, $s$ and $\p$, such that the min-max values \eqref{minmaxval} satisfy
\begin{equation}\label{critlevbound}
    \frac{C^{-1}}{1-s} \p^{s/n} \le c_{\ep,\p} \le \frac{C}{1-s} \p^{s/n} \,, \s \textnormal{for all} \,\, \ep \in (0,\ep_\p) \,,
\end{equation}
for some constant $C=C(M,s_0)$.
\end{theorem}
\begin{proof}
    The proof of this is contained in Subsections \ref{lowbsec} and \ref{subsecupperb} below, which deal with the lower bound and upper bound, respectively.
\end{proof}
To apply the existence result, we need the energy $\mathcal{E}^\ep_M$ to satisfy the Palais-Smale condition along appropriately bounded sequences, and this is addressed by the next lemma. We remark that the proofs of the Palais-Smale property and of the lower bound are quite similar to those of the classical Allen-Cahn equation in \cite{GG1}. 
\begin{lemma}\label{LemPS}
    Let $\ep>0$ and $s\in (0,1)$. Suppose that $(u_k)_k \subset H^{s/2}(M)$ is a sequence of functions satisfying $|u_k| \le 1$, $ | \mathcal{E}_M^{\ep}(u_k) | \le  C$, and $d{\mathcal{E}_M^{\ep}}(u_k) \to 0$ strongly in $H^{s/2}(M)$. Then, there is a subsequence of $(u_k)_k$ converging strongly in $H^{s/2}(M)$.
\end{lemma}
\begin{proof}
    The proof is an adaptation of Proposition 2.25 in \cite{PalaisSmale}. We just prove the statement for $\ep=1$, as exactly the same proof works for every fixed $\ep >0$. The boundedness of the energies $\mathcal{E}_M(u_k)$ gives the convergence
     \begin{align*}
          u_k \conv{L^2} u \s \textnormal{and} \s  u_k \wconv{H^{s/2}} u 
    \end{align*}
     of a subsequence, that we do not relabel, to some $u \in H^{s/2}(M)$. To upgrade the convergence to be in the strong sense, we use the particular form of the functional. First, note that given $v\in H^{s/2}(M)$ we have 
    \begin{align*}
        d\mathcal{E}_M^\ep(u)[v]&= \frac{1}{4}\iint (u(p)-u(q))(v(p)-v(q))\,K_s(p,q)\,dV_pdV_q + \int W'(u)v\,dV\\
        &=\lim_{k \to \infty} \frac{1}{4}\iint (u_k(p)-u_k(q))(v(p)-v(q))\,K_s(p,q)\,dV_pdV_q + \int W'(u_k)v\,dV\\[2pt]
        &= \lim_{k\to \infty} d\mathcal{E}_M^\ep(u_k)[v] =0 \,,
    \end{align*}
    where we used $|u_k|\le 1$ to pass to the limit in the term $\int W'(u_k)v$. In other words, $u$ is a critical point of $\mathcal{E}_M^\ep$. From this we deduce that
    \begin{align*}
        0 & = \lim_{k \to \infty} \left( d\mathcal{E}_M^\ep(u_k)[u_k-u]-d\mathcal{E}_M^\ep(u)[u_k-u] \right)\\
        &= \mathcal{E}^{\rm Sob}_M(u_k-u) + \int (W'(u_k)-W'(u))(u_k-u)\,dV ,
    \end{align*}
    and since the second term tends to zero, the first term must do so as well. This proves that $u_k \to u$ strongly in $H^{s/2}(M)$ and concludes the proof.
\end{proof}

\begin{theorem}\label{existencethm}
For every $\p \in \N$, there exists $\ep_\p>0$ such that: for all $\ep \in (0,\ep_\p) $ there exists $u_{\ep,\p} \in H^{s/2}(M)$ which is a critical point of $\mathcal{E}^\ep_M$ with $\mathcal{E}^\ep_M(u_{\ep,\p}) =c_{\ep,\p}$ and Morse index $m(u_{\ep,\p}) \le \p $ (see Definition \ref{MorseDef}).
\end{theorem}
\begin{proof}

Since $\mathcal{E}_M^{\ep}$ satisfies the Palais-Smale condition along appropriately bounded sequences (see Lemma \ref{LemPS} above) and since $d^2\mathcal{E}_M^\ep$ is a Fredholm operator at critical points, the min-max theorems in \cite{Gho1} will imply that there exists a critical point $u_{\ep,\p}$ for $\mathcal{E}_M^{\ep}$ at energy level $c_{\ep,\p}$ and with Morse index $m(u_{\ep,\p}) \le \p $. There is only one detail that we have to address: the min-max theorems in \cite{Gho1} would apply directly to a complete, connected Banach manifold $X$ on which $\mathbb{Z}_2$ acted freely; in our case, we want to consider the space $H^{s/2}(M)$, together with the action $x\mapsto -x$ under which $\mathcal{F}_\p$ is an invariant $p$-dimensional homotopic family, but which is not free as it maps the point $0$ to itself. This is not an issue in our case: the min-max value $c_{\ep,\p}$ satisfies the upper bound in \eqref{critlevbound}, as we will prove in Subsection \ref{subsecupperb}, and this bound holds independently of $\ep$. On the other hand, the Allen-Cahn energy of the zero function tends to infinity as $\ep \searrow 0$, which shows that there is $\ep_\p>0$ such that, for all $\ep \in (0, \ep_\p)$, every min-max sequence is uniformly separated from $0 \in H^{s/2}(M)$. Hence, the min-max theorems in \cite{Gho1} hold also in our case, and this concludes the proof of existence.
\end{proof}

\begin{proof} First, note that $\widetilde{\mathcal{F}}_\p$ is an invariant $\p$-dimensional homotopic family without boundary, in the sense of Section 3 in \cite{Gho2} with $B=\varnothing$. Moreover, for every $\ep>0$, $\mathcal{E}_M^{\ep}$ satisfies the Palais-Smale condition along appropriately bounded sequences (see Lemma \ref{LemPS} above) and $d^2\mathcal{E}_M^\ep$ is a Fredholm operator on critical points. Then, \cite[Corollary 13]{Gho2} applied with $B=\varnothing$ (see also \cite[Theorem 4]{Gho2}) implies that there exists a critical point $u_{\ep,\p}$ for $\mathcal{E}_M^{\ep}$ at energy level $c_{\ep,\p}$ and with Morse index $m(u_{\ep,\p}) \le \p $. 

\vsp 
There is only one detail that we have to address: \cite[Corollary 13]{Gho2} would apply directly to a complete connected Banach manifold $X$ on which $\mathbb{Z}_2$ acted freely. In our case, we want to consider $X=H^{s/2}(M)$, together with the action $x\mapsto -x$ under which $\widetilde{\mathcal{F}}_\p$ is an invariant $\p$-dimensional homotopic family, but which is not free since it maps the point $0$ to itself. This is not an issue for the following reason. By the upper bound for the min-max values \eqref{critlevbound} (that we will prove in Subsection \ref{subsecupperb}), we have that $c_{\ep,\p} \le C(s,\p,n)$ for every $\ep \in (0,\ep_p)$, for some $\ep_\p = \ep_\p(M,s,\p)>0$. On the other hand, the Allen-Cahn energy of the zero function tends to infinity as $\ep \searrow 0$, which shows that there is $\ep_\p>0$ such that, for all $\ep \in (0, \ep_\p)$, every min-max sequence is uniformly separated from $0 \in H^{s/2}(M)$. Hence, for a small $r>0$ the min-max result \cite[Corollary 13]{Gho2} can be applied to $X=H^{s/2}(M) \setminus \{ \|u\|_{H^{s/2}} < r  \}$ on which the action $x\mapsto -x$ is free. 
\end{proof}

Hence, to obtain Theorem \ref{Existence1} we are only left with proving the lower and upper bounds in \eqref{critlevbound}. The analog bounds in the case of min-max families of hypersurfaces and their areas were first proved by Gromov in \cite{Gromov}. In \cite{Guth}, Guth gave an elegant new proof of the result by Gromov on which the proof of our bounds is partially based. See also \cite{MN} for an adaptation of the proof by Guth to the setting of
closed manifolds, as well as \cite{GG1} for the adaptation to the classical Allen-Cahn case. Nevertheless, our proof of the bounds is closer to (and also simpler than) those in \cite{Guth} or \cite{MN}, thanks to the fact that sets of finite perimeter embed naturally\footnote{Via functions of the form $\chi_E-\chi_{E^c}$.} in the same function space $H^{s/2}(M)$ as their fractional Allen-Cahn approximations, as long as $s\in(0,1)$.

\subsubsection{Lower bound}\label{lowbsec}
In the proofs of both the upper and lower bounds, we will make use of the following simple fact.
\begin{lemma}\label{coveringballlem}
Let $(M^n, g)$ be a closed, $n$-dimensional Riemannian manifold. Then, there exist positive constants $C_0, C_1, C_2$ depending only on $M$ such that: for every $\p\in \mathbb N$ there exist $N$ disjoint balls $B_r(q_1) , \dotsc B_r(q_N)$ with
\begin{gather*}
       C_1 \, \p \le N \le C_2 \, \p \,,  \\ r = C_0 \, \p^{-1/n} \,,
\end{gather*}
and 
\begin{equation*}
     \bigcup_{i \le N} B_{3r}(q_i) = M \,.
\end{equation*}
\end{lemma}
\begin{proof}
 Since $M$ is compact, by a comparison argument there exists a constant $c=c(M) >1$ such that
 \begin{equation}\label{biLipbound}
     c^{-1} r^n \le \mathcal{H}^n (B_r(q)) \le c \, r^n \,, \s \textnormal{for all} \,\, q \in M\,,\,\, r<\textnormal{inj}(M)\,.
 \end{equation}
 We claim that, for $ C_0 = \textnormal{inj}(M)/3  $, the statement holds. Indeed, with this choice $ r=C_0 \, \p^{-1/n} \le \textnormal{inj}(M)/3$ for all $\p \in \N$. Consider the cover $\bigcup_{q \in M} B_r(q)$ of $M$, and let  $B_r(q_1) , \dotsc , B_r(q_N)$ be a maximal disjoint collection. Then, by maximality $\bigcup_{i \le N} B_{3r}(q_i) = M $. Moreover, comparing volumes
\begin{equation*}
  c^{-1} r^n N  \le \sum_{i\le N} \mathcal{H}^n (B_{r}(q_i)) \le \vol(M) \le \sum_{i\le N} \mathcal{H}^n (B_{3r}(q_i)) \le c3^n r^n N \,,
\end{equation*}
which by the choice of $r$ implies
\begin{equation*}
    C_1 \, \p := \frac{\vol(M)}{c3^nC_0^n}\, \p  \le N \le  \frac{c\vol(M)}{C_0^n} \, \p =: C_2 \, \p \,.
\end{equation*}
\end{proof}

The proof of the lower bound, which is based on the one in \cite{Guth}, depends on the next two lemmas.

\begin{lemma}
\label{LowBound1}
Let $ \{ B_r(q_i) \}_{i=1}^\p$ be a family of $\p$ balls on $M$. Then, given any $A\in\mathcal{F}_\p$, there exists some $u \in A$ such that 
\begin{equation*}
    \int_{B_r(q_i)}u=0 \s \textnormal{for all} \s i=1, \dotsc , \p .
\end{equation*}
\end{lemma}
\begin{proof}
This is simply a consequence of the Borsuk-Ulam theorem. Indeed, let $A$ be the continuous odd image of $f:\Sp^\p \to H^{s/2}(M)\setminus \{0\}$. Define the (odd) function
\begin{equation*}
    g:H^{s/2}(M)\to\mathbb{R}^\p \s \textnormal{by} \s g(u):=\left(\int_{B_r(q_1)}u, \dotsc,\int_{B_r(q_\p)}u \right) .
\end{equation*}
Then $g \circ f : \Sp^\p \to \mathbb{R}^\p$ is an odd, continuous map, and by the Borsuk-Ulam theorem there exists $ a\in \Sp^\p$ with $g\circ f(a)=0$. Hence, taking $u=f(a)$ finishes the proof. 
\end{proof}

For the next lemma, it will be convenient to define the local part of the Sobolev energy as follows. Recall
\begin{align*}
    \mathcal E_{\Omega}^{\rm Sob} (v) &= \frac 1 4 \iint_{M\times M\setminus \Omega^c\times \Omega^c} (v(p)-v(q))^2 K_s(p,q)dV_p dV_q \\ &= \frac{1}{4} \iint_{\Omega \times \Omega} (v(p)-v(q))^2 K_s(p,q)dV_p dV_q + \frac{1}{2} \iint_{\Omega \times \Omega^c} (v(p)-v(q))^2 K_s(p,q)dV_p dV_q \,.
\end{align*}
Then, we set
\begin{equation*}
    \mathcal{E}^{\rm Sob}|_{\Omega}(v) :=  \frac{1}{4} \iint_{\Omega \times \Omega} (v(p)-v(q))^2 K_s(p,q)dV_p dV_q \,.
\end{equation*}
Moreover, we also denote by 
\begin{equation*}
    \textnormal{Per}_s|_{\Omega}(E):= \mathcal{E}^{\rm Sob}|_{\Omega}(\chi_E-\chi_{E^c})
\end{equation*}
the local part of the nonlocal perimeter. 

\vsp
We stress that $\textnormal{Per}_s|_{\Omega}(E) \le \textnormal{Per}_s(E,\Omega)$, with equality iff $E\cap \Omega = \emptyset$ or $E^c\cap \Omega = \emptyset$.

\vsp
With this notation, let us recall the fractional isoperimetric inequality (which is implied, actually equivalent, to the fractional Sobolev inequality, for example, in \cite{FR}). 

\vsp 

Let $s \in (0,1)$. Then, there exists $c_{iso}=c_{iso}(n,s)>0$ such that for every $E \subset \B_1$ 
\begin{equation}\label{fracisoo}
  \textnormal{Per}_s|_{\B_1}(E) \ge c_{iso} \min \bigg\{ \frac{|E |}{|\B_1|} , \frac{|\B_1 \setminus E|}{|\B_1|} \bigg\}^{\frac{n-s}{n}} , 
\end{equation}
and actually if $s\in (s_0,1)$ then $c_{iso} \ge \frac{c'}{1-s}$ for some $c'=c'(n,s_0)$ \,.

\begin{lemma}\label{LowBound2}
Let $s_0 \in (0,1)$. Then, there exists a constant $c_0=c_0(n,s_0)>0$ such that the following holds:\\
Let $q \in M$, and assume that the flatness hypothesis ${\rm FA}_1(M,g,q,2R_0,\varphi)$ holds. Given $s\in (s_0,1)$, there exists $\ep_0=\ep_0(n,s)>0$ such that: for every $ r\le R_0 $ and $\ep \le \ep_0 r$, given any $u \in H^{s/2 }(M)$ with $|u|\le 1$ and $\int_{B_r(q)} u =0 $ there holds
\begin{equation*}
        \mathcal{E}^{\ep}|_{B_r(q)}(u) \ge \frac{c_0}{1-s}r^{n-s} \,.
\end{equation*}
\end{lemma}
\begin{proof}

Let $\psi: B_{R_0(q)} \to \R^n$ be normal coordinates at $q$, and let $g_{ij}$ denote the metric in these coordinates. It is not difficult to show that ${\rm FA}_1(M,g,q,2R_0,\varphi)$  implies
\[
\frac{1}{C} {\rm id} \le (g_{ij})(x) \le C{\rm id} \quad \mbox{in }\B_{R_0}
\]
for some $C>1$ depending only on $n$.

Let $v= u \circ \psi^{-1}$.
By assumption we have $\int_{\B_r} v \sqrt{|g|} \, dx =0$, and $|v|\le 1$. This implies 
\begin{equation*}
-1 + 1/C\le \fint_{\B_r} v  \le +1-1/C \, .
\end{equation*}
By Lemma \ref{loccomparability}
\begin{equation}\label{eqmeuc}
 \mathcal{E}^{\rm Sob}|_{B_{r}(q)}(u) + \ep^{-s} \int_{B_r(q)} W(u) \, dV \ge \frac {1}{C} \bigg(\frac{1}{4}\iint_{\B_r\times \B_r} \frac{|v(x)-v(y)|^2}{|x-y|^{n+s}}\, dxdy+ r^{-s} (\ep/r)^{-s} \int_{\B_r} W(v) \,  dx \bigg) .
\end{equation}
We claim that the right-hand side of \eqref{eqmeuc} is bounded below by $\frac{c_{iso}}{4C^2}r^{n-s}$, provided $\ep/r \le \ep_0(n,s) $, where $c_{iso}$ is the constant in the fractional isoperimetric inequality \eqref{fracisoo}.

\vsp
To prove this lower bound, by scaling invariance we may assume without loss of generality $r=1$. We argue by contradiction. Suppose there exist sequences $ \ep_k\downarrow 0 $ and $v_k \in H^{s/2}(\B_1)$ with $|v_k|\le 1$, $\fint_{\B_1} v_k \in [-1+1/C,\, 1-1/C]$, but such that
\begin{align}\label{nnnnnnn}
    \frac{1}{4}\iint_{\B_1\times \B_1} \frac{|v_k(x)-v_k(y)|^2}{|x-y|^{n+s}}\, dxdy  +  \ep_k^{-s} \int_{\B_1} W(v_k) \,  dx < \frac{c_{iso}}{4C} \,.
\end{align}
In particular, $\|v_k \|^2_{H^{s/2}(\B_1)}$ is uniformly bounded in $k$ and $\int_{\B_1} W(v_k) \to 0$. Hence, up to subsequences (that we do not relabel), we have that $v_k \wconv{} v_\infty$ weakly in $H^{s/2}(\B_1)$ and $v_k \to v$ almost everywhere. Then, by Fatou's Lemma, $\int_{\B_1} W(v_\infty) = 0$ and therefore $|v_\infty|=1$ almost everywhere. Moreover
$$\fint_{\B_1} v_\infty \in [-1+1/C,\, 1-1/C]\, .$$
This implies that  $v_\infty=\chi_E-\chi_{E^c}$ for some set $E\subset\B_1$ with $\frac{1}{|\B_1|}\min \{|E|,|\B_1 \setminus E|\}\geq \frac{1}{2C}$. By the lower-semicontinuity of the Sobolev seminorm, the fractional isoperimetric inequality \eqref{fracisoo} and \eqref{nnnnnnn} we get
\begin{equation*}
    \frac{c_{iso}}{2C} \le \textnormal{Per}_s|_{\B_1}(E) \le \frac{1}{4}\liminf_{k\to\infty} \iint_{\B_1\times \B_1} \frac{|v_k(x)-v_k(y)|^2}{|x-y|^{n+s}}\, dxdy\leq \frac{c_{iso}}{4C}\, ,
\end{equation*}
a contradiction.

\vsp
Going back to \eqref{eqmeuc}, we have therefore proved that there exists $\ep_0=\ep_0(n,s)>0$ such that, for every $\ep\leq\ep_0r$ and every $u$ as in the statement 
$$\mathcal{E}^{\ep}|_{B_{r}(q)}(u) \geq\frac{c_{iso}}{4C^2}r^{n-s}\, .$$
Since the constant $c_{iso}=c_{iso}(n,s)$ for the fractional isoperimetric inequality satisfies that $c_{iso} \ge \frac{c'(n,s_0)}{1-s}$ for some $c'=c'(n,s_0)>0$, we conclude the desired result.
\end{proof}

We can now give the proof of the lower bound.
\begin{proof}[Proof of Theorem \ref{mmbounds} (part 1).] The lower bound in \eqref{critlevbound} follows, in a simple manner, from the lemmas above: Given $\p \in \N $, by Lemma \ref{coveringballlem} find $N \ge C_1 \p $ disjoint balls $\{B_r(q_i) \}_{i=1}^N$ in $M$ with radius $r=C_0 \, \p^{-1/n}$. Moreover, up to taking $C_1$ bigger and $C_0$ smaller, we can also assume that $r<R_0$, where $R_0$ is such that ${\rm FA}_1(M,g,q,2R_0,\varphi)$ holds for all $q\in M$ (see Remark \ref{fbsvdg}). Given any $A\in\mathcal{F}_\p$, by Lemma \ref{LowBound1} there exists $u \in A$ such that
\begin{equation*}
    \int_{B_r(q_i)}u=0 \s \textnormal{for all} \s i=1, \dotsc , N .
\end{equation*}
Then, by Lemma \ref{LowBound2}, for $\ep \le \ep_0 r $ we have
\begin{equation*}
    \mathcal{E}^{\ep}|_{B_{r}(q_i)}(u) \geq \frac{c_0}{1-s}r^{n-s} \s \textnormal{for all} \s i=1, \dotsc N , 
\end{equation*}
which by the choice of $r$ implies
\begin{equation*}
    \mathcal{E}_M^{\ep}(u) \geq \sum_{i=1}^N \mathcal{E}^{\ep}|_{B_r(q_i)}(u) \geq N \frac{c_0}{1-s}r^{n-s} \ge \frac{C^{-1}}{1-s}\p^{s/n}\, ,
\end{equation*}
for some constant $C$ that depends only on $M$ and $s_0$. Since we have found such a $u\in A$ given any $A\in\mathcal{F}_p$, we deduce the desired lower bound.
\end{proof}

\subsubsection{Upper bound}\label{subsecupperb}

While the lower bound required finding a function with high energy inside every admissible set $A\in\mathcal{F}_\p$, the upper bound requires finding a single admissible set $A$ such that all its elements have "low" energy. We will explicitly construct a continuous odd map $ f :\Sp^\p \to H^{s/2}(M)\setminus \{0\}$ so that all the elements in $A=f(\Sp^\p)$ have controlled energy. These functions will be of the form $\chi_{E}-\chi_{E^c}$, for some set  $E\subset M$, and our task is then to bound the fractional perimeter of these sets.

\vsp 
The proof of the upper bound in \eqref{critlevbound} goes as follows. 

\begin{proof}[Proof of Theorem \ref{mmbounds} (part 2).] By Lemma \ref{coveringballlem} (with $\p$ replaced by $k$), for every $k \ge 1$ there exist $N\le C_2 k$ disjoint balls $B_r(q_1), \dotsc , B_r(q_N)$ of radius $r=C_0k^{-1/n}$ and with $\bigcup_{i\le N} B_{3r}(q_i) =M$. Moreover, recalling the proof of Lemma \ref{coveringballlem}, the bounds \eqref{biLipbound} hold for such an $r$.    

\vsp
Now, given $(a_0, a_1, \dotsc, a_\p) \in \Sp^\p$ consider the polynomial $P_a(z)=a_0+a_1 z +\dotsc +a_p z^\p$ and name $\{ \alpha_1, \dotsc, \alpha_\ell \} $ its real roots in increasing order, so that $\ell \le \p$. In $\R^n$ consider the set
\begin{equation*}
    E:= \bigcup_{i=1}^N \B_{3r}(3r(2i+1), 0, \dotsc, 0) \,;
\end{equation*}
these are $N$ aligned balls, tangent to each other, with centers on the $x_1$-axis. Now we split the set $E$ into two disjoint subsets $E=E^+ \cup E^-$. Given the real roots $\{ \alpha_1, \dotsc, \alpha_\ell \} $, assign the set $E \cap \{x_1\le \alpha_1\}$ to $E^+$ if $P_a(z)\ge 0$ for all $z \in (-\infty, \alpha_1 ]$, and else assign it to $E^-$ if $P_a(z)\le 0$ for all $z \in (-\infty, \alpha_1 ]$. Then, analogously assign $E\cap \{\alpha_1 < x_1 \le \alpha_2\}$ to $E^+$ if $P_a(z)\ge 0$ for all $z\in(\alpha_1,\alpha_2]$, and assign it to $E^-$ if $P_a(z) < 0$ for $z\in(\alpha_1,\alpha_2]$. Repeat this procedure until $E$ is divided into the two subsets $E^+$ and $E^-$. Note that there are at most $\p$ transitions\footnote{This corresponds to the case when $P_a(z)$ has $\p$ distinct real roots in the interval $(0, 6rN)$.} between $E^+$ and $E^-$, and thus $E^+$ has perimeter at most $|\partial E^+ | \le N|\partial \ov{B}_{3r}| + (6r)^{n-1}\p$. Now, basically we want to do the same on the balls $\{B_{3r}(q_i)\}_{i=1}^N$ on $M$ identifying $B_{3r}(q_i)$ with $\B_{3r}(3r(2i+1), 0, \dotsc, 0)$ via the exponential map, that we consider as a map 
$$ \exp_{q_i} : \B_{3r}(3r(2i+1), 0, \dotsc, 0) \to B_{3r}(q_i) \,.$$ 
In order to do so, we first have to make the covering $\{B_{3r}(q_i)\}_{i=1}^N$ of $M$ disjoint. For this purpose, for all $1\le i\le N$ we consider
\begin{equation*}
    Q_i := B_{3r}(q_i) \setminus \bigcup_{j \le i-1} B_{3r}(q_j) \,,
\end{equation*}
and note that $\{Q_i\}_{i=1}^N$ is a disjoint partition of $M$ with $Q_i \subset B_{3r}(q_i)$. Let $u_a:M \to \{+1,-1\}$ be defined as  
\begin{align*}
    u_a(q) &:= 
  \begin{cases}
  +1 &  \textnormal{if} \s q \in Q_i \,\, \textnormal{and} \,\, (\exp_{q_i})^{-1}(q) \in E^+ , \\
    -1 &  \textnormal{if} \s q \in Q_i\,\, \textnormal{and} \,\, (\exp_{q_i})^{-1}(q) \in E^- . 
  \end{cases}
\end{align*}
 Set $\Sigma_a := \partial \{u_a=1 \}$. By interpolation (use for example Proposition \ref{interprop} applied to a covering of $M$ with small enough balls) there exists $C=C(M,s_0)$ so that
\begin{align}\label{ub1}
    \mathcal{E}^{\rm Sob}_{M}(u_a) &= \frac{1}{4} [u_a]_{H^{s/2}(M)}^2 = \textnormal{Per}_s( \{u_a=1\} ) \nonumber \\ & \le \frac{C}{1-s} \mathcal{H}^{n-1} (\Sigma_a)^s \, \, \vol (\{u_a=1 \}) \le \frac{C}{1-s} \mathcal{H}^{n-1}(\Sigma_a)^s .
\end{align}
Moreover, by \eqref{biLipbound} we have 
\begin{align*}
      \mathcal{H}^{n-1} (\Sigma_a) &\le C(N|\partial\B_{3r}|+r^{n-1} \p) \\ & \le C(k r ^{n-1}+\p r^{n-1}) = C(k^{1/n}+\p k^{-1+1/n}) \,,
\end{align*}
for all $k\ge 1$, with $C$ that depends only on $M$. Clearly, the optimal value of the right-hand side is attained for $k=\p$ and gives $\mathcal{H}^{n-1} (\Sigma_a) \le C \p^{1/n}$. This, together with \eqref{ub1} and the fact that $\mathcal{E}^{\rm Pot}_M(u_a)=0$, implies 
\begin{equation}\label{upperper}
   \mathcal{E}^\ep_M(u_a) =  \mathcal{E}^{\rm Sob}_M(u_a) \le \frac{C}{1-s} \p^{s/n} \,,
\end{equation}
for every $a \in \Sp^{\p}$ and $\ep >0$, with $C$ depending only on $M$ and $s_0$.

From the definition of $E^{\pm}$, it is clear that $u_{-a}(x)=-u_a(x)$. Hence, to conclude that $ a \mapsto u_a $ is an element of $\mathcal{F}_\p$ we are left to show that this map is continuous for the strong $H^{s/2}(M)$ topology. This easily follows by interpolation. Indeed, for every $a,b \in \Sp^\p$, by Proposition \ref{interprop} (again applied to a finite covering of $M$ with small balls) we have 
\begin{align*}
    [u_a-u_b]^2_{H^{s/2}(M)} & \le \frac{C}{1-s} [u_a-u_b]_{BV(M)}^s \|u_a-u_b\|_{L^1(M)}^{1-s} \\ & \le \frac{C}{1-s} \Big(\mathcal{H}^{n-1} (\Sigma_a) + \mathcal{H}^{n-1} (\Sigma_b) \Big)^s \|u_a-u_b\|_{L^1(M)}^{1-s} \le \frac{C \p^{s/n}}{1-s}  \|u_a-u_b\|_{L^1(M)}^{1-s} , 
\end{align*}
where we have used that $\sup_{a\in \Sp^\p} \mathcal{H}^{n-1} (\Sigma_a) \le C \p^{1/n}$. From here, continuity follows since $ a \mapsto u_a $ is continuous in $L^1(M)$ by construction. 

Together with the bound \eqref{upperper}, which holds for all $a\in\Sp^\p$, this concludes the proof.

\end{proof}

One should compare the simplicity of this construction, with the one in \cite{GG1} for the classical Allen-Cahn equation and the classical perimeter, In that case, to define the $\p$-sweepouts with the correct interface one has to consider functions that are compositions of:

\begin{enumerate}[label=\textit{(\roman*)}]

\item The solution to the $1$-dimensional Allen-Cahn equation with parameter $\ep >0$.

\item  A ``modified" distance function, measuring the distance to hyperplanes $\{x_1\leq c\}$ (which play a similar role to those in our construction) but also accounting for the complex parts of the roots of the polynomials in order to smooth out the cancellations of the leaves.
\end{enumerate}
Using this composition of functions is necessary in the classical Allen-Cahn case in order to regularize the construction, as characteristic functions of sets of finite perimeter do not belong to $H^1(M)$, while they do belong to $H^{s/2}(M)$ for $s<1$.
\begin{remark}
Notice that, for every fixed $\p$, both the proofs of the lower bound and the upper bound in \eqref{critlevbound} rely on the fact that there is the same "critical scale" $r=C\p^{-1/n}$ in the construction.  This is given by dividing $M$ in $ N \sim  \p$ disjoint patches of volume of order $ \sim 1/\p $. The lower bound shows that, given any $A \in \mathcal{F}_\p$, there is one element of $A$ that has zero average - see Lemma \ref{LowBound1} - in each of these patches, and in particular this element has energy uniformly bounded from below of order $\p^{s/n}$. On the other hand, the upper bound shows that this configuration, i.e. making the transitions take place in $ N \sim \p$ disjoint patches that cover $M$, is also (of the order of) the best that one can achieve.
\end{remark}
As a consequence of the results above, we deduce our complete existence result.
\begin{proof}[Proof of Theorem \ref{Existence1}]
The statement follows from combining the existence result of Theorem \ref{existencethm} and the bounds on the min-max values given by Theorem \ref{mmbounds}.
\end{proof}

\subsection{Estimates for Allen-Cahn solutions with bounded Morse index}\label{EstimatesSection}

By Theorem \ref{extmfd}, notice that $u$ is a solution to the Allen-Cahn equation 
\begin{equation*}
    (-\Delta)^{s/2}u + \ep^{-s}W'(u)=0 \, \s \textnormal{on} \,\, M \,,
\end{equation*}
if and only if the Caffarelli-Silvestre extension $U$, i.e. the unique solution to \eqref{caffextMfd}, solves  
\begin{equation}
    \begin{cases}
        \widetilde{\textnormal{div}}(z^{1-s} \nabla U) = 0 \, , &  \mbox{in } \ov{M} \label{extprob}\\ \lim_{z \to 0^+} z^{1-s} U_z(\,\cdot\,,z) = \beta_{s}^{-1} \ep^{-s}W'(u(\cdot)) \,, & \mbox{on } M .
    \end{cases}
\end{equation}

Recall the definition of finite Morse index solutions of Definition \ref{MorseDef}.
\subsubsection{Finite Morse index and almost stability}

For critical points of local functionals, it is well known that having Morse index bounded by $m$ implies stability in one out of every $m+1$ disjoint open sets. In the nonlocal case this is not the case anymore, but in one of the sets we will be able to obtain a weaker, quantitative lower bound on the second derivative which we will refer to as \textit{almost stability}.
\begin{definition}[\textbf{Almost stability}]\label{almoststab}
Let $\Omega\subset M$ open and $u:M\to(-1,1)$ be a critical point of $\mathcal E_{\Omega}$. Given $\Lambda\in\R$, we say that $u$ is \textit{$\Lambda$-almost stable in $\Omega$} if 
\begin{equation*}
    \mathcal{E}_{\Omega}''(u)[\xi,\xi] \ge -\Lambda\|\xi \|^2_{L^1(\Omega)} \,\quad \forall \, \xi \in C_c^1(\Omega) \,.
\end{equation*}
\end{definition}

\begin{lemma}[Finite Morse index and almost stability]\label{asineq}
Let $u:M\to (-1,1)$ be a solution of the Allen-Cahn equation $(-\Delta)^{s/2}u + \ep^{-s}W'(u)=0$ on $M$ with Morse index at most $m$ (see Definition \ref{MorseDef}, with $\Omega=M$). Consider a collection $\mathcal{U}_1,\dotsc, \mathcal{U}_{m+1} \subset M$ of $(m+1)$ disjoint open sets at positive distance from each other, and set
\begin{equation*}
\Lambda := m\max_{i\neq j} \sup_{\mathcal{U}_i\times \mathcal{U}_j} K_s(p,q) <+\infty .    
\end{equation*}
Then, there is (at least) one set $\mathcal{U}_k$ among $\mathcal{U}_1, \dotsc , \mathcal{U}_{m+1}$ such that $u$ is \textit{$\Lambda$-almost stable in $\mathcal{U}_k$}, that is 
\begin{equation*}
    \mathcal{E}''(u)[\xi,\xi] \ge -\Lambda\|\xi \|^2_{L^1(\mathcal{U}_k)} \,\quad \forall \, \xi \in C_c^{1}(\mathcal{U}_k) \,.
\end{equation*}
\end{lemma}

\begin{proof}
We prove the Lemma just for $m=1$ for the sake of clarity, the proof goes on exactly the same for general $m$. Let $\xi_1\in C_c^{\infty}(\mathcal{U}_1)$ and $\xi_2\in C_c^{\infty}(\mathcal{U}_2)$. Testing the second variation of the Allen-Cahn energy, explicitly in (\ref{2ndvar}), with linear combinations of $\xi_1$ and $\xi_2$ gives
\begin{align*}
    \mathcal{E}''(u)[a\xi_1+b\xi_2] =\,
    & a^2\mathcal{E}''(u)[\xi_1,\xi_1] +b^2\mathcal{E}''(u)[\xi_{2},\xi_{2}] -2ab\iint_{\mathcal{U}_1 \times \mathcal{U}_2} \xi_1(p)\xi_2(q) K_s(p,q) \,.
\end{align*}
Since $K_s(p,q) \le \Lambda$ for all $(p,q)\in \mathcal{U}_1\times \mathcal{U}_2$, the interaction term can be bounded as
\begin{equation*}
\begin{split}
   - 2ab\iint_{\mathcal{U}_1 \times \mathcal{U}_2}\xi_1(p)\xi_2(q) K_s(p,q) & \leq 2ab \Lambda \|\xi_1\|_{L^1(\mathcal{U}_1)}\|\xi_2\|_{L^1(\mathcal{U}_2)} \\
     &\leq a^2 \Lambda\|\xi_1\|_{L^1(\mathcal{U}_1)}^2+b^2 \Lambda\|\xi_2\|_{L^1(\mathcal{U}_2)}^2 . 
\end{split}
\end{equation*}
Hence
\begin{equation}\label{fmias}
    \mathcal{E}''(u)[a\xi_1+b\xi_2]  \le a^2 \Big( \underbrace{\mathcal{E}''(u)[\xi_1,\xi_1]+\Lambda\|\xi_1\|_{L^1(\mathcal{U}_1)}^2}_{=:F_1(\xi_1)}\Big)+b^2\Big( \underbrace{\mathcal{E}''(u)[\xi_2,\xi_2]+\Lambda\|\xi_2\|_{L^1(\mathcal{U}_2)}^2}_{=:F_2(\xi_2)}\Big).
\end{equation}
We want to show that either $F_1(\xi_1) \ge 0$ for all $\xi_1 \in C_c^{\infty}(\mathcal{U}_1) $ or $F_2(\xi_2) \ge 0$ for all $\xi_2 \in C_c^{\infty}(\mathcal{U}_2) $. Suppose neither of these two held, then there would exist $\xi_1 \in  C_c^{\infty}(\mathcal{U}_1)$, $\xi_2 \in C_c^{\infty}(\mathcal{U}_2)$ such that $F_1(\xi_1)<0$ and $F_2(\xi_2)<0$. This would imply, however, that \eqref{fmias} is negative for all $(a,b)\neq (0,0)$, thus contradicting that the Morse index of $u$ is at most one.
\end{proof}

\subsubsection{Control of \texorpdfstring{$\mathcal{E}^{\rm Pot}$}{} by \texorpdfstring{$\mathcal{E}^{ \rm Sob}$}{}}

For the next results, recall the notation for balls in the extended manifold \eqref{notationballs}.  We begin with an auxiliary lemma.
\begin{lemma}\label{cutoffestlem}
    Let $s\in (0,1)$, $(M,g)$ satisfies the flatness assumption ${\rm FA}_2(M,g,2,p,\varphi) $, and let $\eta \in C_c^{2}(\B_{3/4}(0))$ be a cutoff function with $\eta =1$ in $\B_{1/2}(0)$. Define $\eta_0 =\varphi \circ \eta $ and let $\widetilde{\eta}$ solve
    \begin{equation*}
    \begin{cases} \widetilde{\textnormal{div}}(z^{1-s} \widetilde \nabla \ov{\eta}) =0  & \textnormal{in} \s \ov{B}_1^+(p,0) \,, \\[2pt] \ov{\eta}=0 & \textnormal{on} \s \partial^+\ov{B}_1(p,0) \,,   \\[2pt] \ov{\eta} = \eta_0  &   \textnormal{on} \s B_1(p) \times \{0\} \, . \end{cases}
\end{equation*} 
Then, for all $q\in B_{3/4}(p)$ there holds that
\begin{equation}\label{rrrrrrr}
       \beta_s \left| \big(-z^{1-s}\partial_z \ov{\eta}\big)( q, 0^+) \right| \le C  \s \s \textit{and} \s\s   \beta_s  \int_{\widetilde B_{1}^+(p,0)} z^{1-s} | \widetilde \nabla \ov{\eta}|^2 dVdz \le  C \,,
\end{equation}
for some dimensional $C=C(n)>0$, where $\beta_s$ is the constant defined in \eqref{betadef}.
\end{lemma}
\begin{proof}
    Let $U_0 \in \widetilde{H}^1(M \times (0,\infty))$---see Definition \ref{weighSobspacemfd}---be the unique Caffarelli-Silvestre extension of $\eta_0$ (considered on $M$ extended by zero outside $B_1(p)$), in the sense of Theorem \ref{extmfd}. Since $U_0$ and $\widetilde \eta$ are two different solutions of $\widetilde{ {\rm div}}(z^{1-s} \widetilde \nabla U) = 0$ with the same trace on $B_1(p)$, by Lemma 3.12 in \cite{FracSobPaper} (rescaled) we have that $ \beta_s \,z^{1-s}| \widetilde \nabla (U_0-\widetilde \eta)| \le C $ in $\widetilde B^+_{3/4}(p)$ for some dimensional $C$. Hence in $B_{3/4}(p)$ there holds
    \begin{equation*}
        \beta_s \left| \big(-z^{1-s}\partial_z \ov{\eta}\big)( \, \cdot \,, 0^+) \right| \le \beta_s \left| \big(-z^{1-s}\partial_z U_0 \big)( \, \cdot \,, 0^+) \right| + C = \big|(-\Delta)^{s/2} \eta_0 \big| + C \,,
    \end{equation*}
    where we have used Theorem \ref{extmfd} in the last equality. Now, a dimensional bound for the $\big|(-\Delta)^{s/2} \eta_0 \big|$ follows, for $q \in B_{3/4}(p)$, by writing 
 \begin{equation*}
     (-\Delta)^{s/2} \eta_0 (q) = \int_{B_{1}(p)} (\eta_0(q)-\eta_0(r))K_s(q,r) \, dV_r + \int_{M\setminus B_{1}(p)} (\eta_0(q)-\eta_0(r))K_s(q,r) \, dV_r 
 \end{equation*}
 and using Lemma \ref{loccomparability} and \eqref{remaining3} of Proposition \ref{prop:kern1} respectively to bound these two integrals. This concludes the proof of the first estimate in \eqref{rrrrrrr}. 

 \vsp
 The second estimate follows from the first one just integrating by parts and using ${\rm FA}_2(M,g,1,p,\varphi) $.
\end{proof}

\begin{lemma}
\label{sob-contr-pot}
Let $R \in (0,1]$, and assume $M$ satisfies flatness assumption ${\rm FA}_2(M,g,2R,p,\varphi)$. Let $\ep>0$, $s\in (0,1)$ and $u : M \to (-1,1)$ be a solution of the Allen-Cahn equation
\begin{equation}\label{pdech4}
    (-\Delta)^{s/2} u +  \ep^{-s}W'(u) =0 
\end{equation}
in $B_R(p)$, that is $\Lambda$-almost stable in $B_R(p)$, in the sense of Definition \ref{almoststab}, for $\Lambda \le  \Lambda_0 / R^{n+s}$. Then, there exists a positive constant $C=C(n, \Lambda_0)$ such that, for all $a\in (0,1]$
\begin{equation*}
R^{s-n} \mathcal E^{\rm Pot}_{B_{R/2}(p)}(u) \le C \left(\frac{\beta_s  }{a} R^{s-n}\int_{\widetilde B_{R}^+(p,0)} z^{1-s}|\widetilde \nabla U|^2 + a + (\ep/R)^s R^{s-n}\mathcal E^{\rm Pot}_{B_{R}(p)}(u) \right) , 
\end{equation*}
where $\beta_s$ is the constant defined in \eqref{betadef}. In particular, since $|u| \le 1$ for $a=1$ we have 
\begin{equation*}
  R^{s-n} \mathcal E^{\rm Pot}_{B_{R/2}(p)}(u) \le C \left( \beta_s  R^{s-n} \int_{\widetilde B_{R}^+(p,0)} z^{1-s}|\widetilde \nabla U|^2  +1 \right).
\end{equation*}
\end{lemma}
\begin{proof}
Notice that, as $\ep>0$ is arbitrary, the statement is scaling invariant. Indeed, if $u : B_R(p) \to (-1,1) $ is a $\Lambda$-almost stable solution of \eqref{pdech4} with parameter $\ep$ and $\Lambda$, then, on the rescaled manifold $(M, R^{-2}g)$, $u$ is an $(R^{n+s}\Lambda)$-almost stable solution of \eqref{pdech4} with parameter $\ep$ replaced by $\ep/R$, and $\Lambda$ replaced by $R^{n+s}\Lambda \le \Lambda_0$ since $R\le 1$ by hypothesis. Hence, we can assume $R=1$.

\vsp
In what follows, $C$ denotes a general constant that depends only on $n$. To compare the potential energies on $M$ with the Sobolev energies in the extended manifold, we need a well-chosen cutoff function $\ov{\eta}$ defined on the extended manifold $\widetilde{M}$. To this aim, let $\widetilde{\eta}$ solve
\begin{equation*}
    \begin{cases} \widetilde{\dive}(z^{1-s} \widetilde \nabla \ov{\eta}) =0  & \textnormal{in} \s \ov{B}_1^+(p,0) \,, \\[2pt] \ov{\eta}=0 & \textnormal{on} \s \partial^+\ov{B}_1(p,0) \,,   \\[2pt] \ov{\eta} = \eta_0  &   \textnormal{on} \s B_1(p) \times \{0\} \, , \end{cases}
\end{equation*}
where $\eta_0 = \varphi \circ \eta \in C_c^2(B_1(p))$ and $\eta$ is a fixed cutoff function with $ \eta = 1 $ in $\B_{2/3}(0)$ and $\eta =0$ outside $\B_{3/4}(0)$. First, since ${\rm FA}_2(M,g,2,p,\varphi)$ holds, by the estimates of Lemma \ref{cutoffestlem} we have for all $q\in B_1(p)$
\begin{equation*} 
      \beta_s \left| \big(-z^{1-s}\partial_z \ov{\eta}\big)( q, 0^+) \right| \le C \s\s \textnormal{and} \s\s  \beta_s  \int_{\widetilde B_{1}^+(p,0)} z^{1-s} | \widetilde \nabla \ov{\eta}|^2 \le C \,,
\end{equation*}
 for some dimensional constant $C=C(n)$. Note also that $|\ov{\eta}| \le 1$. Then
\begin{align*}
    \mathcal{E}^{\rm Pot}_{B_{1/2}(p)} &= \frac{\ep^{-s}}{4} \int_{B_{1/2}(p)} (1-u^2)^2 \le \frac{\ep^{-s}}{4} \int_{B_{1}(p)} (1-u^2)^2 \eta_0^2 \\&= \frac{1}{4} \left( \ep^{-s}\int_{B_{1}(p)} u^2(1-u^2)^2 \eta_0^2 + \ep^{-s} \int_{B_{1}(p)} (1-u^2)^3 \eta_0^2  \right) =: \frac{1}{4}(I+ J) \,.
\end{align*}
On the one hand by \eqref{pdech4} and the divergence theorem
\begin{align*}
I &\leq  \int_{B_{1}(p)} \ep^{-s} u^2(1-u^2) \eta_0^2 = \int_{B_{1}(p)} u\eta_0^2(-\Delta)^{{s/2}}u  \\ & = \beta_{s} \int_{B_{1}(p)} u \eta_0^2\,\big(-z^{1-s} U_z \big)( \,\cdot\,, 0^+) \\ &= \beta_{s} \int_{\widetilde B_{1}^+(p,0)} \widetilde{\dive}(z^{1-s} \widetilde \nabla  U\, U\ov{\eta}^2)
\\ &= \beta_{s} \int_{\widetilde B_{1}^+(p,0)} z^{1-s} (|\widetilde \nabla  U|^2\ov{\eta}^2 +2\ov{\eta} U \widetilde \nabla  \ov{\eta} \cdot \widetilde \nabla  U) \\
&\le \beta_{s} \int_{\widetilde B_{1}^+(p,0)} z^{1-s}\bigg( \big( 1+\tfrac{1}{a}\big)|\widetilde \nabla  U|^2\ov{\eta}^2 + aU^2|\widetilde \nabla  \ov{\eta}|^2\bigg) 
\\ &\le \frac{C\beta_s}{a} \int_{\widetilde B_{1}^+(p,0)} z^{1-s} |\widetilde \nabla  U|^2 + Ca \,,
\end{align*}
where $\beta_s$ is the constant defined in \eqref{betadef} (see also Proposition \ref{afdsgsdgh}) and we have used \eqref{extprob} and Young's inequality in the second to last line. 

\vsp
On the other hand, since $W''(u)=3u^2-1$ and $u$ is $\Lambda$-almost stable (recall Definition \ref{almoststab} and the form of the second variation \eqref{2ndvar} for the Allen-Cahn equation), testing almost stability with $\xi=(1-u^2)\eta_0$ we have
{
\allowdisplaybreaks
\begin{align*}
J &=  \int_{B_{1}(p)} \ep^{-s} (1-3u^2+2u^2)\big((1-u^2) \eta_0\big)^2 \\ & = -\ep^{-s}\int_{B_1(p)} W''(u) \big((1-u^2) \eta_0\big)^2 + 2 \ep^{-s} \int_{B_1(p)} u^2(1-u^2)^2 \eta_0^2  \\ & \le \mathcal{E}^{\rm Sob}_{B_1(p)}((1-u^2)\eta_0) + \Lambda\bigg(\int_{B_{1}(p)} \big|(1-u^2) \eta_0\big|\bigg)^2 + 2  I
\\ &\le \underbrace{\frac{\beta_s}{4} \int_{\widetilde B_{1}^+(p,0)} z^{1-s}\big|\widetilde \nabla  \big((1-U^2)\ov{\eta}\big)\big|^2 }_{=:J_1} 
+ C \underbrace{\int_{B_{1}(p)} (1-u^2)^2 \eta_0^2}_{=:J_2}
+2 I .
\end{align*}
}
Here we have bounded $\mathcal{E}^{\rm Sob}((1-u^2)\eta_0)$ by $J_1$ since the former is the infimum of $\frac{\beta_s}{4} \int z^{1-s}|\widetilde \nabla  V|^2$ over all the extensions $V$ of $(1-u^2)\eta$, and $(1-U^2) \widetilde \eta $ is one such extension. Now, since $\ov{\eta} \equiv 0$ on $\partial^+ \ov{B}_{1}(p,0)$ and $\widetilde{\dive}(z^{1-s}\widetilde \nabla  \ov{\eta})=0$ we have 
{ \allowdisplaybreaks
\begin{align*}
J_1 &= \frac{\beta_s}{4} \int_{\widetilde B_{1}^+(p,0)} z^{1-s}\bigg( 4U^2 |\widetilde \nabla  U|^2\ov{\eta}^2 + \tfrac 1 2 \widetilde \nabla  \big((1-U^2)^2\big) \cdot\widetilde \nabla (\ov{\eta}^2) + (1-U^2)^2 |\widetilde \nabla  \ov{\eta}|^2\bigg) \\ &= \frac{\beta_s}{4} \bigg(
4\int_{\widetilde B_{1}^+(p,0)} z^{1-s} U^2 |\widetilde \nabla  U|^2\ov{\eta}^2 - \int_{B_1(p)} z^{1-s}(1-u^2)^2 \eta_0 \partial_z \ov{\eta} +  \\ & \quad \qquad - \int_{\widetilde B_{1}^+(p,0)} (1-U^2)^2 \widetilde{\dive} (z^{1-s}\ov{\eta} \widetilde \nabla \ov{\eta}) + \int_{\widetilde B_{1}^+(p,0)} (1-U^2)^2|\widetilde \nabla  \ov{\eta}|^2 \bigg) \\ &=
C \beta_s \int_{\widetilde B_{1}^+(p,0)} z^{1-s} U^2 |\widetilde \nabla  U|^2\ov{\eta}^2 + \int_{B_{1}(p)}  (1-u^2)^2 \eta_0\,\big(-\beta_s z^{1-s}\partial_z \ov{\eta}\big)(\,\cdot\,, 0^+)\\
& \le C\beta_s \int_{\widetilde B_{1}^+(p,0)} z^{1-s} |\widetilde \nabla  U|^2 + C\ep^s\mathcal E^{\rm Pot}_{B_{1}(p)} \,, 
\end{align*} }
and also
\[
J_2 = \int_{B_{1}(p)} (1-u^2)^2 \eta_0^2 \le C\ep^s\mathcal E^{\rm Pot}_{B_{1}(p)}.
\]
Thus 
\begin{align*}
    J=J_1+ C J_2+2I \le C \left( \frac{\beta_s }{a} \int_{\widetilde B_{1}^+(p,0)} z^{1-s} |\widetilde \nabla  U|^2 + a + \ep^s\mathcal E^{\rm Pot}_{B_{1}(p)} \right) ,
\end{align*}
for some $C=C(n, \Lambda_0)$. Putting together the estimates above gives the result.
\end{proof}

\subsubsection{BV estimate -- Proof of Theorem \ref{BVEst}}\label{BVSection}

The aim of this subsection is to prove  Theorem \ref{BVEst}.

\begin{lemma}\label{ext2derbound}
Assume $M$ satisfies the local flatness assumption ${\rm FA}_2(M,g,1,p,\varphi)$. Let $X$ be a vector field with ${\rm spt}\, X \subset B_{3/4}(p)$, and denote by $\xi : \B_1(0) \to \R^n$ the pull back $\xi := \varphi^* X $ of $X$ via the chart $\varphi^{-1}$. Let $\ep>0$, $s\in (0,1)$ and $u : M \to (-1,1)$ be a solution of the Allen-Cahn equation $(-\Delta)^{s/2}u+\ep^{-s}W'(u)=0$ in $B_1(p)$. Then 
\begin{equation*}
    \mathcal{E}''(u)[\nabla_X u,\nabla_X u] \le C \bigg( \beta_s \int_{\widetilde B^+_{3/4}(p,0)} |\widetilde \nabla U|^2 z^{1-s}dV dz +\int_{B_{3/4}(p)} \ep^{-s}W(u)dV\bigg)\,,
\end{equation*}
where $C= C\big(n, \|\xi\|_{C^2(\B_1)}\big)>0$, $U$ is the extension of $u$, and $\beta_s$ is the constant defined in \eqref{betadef}. 
\end{lemma}
\begin{proof}
Denote by $\psi_X^t$ the flow of $X$ at time $t$ and $u_t := u \circ \psi_X^{-t}$, then by vanishing of the first variation
\begin{align}\label{hghg}
    \mathcal{E}''(u)[\nabla_X u,\nabla_X u] = \frac{d^2}{dt^2}\bigg |_{t=0} \mathcal{E}(u_t) = \lim_{t\to 0} \frac{\mathcal{E}(u_t)+\mathcal{E}(u_{-t}) -2\mathcal{E}(u)}{t^2}\,.
\end{align}

Let $\widetilde X$ be any smooth  extension of $X$ in  $ \widetilde B^+_{3/4}(p,0) $ with support compactly contained in $\widetilde B^+_{3/4}(p,0)$, in other words $\widetilde X$ vanishes in a neighborhood of $\partial^+ \widetilde B^+_{3/4}(p,0)$, and such that $\widetilde{X}^{\, n+1} \equiv 0$. This last condition implies, if $\widetilde{\psi}^{\,t}$ is the flow of $\widetilde{X}$, that $\widetilde{\psi}^{\,t}$ leaves invariant the $z$-component in the extended manifold $\ov{M}$.

\vsp
To bound the above we split the energy $\mathcal{E}$ in its Sobolev part and potential part. For the Sobolev part, by the minimality of the extension in the energy space 
\begin{equation*}
    \mathcal E^{\rm Sob} (u_t) = \frac{ \beta_s}{4} \int_{\ov{M}} z^{1-s}|\widetilde \nabla \ov{u_t} |^2 \, dVdz \le \frac{\beta_s}{4} \int_{\ov{M}} z^{1-s}|\widetilde \nabla U_t|^2 \, dVdz \,,
\end{equation*}
where $\beta_s$ is the constant in Theorem \ref{extmfd}. Here $\ov{u_t}$ is the extension of $u_t$ and $U_t := U \circ \widetilde{\psi}^{-t}$. We emphasize that, with our current notation, $U_t$ is not the extension of $u_t$, but instead the translation of $U$ via $\widetilde{\psi}^{\,t}$ in the extended manifold $\ov{M}$. Denote
\begin{equation*}
    I (t): = \int_{\widetilde M} |\widetilde \nabla U_t|^2 z^{1-s} dV  dz \,.
\end{equation*}
We then have 
\begin{align*}
    \lim_{t\to 0}  \frac{ \mathcal{E}^{\rm Sob}(u_t)+\mathcal{E}^{\rm Sob}(u_{-t}) -2\mathcal{E}^{\rm Sob}(u)}{t^2} &  \le \frac{\beta_s}{4} \left( \lim_{t\to 0} \frac{I(t)+I(-t) -2I(0)}{t^2} \right) \\ &=  \frac{\beta_s}{4} \, \frac{d^2}{dt^2}\bigg |_{t=0} I(t) \\ &= 
\frac{\beta_s}{4} \, \frac{d^2}{dt^2}\bigg |_{t=0} \int_{\widetilde B^+} |\widetilde \nabla U_t|^2 z^{1-s}dV  dz \,,
\end{align*}
Now, since $M$ satisfies local flatness assumption ${\rm FA}_2(M,g,1,p,\varphi)$, setting $\ov{\varphi}(x,z)=(\varphi(x),z)$, $\widetilde{\Omega} :=\ov{\varphi}^{-1}( \widetilde B^+_{3/4}(p,0))$, $\phi_t := \ov{\varphi}^{-1} \circ \ov{\psi}^{\, t} \circ \ov{\varphi}$, and $\widetilde U : = U \circ \ov{\varphi}$, we have
\begin{align*}
I(t)&=\int_{\widetilde B^+_{3/4}(p,0)} |\widetilde \nabla (U\circ \widetilde{\psi}^{-t})|^2 z^{1-s}dV  dz = \int_{\widetilde{\Omega}} \ov{g}^{ij} \partial_i(\widetilde U\circ \phi_{-t}) \partial_j (\widetilde U\circ \phi_{-t}) z^{1-s}\sqrt{|g|}\, dxdz \\
&=\int_{\widetilde{\Omega}}\ov{g}^{ij} ((\partial_k\widetilde U)\circ \phi_{-t}) \partial_i \phi_{-t}^k  \,((\partial_l\widetilde U)\circ \phi_{-t}  )\partial_j \phi_{-t}^l  z^{1-s} \sqrt{|{g}|} \,dxdz
\nonumber \\
& = \int_{\phi_{-t}(\widetilde{\Omega})}
 (\partial_k {U}) (\partial_l {U}) \bigg(
 \ov{g}^{ij}\partial_i     \phi_{-t}^k  \partial_j \phi_{-t}^l \sqrt{|{g}|} \bigg) \circ \phi_t  (\phi_t^{n+1})^{1-s} \,d\phi^1_t\wedge\dots\wedge  d\phi^{n+1}_t
\nonumber \\
&=\int_{\widetilde{\Omega}}
 (\partial_k U) (\partial_l U)z^{1-s} \bigg(
 \ov{g}^{ij}\,(\partial_i     \phi_{-t}^k  \partial_j \phi_{-t}^l \sqrt{|{g}|} \bigg) \circ \phi_t \,|D \phi_t| \,dxdz \nonumber
\end{align*}
Hence 
\[
I''(0) =\int_{\widetilde{\Omega}}
 (\partial_k U) (\partial_l U) z^{1-s} \frac{\partial ^2F^{kl}}{\partial t^2} (0,x,z) \,dxdz \,,
\]
where
\[
F^{kl}(t,\, \cdot\,,\, \cdot\,):=  \bigg(
 \ov{g}^{ij}\partial_i     \phi_{-t}^k  \partial_j \phi_{-t}^l \sqrt{|\ov{g}|} \bigg) \circ \phi_t \, |D \phi_t| .
\]
Since $\phi: [0, \infty ) \times\widetilde{\Omega}\to \R^{n+1}$ is the flow of $(\xi,0)$, together with the flatness assumption, a direct computation shows that 
\begin{equation*}
    \left \| \frac{\partial ^2F^{kl}}{\partial t^2} (0,\,\cdot\, ) \right \|_{L^\infty(\widetilde{\Omega})} \le C \big(n,\|\xi\|_{C^2(\B_1)} \big) .
\end{equation*}
Thus
\begin{align*}
    I''(0)  \le C \int_{\widetilde{\Omega}}
 \sum_{k=1}^{n+1} |\partial_k U|^2 z^{1-s} \,dxdz \le C\int_{\widetilde B^+_{3/4}(p,0)}  |\widetilde \nabla U|^2 z^{1-s} \, dV  dz \,,
\end{align*}
where $C= C\big(n, \|\xi\|_{C^2(\B_1)}\big)$ and we have used the flatness assumption to compare the Euclidean metric on $\R^{n+1}$ to the one on $\widetilde{M}$.

\vsp
Similarly, for the potential part of the energy 
\begin{align*}
    \lim_{t\to 0}  \frac{ \mathcal{E}^{\rm Pot}(u_t)+\mathcal{E}^{\rm Pot}(u_{-t}) -2\mathcal{E}^{\rm Pot}(u)}{t^2} & = \frac{d^2}{dt^2} \mathcal{E}^{\rm Pot}(u_t) \,.
\end{align*}
Arguing as in the last part of the proof of Lemma \ref{enboundslemma} (the one regarding the potential part of the energy, with $\ell=2$) we have
\begin{equation*}
    \frac{d^2}{dt^2} \mathcal{E}^{\rm Pot}(u_t) \le C \mathcal{E}^{\rm Pot}_{B_{3/4}(p)}(u) = C \int_{B_{3/4}(p)} \ep^{-s}W(u) \, dV \,,
\end{equation*}
where $C>0$ depends only on $\|\xi\|_{C^2(\B_1)}$ since by direct computation for the Jacobian
\begin{equation*}
    \left \| \frac{\partial^2}{\partial t^2} \left( \sqrt{|g|}|D\psi_t| \right)(0,\, \cdot \,) \right \|_{L^\infty} \le C \big( \|\xi\|_{C^2(\B_1)} \big) \,.
\end{equation*}
This, together with \eqref{hghg} and the bound for $I''(0)$, concludes the proof.
\end{proof}

\begin{proposition}[Almost stability $\Rightarrow$ BV]\label{BVest}
Let $p \in M$, $s_0 \in (0,1)$, $s \in (s_0,1)$ and assume that $M$ satisfies the flatness assumption ${\rm FA}_2(M,g,1,p,\varphi)$. Let $u :M\to (-1,1)$ be a solution of $(-\Delta)^{{s/2}} u + \ep^{-s}W'(u)=0$ which is $\Lambda$-almost stable in $ B_1(p)\subset M$ (see Definition \ref{almoststab}). 

\vspace{1pt}
Then, there exist constants $\Lambda_0$ and $ C$, depending only on $n$ and $s_0$, such that: if $\Lambda \le \Lambda_0$ then
\begin{equation*}
\int_{B_{1/4}(p)}|\nabla u| \, dV \leq \frac{C}{1-s} \,.
\end{equation*}
\end{proposition}

\begin{remark}
We emphasize that in Theorem \ref{BVest} above, $C$ does not depend on $\ep$.
\end{remark}
\begin{remark}
The blow up rate $(1-s)^{-1}$ as $s \nearrow 1$ is not expected to be sharp, but $(1-s)^{-1/2}$ is;  see Remark 2.3 in \cite{Stable}.  
\end{remark}
To prove Proposition \ref{BVest} we will need the following lemma.

\begin{lemma}\label{bdfsggwe}
Let $n\ge 2$, $p \in M$, $s_0 \in (0,1)$, $s \in (s_0,1)$ and assume that $M$ satisfies the flatness assumption ${\rm FA}_2(M,g,R,p,\varphi)$. Then, there exist $\Lambda_0$ and $C_0$, depending only on $n$ and $s_0$, such that the following property holds. Let $u :M\to (-1,1)$ be a solution of $(-\Delta)^{{s/2}} u + \ep^{-s}W'(u)=0$ which is $\Lambda$-almost stable in $ B_R(p)\subset M$ for $\Lambda \le \Lambda_0/R^{n+s}$ (see Definition \ref{almoststab}). Then, for every $\delta>0$
\begin{equation*}
      R^{1-n} \int_{B_{R/2}(p)} |\nabla  u |  dV \le \frac{C_0}{(1-s)\delta} + \delta  R^{1-n} \int_{B_{R}(p)} |\nabla  u | dV \,.
\end{equation*}
\end{lemma}

\begin{proof}
     Since the statement is scaling invariant, as the constant $C$ does not depend on $\ep$, we can assume $R=1$. See the beginning of the proof of Lemma \ref{sob-contr-pot} for details on the scaling.

     \vsp 
     We show that there exists a constant $C_0=C_0(n,s_0)>0$ such that, for any given $\delta>0$, there holds 
\begin{equation*}
      \|\nabla  u \|_{L^1(B_{1/2}(p))} \le \frac{C}{(1-s)\delta} + \delta  \|\nabla u \|_{L^1(B_1(p))} \,.
\end{equation*}
In particular, this $C$ does not depend on $\ep$. 

\vsp
Let $X$ be a vector field compactly supported in $B_{3/4}(p)$ to be chosen later, and let us denote $B:=B_1(p)$ during this proof. Let also $\nabla_X u := \lp X, \nabla u \rp$. Since the second variation \eqref{2ndvar} is continuous with respect to the $H^{s/2}(M)$ topology, by density we can test the almost stability inequality with any $\xi \in H^{s/2}(M)$ (and in particular, for any $\xi$ Lipschitz). Testing the almost stability inequality with $\xi = |\nabla_X u|$ gives
\begin{equation*}
    0 \le \mathcal{E}_B''(u)[|\nabla_X u|,|\nabla_X u|] + \Lambda\| \nabla_X u \|^2_{L^1(B)} \, .
\end{equation*}
On the other hand, 
\begin{align*}
    \mathcal{E}_B''(u)[|\nabla_X u|,|\nabla_X u|] = \mathcal{E}_B''(u)[\nabla_X u,\nabla_X u] - 4 \int_B \int_B (\nabla_X u)_+(p) (\nabla_X u)_-(q) K_s(p,q) dV_p dV_q \,,
\end{align*} 
thus we find that
\begin{equation*}
    4 \int_B \int_B (\nabla_X u)_+(p) (\nabla_X u)_-(q) K_s(p,q) dV_p dV_q \le \mathcal{E}_B''(u)[\nabla_X u,\nabla_X u] + \Lambda\| \nabla_X u \|^2_{L^1(B)} \,.
\end{equation*}
Moreover, by Lemma \ref{ext2derbound} and Lemma \ref{sob-contr-pot} respectively, we have
\begin{align*}
   \mathcal{E}_B''(u)[\nabla_X u,\nabla_X u] &\le C  \left( \beta_s \int_{\widetilde B^+_{3/4}(p,0)} |\widetilde \nabla U|^2 z^{1-s}dV dz + \mathcal{E}^{\rm Pot}_{B_{3/4}(p)}(u) \right) \\ & \le C\left( \beta_s \int_{\widetilde B^+_{1}(p,0)} |\widetilde \nabla U|^2 z^{1-s}dV dz +1 \right),
\end{align*}
for some $C=C(n, \|\xi \|_{C^2(\B_1(0))}, \Lambda_0)$, where $\xi^i = X^i \circ \varphi $ and $\Lambda_0$ will be chosen shortly depending only on $n$ and $s_0$. Hence
\begin{equation*}
    4 \int_B \int_B (\nabla_X u)_+(p) (\nabla_X u)_-(q) K_s(p,q) dV_p dV_q \le C \left( \beta_s \int_{\widetilde B^+} |\widetilde \nabla U|^2 z^{1-s}dV dz +1\right) + \Lambda\| \nabla_X u \|^2_{L^1(B)} \,.
\end{equation*}
Now, since by Lemma \ref{loccomparability} there holds $K_s(p,q) \ge c_0  >0$ for all $(p,q) \in B \times B$, for some constant $c_0=c_0(n, s_0)>0$, we have
\[
\begin{split}
       \|(\nabla_X u)_+ \|_{L^1(B)} \|(\nabla_X u)_- \|_{L^1(B)} &=  \int_B \int_B (\nabla_X u)_+(p) (\nabla_X u)_-(q) dV_p dV_q
     \\
     & \le  \frac{1}{c_0} \int_B \int_B (\nabla_X u)_+(p) (\nabla_X u)_-(q) K(p,q) dV_p dV_q \,.
\end{split}    
\]
Also
\begin{align*}
     \|(\nabla_X u)_+ & \|_{L^1(B)} - \|(\nabla_X u)_- \|_{L^1(B)} = \int_B \nabla_X u \, dV= \int_B \lp \nabla u, X \rp \, dV \\&= \int_B \dive(uX)-u\dive(X) \, dV = \int_{\partial B} u\lp X,N \rp \, d\sigma -\int_B u\dive(X) \, dV \,,
\end{align*}
where $N$ is the outer unit normal vector field to $\partial B$. Then, since $|u|\le 1$
\begin{align*}
    \left|  \|(\nabla_X u)_+  \|_{L^1(B)} - \|(\nabla_X u)_- \|_{L^1(B)} \right| \le \|X \|_{L^\infty(B)} +  \|{\rm div}(X) \|_{L^{\infty}(B)} \le C\left(\|\xi\|_{C_1(\B_1(0))} \right) 
\end{align*}
Hence,  we get
\begin{align*}
    \|\nabla_X  u \|_{L^1(B)}^2 &= \left( \|(\nabla_X u)_+  \|_{L^1(B)} + \|(\nabla_X u)_- \|_{L^1(B)}\right)^2 \\ &= \left( \|(\nabla_X u)_+  \|_{L^1(B)} - \|(\nabla_X u)_- \|_{L^1(B)}\right)^2 +4 \|(\nabla_X u)_+  \|_{L^1(B)} \|(\nabla_X u)_- \|_{L^1(B)} \\ &\le C \beta_s \int_{\widetilde B^+} |\widetilde \nabla U|^2 z^{1-s}dV dz + C+ \frac{\Lambda}{c_0}\| \nabla_X u \|^2_{L^1(B)}  \,.
\end{align*}
Thus, for $\Lambda \le \frac{1}{2c_0} =:\Lambda_0$ we obtain 
\[
    \|\nabla_X  u \|^2_{L^1(B)}  \le 
    C \beta_s \int_{\widetilde B^+} |\widetilde \nabla U|^2 z^{1-s}dV dz + C .
\]
Moreover, by Lemma \ref{lem:whtorwohh} with $R=1$, $k=0$ we have
\begin{align*}
       \beta_s \int_{\widetilde B^+} |\widetilde \nabla U|^2 z^{1-s}dV dz \le \frac{C}{1-s}(1+\|\nabla u \|_{L^1(B)}) .
\end{align*}
Hence, for every $\delta>0$ by Young's inequality
\begin{align*}
    \|\nabla_X  u \|_{L^1(B)} &  \le C + C \sqrt{\frac{1}{1-s} (1 + \|\nabla u \|_{L^1(B)})} \\ &\le \frac{C}{(1-s)\delta} + \delta \|\nabla u \|_{L^1(B)} .
\end{align*} 

Now, let $\eta$ be a smooth cutoff compactly supported in $B_{3/4}(p)$ and with $\eta \equiv 1$ on $B_{1/2}(p)$. Choosing $X=\eta \frac{\partial }{\partial x^i} $ above and summing up from $i=1$ to $i=n$, together with \eqref{hsohoh1}, gives 
\begin{equation*}
      \|\nabla  u \|_{L^1(B_{1/2}(p))} \le \frac{C}{(1-s)\delta} + \delta  \|\nabla u \|_{L^1(B_1(p))} ,
\end{equation*}
for some $C=C(n,s_0)$, and this concludes the proof.
     
\end{proof}

Before giving the proof of Proposition \ref{BVest}, we recall a useful covering lemma by L. Simon \cite{Simon}.

\begin{lemma}[\cite{Simon}]\label{LSlemma}
Let $\beta \in \R$, $M_0>0$, $\rho >0$ and $\mathcal{S} : \mathfrak{B} \to [0, +\infty)$ be a nonnegative function defined on the family $\mathfrak{B}$ of open balls contained in the Euclidean ball $\B_\rho(0) \subset \R^n$ that is subadditive for finite unions, meaning that whenever $B \subset \bigcup_i B_i$ a finite union then $\mathcal{S}(B) \le \sum_i \mathcal{S}(B_i)$. Then, there exists a constant $\delta = \delta(n, \theta, \beta)>0$ such that, if 
\begin{equation*}
    r^{\beta} \mathcal{S}(\B_{\theta r}(x_0)) \le \delta r^{\beta} \mathcal{S}(\B_r(x_0)) + M_0 \, \quad \textnormal{whenever} \s B_{r}(x_0) \subset \B_\rho(0) ,
\end{equation*}
then 
\begin{equation*}
    \mathcal{S}(\B_{\rho/2}(0)) \le CM_0 \,,
\end{equation*}
for some constant $C=C(n, \beta, \rho)>0$.
\end{lemma}
We can now give the proof of Proposition \ref{BVest}.

\begin{proof}[Proof of Proposition \ref{BVest} ]
Let $\Lambda_0$ and $C_0$ be the constants given by Lemma \ref{bdfsggwe}. Fix any Euclidean ball $\B_r(x) \subset \B_{3/4}(0)$. Consider the subadditive function (defined on the family of Euclidean balls contained in $\B_{3/4}(0)$)
\begin{equation*}
    \mathcal{S}(\B_r(x)) := \int_{\varphi(\B_r(x))} |\nabla u| \, dV \,.
\end{equation*}
Notice that ${\rm FA}_2(M,g,1,p,\varphi)$ implies $B_{4r/5}(\varphi(x)) \subset \varphi(\B_r(x))$ and $\varphi(\B_{1/8}(x))  \subset B_{r/5}(\varphi(x))$. Hence, by Lemma \ref{bdfsggwe} applied with $R=4r/5$, for every $\delta>0$ and $\B_r(x) \subset \B_{3/4}(0)$ we have 
\begin{equation*}
    r^{1-n} \mathcal{S}(\B_{r/8}(x)) \le \delta r^{1-n} \mathcal{S}(\B_{r}(x)) + \frac{C}{(1-s)\delta} \,,
\end{equation*}
 for some $C=C(n,s_0)>0$. Using Lemma \ref{LSlemma}, taking $\delta $ the one given by the lemma with $\beta=1-n$, $\rho=3/4$, $\theta=1/8$, we find that 
\begin{equation*}
   S(\B_{3/8}(0)) = \int_{\varphi(\B_{3/8}(0))} |\nabla u| \, dV \le \frac{C}{1-s} \,,
\end{equation*}
where $C$ depends only on $n$ and $s_0$. In particular, since $B_{1/4}(p) \subset \varphi(\B_{3/8}(0))$ this concludes the proof. 
\end{proof}

Now, we will prove the full $BV$ estimate by iteratively reducing to the almost-stable case thanks to a covering lemma, which is inspired by the proof of Proposition 2.6 in \cite{FZ}. 

\vsp
In the following lemma we denote by $\Qb_r(x) \subset \R^n$ the (hyper)cube of center $x$ and side $r$.  

\begin{lemma}\label{morsecovering}
    Let $n\ge 1$, $m\ge 0$, $\theta \in (0,1)$, $D_0>0$ and $\beta>0$. Let $\mathcal{S} : \mathfrak{B} \to [0, +\infty)$ be a subadditive\footnote{Meaning subadditive for finite unions of (hyper)cubes.} function defined on the family $\mathfrak{B}$ of the (hyper)cubes contained in $\Qb_1(0) \subset \R^n$, such that 
    \begin{itemize} 
        \item[(i)] $\displaystyle\sup_{\{x\,:\,\Qb_r(x)\in\mathfrak{B}\}}\mathcal{S}(\Qb_r(x)) \to 0$ \hspace{0.2cm} as $r\to 0$. 
        \item[(ii)] Whenever $\Qb_{r}(x_0), \Qb_r(x_1), \dotsc , \Qb_r(x_m) \subset \Qb_1(0) $ are $(m+1)$ disjoint cubes of the same side at pairwise distance at least $D_0r $, then 
    \begin{equation*}
        \exists \,i \in \{0,1,\dotsc,m \} \s \textit{such that} \s \mathcal{S}(\Qb_{\theta r}(x_i)) \le r^\beta M_0 \,.
    \end{equation*}
    \end{itemize}
     Then 
    \begin{equation*}
        \mathcal{S}(\Qb_{1/2}(0)) \le CM_0 \,, 
    \end{equation*}
    for some $C=C(n,\theta, m , \beta, D_0) >0$. 

\end{lemma}

\begin{proof}
Let $\rho=2^{-k}$, for a fixed integer $k>1$, and consider the regular partition  of $\Qb_{\theta}(0)$ into $2^{kn}$ cubes of sidelength $\theta\varrho$. Let us call $\mathfrak{F}_1 := \{ \Qb_i^1 \}$ the family of cubes in this partition. In this way, clearly
$ \#\mathfrak{F}_1 \le  \rho^{-n}$. Let $x_i^1$ denote the center of the cube $\Qb_i^1$ and, for every $\lambda>0$ and cube $\Qb$ of side $r$, let $\lambda \Qb$ be the cube with the same center and side $\lambda r$. 

\vsp
Now, we split the family $\mathfrak{F}_1$ as $\mathfrak{F}_1 = \mathfrak{G}_1 \cup \mathfrak{B}_1 $ into the families of \textit{good} and \textit{bad} cubes as follows. Start by considering $\Qb_1^1$, if there holds
\begin{equation}\label{goodballsineq}
    \mathcal{S}( \Qb_1^1) \le M_0 \rho^\beta 
\end{equation} 
then it is considered a good cube, we assign it to $\mathfrak{G}_1$ and we remove it from $\mathfrak{F}_1$. On the other hand, if $\Qb_1^1$ does not satisfy \eqref{goodballsineq}, then we assign it to the bad cubes $\mathfrak{B}_1$ and remove it from $\mathfrak{F}_1$. Moreover, if this happens, also all the cubes $\Qb \in \mathfrak{F}_1$ such that the distance of $\tfrac{1}{\theta}\Qb $ from $\tfrac{1}{\theta} \Qb_1^1 $ is less than $D_0\rho$ are considered bad as well, so they are assigned to $\mathfrak{B}_1$ and removed from $\mathfrak{F}_1$. Importantly, this last rule (of labeling the neighbouring cubes as also bad) is applied only to the cubes that are still in $\mathfrak{F}_1$. Once a cube is classified as good and placed in $\mathfrak{G}_1$, it is no longer reclassified in later steps.

\vsp 
By a simple count, there are at most $(2+2 D_0 +4\sqrt{n}/\theta)^n$ such cubes. We continue this procedure of splitting $\mathfrak{F}_1$ in good cubes and bad cubes until there are no cubes left. 

\vsp
By property $(ii)$, we may
have assigned cubes to the bad set $\mathfrak{B}_1$ at most at $m$ steps. Since at each of these steps we removed at most $(2+2 D_0 +4\sqrt{n}/\theta)^n$ cubes, this means that $ \# \mathfrak{B}_1 \le m (2+2 D_0 +4\sqrt{n}/\theta)^n =: N_0 $.  

\vsp
Regarding the good set $\mathfrak{G}_1$, we know it contains at most $ \# \mathfrak{F}_1 \le \rho^{-n}$ cubes since this is just the total number of
cubes in the cover. Moreover, by construction in every $ \Qb \in \mathfrak{G}_1$ we have
\begin{equation*}
    \mathcal{S}( \Qb) \le M_0\rho^{\beta} \,.
\end{equation*}
Hence 
\begin{align*}
    \mathcal{S}( \Qb_{\theta}(0) ) & \le \sum_{\Qb \in \mathfrak{G}_1} \mathcal{S}( \Qb) + \sum_{\Qb \in \mathfrak{B}_1} \mathcal{S}( \Qb)  \le  M_0 \rho^{\beta-n} + \sum_{\Qb \in \mathfrak{B}_1} \mathcal{S}(\Qb) \,.
\end{align*}
The argument continues iteratively under the same scheme, on the union of the at most $N_0$ bad cubes that are in $\mathfrak{B}_1$. Consider the partition $\mathfrak{F}_2 := \{\Qb_i^2 \}$ of the cubes in $\mathfrak{B}_1$ obtained splitting each cube into $2^{kn}$ smaller cubes of side $ \theta \rho^2$. Notice that  $\# \mathfrak{F}_2 \le N_0 \rho^{-n}$. Now assign cubes in $\mathfrak{F}_2$ to the good cubes $\mathfrak{G}_2$ or bad cubes $\mathfrak{B}_2$ as before: starting from $\Qb_1^2$, assign it to $\mathfrak{G}_2 $ if 
\begin{equation*}
    \mathcal{S}( \Qb_1^2) \le M_0 \rho^{2\beta} \,,
\end{equation*}
and then remove it from $\mathfrak{F}_2$. Else, if this is not the case we assign $\Qb_1^2$ to the bad cubes $\mathfrak{B}_2$ and remove it, together with all the cubes $\Qb \in \mathfrak{F}_2$ such that $\tfrac{1}{\theta}Q$ is at distance less than $D_0\rho^2$ from $\tfrac{1}{\theta} \Qb_1^2 $. Continue the procedure until there are no cubes left in $\mathfrak{F}_2$. By property $(ii)$ again, exactly the same argument as in the first part shows that $ \mathfrak{F}_2 $ contains at most $ N_0 = m (2+2 D_0 +4\sqrt{n}/\theta)^n $ cubes assigned to the bad set, that is $\# \mathfrak{B}_2 \le N_0$. This produces a partition $\mathfrak{F}_2 = \mathfrak{G}_2 \cup \mathfrak{B}_2$, and we get
\begin{align*}
      \sum_{\Qb \in \mathfrak{B}_1} \mathcal{S}( \Qb) & \le \sum_{\Qb \in \mathfrak{G}_2} \mathcal{S}( \Qb) + \sum_{\Qb \in \mathfrak{B}_2} \mathcal{S}( \Qb)  \le N_0 M_0 \rho^{2\beta-n} + \sum_{\Qb \in \mathfrak{B}_2} \mathcal{S}( \Qb) \,.  
\end{align*}
Iterating this argument, after $k$ steps we have always $\# \mathfrak{B}_k \le N_0$, and in particular by $(i)$ and subadditivity 
\begin{equation*}
    \mathcal{S} (\mathfrak{B}_k) \le \sum_{\Qb \in \mathfrak{B}_k} \mathcal{S}(\Qb) \to 0 \,,
\end{equation*}
since each $\Qb \in \mathfrak{B}_k$ has side $\theta \rho^{k} \to 0 $. Thus, the set of the points belonging to infinitely many bad families is $\mathcal{S}$-negligible. Hence\footnote{Note that we could also have stopped the exhaustion process when the error in the tail of \eqref{fff} is less than the constant on the right-hand side, and we would have obtained the estimate with two times this constant.}
\begin{align}
\mathcal{S}(  \Qb_{\theta} (0) ) & \le M_0 \rho^{\beta-n} + N_0 M_0 \rho^{2\beta-n}  + N_0 M_0 \rho^{3\beta-n} + \dotsc \nonumber \\ \label{fff} &  \le N_0 M_0  \rho^{\beta-n} \sum_{j \ge 0} \rho^{j\beta}  =  \frac{N_0}{\rho^n(\rho^{-\beta}-1)} M_0  \,.
\end{align} 
Now notice that $\Qb_{1/2}(0) $ can be covered, for some $\xi=\xi_n$ dimensional constant, by $\xi_n  \theta^{-n}$ many cubes of side $\theta/10$ such that the cube with the same center and side $1/10$ still is contained in $\Qb_1(0)$. Since property $(ii)$ is translation invariant, covering $\Qb_{1/2}(0)$ in such a way gives
\begin{equation*}
    \mathcal{S}(\Qb_{1/2}(0)) \le \frac{\xi_n \theta^{-n} N_0}{\rho^n(\rho^{-\beta}-1)} M_0  = \frac{\xi_n \theta^{-n} m (2 D_0 +3\sqrt{n}/\theta)^n}{\rho^n(\rho^{-\beta}-1)} M_0 \,,
\end{equation*}
and as this holds for every $\rho=2^{-k}$, just choosing any fixed $k$ gives the desired estimate.

\end{proof}

\begin{theorem}\label{BVunif1}
Suppose that $M$ satisfies the flatness assumption ${\rm FA}_2(M,g,1,p,\varphi)$, in the sense of Definition \ref{flatnessassup}. Let $s_0 \in (0,1)$, $s\in (s_0,1)$ and $u:M\to (-1,1)$ be a solution of the Allen-Cahn equation (\ref{restrictedeq}) in $B_{1}(p)\subset M$ with parameter $\ep$, and with Morse index $m_{B_{1}(p)}(u)\leq m$. Then 
\begin{equation*}
\int_{B_{1/2}(p)}|\nabla u|dx \leq \frac{C}{1-s},
\end{equation*}
for some constant $ C= C(n, s_0 ,m )$. 
\end{theorem}

\begin{proof} For a set $E\subset \R^n$ denote by $\lambda E := \{\lambda y \, : \, y \in E \}
$. Consider the subadditive function\footnote{The factor $\frac{1}{2\sqrt{n}}$ inside $\varphi \big( \frac{1}{2\sqrt{n}}\Qb \big)$ is needed to have $\frac{1}{2\sqrt{n}}\Qb \subset \B_{1/2}(0)$ for $\Qb \subset \Qb_1(0)$ in order to apply Lemma \ref{loccomparability}. } 
\begin{equation*}
    \mathcal{S}(\Qb) := \int_{\varphi \big(\frac{1}{2\sqrt{n}}\Qb \big)} |\nabla u| \, dV \,,
\end{equation*}
 defined on the cubes $\Qb \subset \Qb_1(0)$. 
 
 \vsp 
 \textbf{Claim.} $\mathcal{S}$ satisfies properties $(i)$ and $(ii)$ of Lemma \ref{morsecovering} with $M_0=C/(1-s)$, $\beta=n-1$, $\theta=1/8$, and $D_0$ depending only on $n$, $s_0$ and $m$.

 \vsp \noindent
 \textit{Proof of the claim.} The first property is clear from the definition of $\mathcal{S}$, since $u$ is smooth. The second property is a consequence of the Morse index of $u$ being at most $m$: 

\vsp
Indeed, let $\Qb_{r}(x_0), \Qb_r(x_1), \dotsc , \Qb_r(x_m) \subset \Qb_1(0) $ be $(m+1)$ disjoint cubes of the same side at pairwise distance at least $D_0r $, and let $q_i := \varphi(x_i)$. Then, since $\B_{r/2}(x_i) \subset \Qb_r(x_i)$ by Lemma \ref{asineq} and Lemma \ref{loccomparability} for at least one $\ell \in \{1, \dotsc m \}$, the inequality
\begin{equation*}
    \mathcal{E}''(u)[\xi,\xi] \ge - \frac{C_1 m}{(D_0r/2)^{n+s}} \|\xi \|^2_{L^1(B_{r/2}(q_\ell))} 
\end{equation*}
holds for all $\xi \in C_c^\infty(B_{r/2}(q_\ell))$, for some $C_1=C_1(n,s_0)$. That is, $u$ is a $\Lambda$-almost stable solution (in the sense of Definition \ref{almoststab}) in $B_{r/2}(q_\ell)$ with $\Lambda= \frac{C_1 m}{(D_0r/2)^{n+s}} $. Note that, in this case, on the rescaled manifold $ \widehat M := (M, (2/r)^2 g ) $ we have that $u$ is a $\Lambda  (r/2)^{n+s}$-almost stable solution of $(-\Delta)^{{s/2}} u + (2\ep/r)^{-s}W'(u) = 0 $ in $ \widehat B_{1}(q_\ell)$, and the flatness assumption ${\rm FA}_2(M,  (2/r)^2 g, 1 , q_\ell , \varphi_{x_\ell,r/2}) $ holds. 

\vsp
Let $\Lambda_0$ be the constant given by Proposition \ref{BVest}. Then, there exists $D_0 = D_0(n,s_0,m)>0$ so that $u$ is a $\Lambda$-almost stable solution of the Allen-Cahn equation in $\widehat B_1(q_\ell)$ with $\Lambda = \frac{C_1 m}{(D_0r/2)^{n+s}} \le \Lambda_0$, for $D_0$ sufficiently large. Hence, by Proposition \ref{BVest} we get 
\begin{equation*}
    \int_{\widehat B_{1/4}(q_\ell)} | \nabla u|_{\widehat g} \, d\widehat V \le \frac{C}{1-s} \,,
\end{equation*}
and, since $\varphi \big(\Qb_{\frac{r}{16\sqrt{n}}}(x_\ell) \big) \subset B_{r/8}(q_\ell)$, scaling back this inequality on $M$ gives
\begin{equation*}
    \mathcal{S}(\Qb_{r/8}(x_\ell)) \le \int_{ B_{r/8}(q_\ell)} |\nabla u| \, d V \le \frac{C}{1-s} r^{n-1} \,,
\end{equation*}
for some $C=C(n,s_0,m)$, and this concludes the proof of the claim.

\vsp
Hence, by Lemma \ref{morsecovering} 
\begin{equation*}
    \mathcal{S}(\Qb_{1/2}(0)) = \int_{\varphi \big(\Qb_{\frac{1}{4\sqrt{n}}}(0) \big)} |\nabla u| \, dV  \le \frac{C}{1-s} \,.
\end{equation*}
Now, the fact that the $BV$ estimate holds in $B_{1/2}(p)$ follows by ${\rm FA}_2(M,g,1,p,\varphi)$ and a standard covering argument, and this concludes the proof.
\end{proof}

As a corollary, simply by scaling we immediately get Theorem \ref{BVEst}.

\begin{proof}[Proof of Theorem \ref{BVEst}] Since ${\rm FA}_2(M,g,R,p,\varphi)$ holds, then the rescaled manifold $ \widehat{M} := (M,R^{-2}g)$ satisfies ${\rm FA}_2(M,R^{-2}g,1,p,\varphi_{0,R})$. Hence, Theorem \ref{BVunif1} gives
\begin{equation*}
    \int_{B_{1/2}(p)} |\nabla u|_{\widehat g} \, d\widehat V\le \frac{C}{1-s} \,,
\end{equation*}
for some $C=C(n,s_0,m)$. Scaling back this inequality on $M$ gives the result.

\end{proof}

\subsubsection{Density estimates -- Proof of Proposition \ref{DensEst}}\label{DensEstSection}

\begin{proof}[Proof of Proposition \ref{DensEst}]
Since the statement is scaling-invariant, we prove the result just for $R=1$. In what follows, $C,c>0$ denote constants depending only on $n$,$s$, and $m$ that can change from line to line, and, in general, $C$ is big and $c$ is small.

\vsp 
We argue by contradiction, suppose that $ \int_{B_1(p)} |1+u_\ep| \le \omega_0$ and that $\{u_{\ep} \ge -\tfrac{9}{10}\} \cap B_{1/2}(p) \neq \varnothing $, for some $1 \ge C_0 \ep$. The constant $C_0 =C_0(n,s,m)>0$ that will be chosen during the proof.

\vsp 
First, by continuity of $u_\ep$ and by taking $\omega_0< |B_{1/2}(p)|$, there will be a point
$q \in B_{1/2}(p)$  for which $|u_\ep(q)|\le \tfrac{9}{10}$.\\
{\bf Step 1.} Density lower bound.\\
 We first show:
 \textbf{Claim.} There exists $\alpha=\alpha(n,s) \in (0,1)$ such that 
 \begin{equation*}
     [u_\ep]_{C^\alpha(B_{\ep/3}(q)) } \le C \ep^{-\alpha} , \s \mbox{for all } \ep\le 1/10 . 
 \end{equation*}
Indeed, let $\eta \in C^{\infty}_c(\B_2(0))$ be a cutoff function with $\chi_{\B_{3/2}(0)}\le \eta \le \chi_{\B_{2}(0)}$. Then, the function $\widetilde u(x) : = u_\ep(\varphi(\varphi^{-1}(q) + \ep x))\eta(x)$ is well defined on the whole $\R^n$, since $\widetilde u$ depends only on the values of  $u_\ep$ in  $\varphi(\B_{2\ep}(\varphi^{-1}(q)))\subset \varphi(B_1) $. Now, by the flatness assumption ${\rm FA}_2(M,g,1,p,\varphi)$ we have that $\widetilde u $ satisfies $|L\widetilde u|\le C$ in $\B_1(0)$, where the kernel of $L$ satisfies \eqref{cdcdcd} by Proposition \ref{prop:kern1}---see in particular inequality \eqref{remaining0}. Hence, by Lemma \ref{fracholest} we have $ [\widetilde u ]_{C^\alpha(\B_{1/2}(0))}  \le C $, thus $[u_\ep]_{C^\alpha(B_{\ep/3}(q)) } \le C \ep^{-\alpha}$ as desired. 

\vsp
By the claim, for all $\ep$ sufficiently small (depending on $n$ and $s$), we have $ |u_\ep|\leq \tfrac{19}{20}$ in the ball $B_{\ep}(q)$. Using that $W(u)=\frac{1}{4}(1-u^2)^2$, we deduce
\begin{equation*}
  \ep^{s-n} \cdot \ep^{-s}\int_{B_{\ep}(q)} W(u_\ep) \,dV \geq c_n > 0
\end{equation*}
for some dimensional constant $ c_n $.

\vsp
Let now $U$ be the Caffarelli-Silvestre extension of $u_\ep$ in $ \ov{M} = M\times (0, \infty)$. The previous lower bound on the potential energy in $B_{\ep}(q)$ leads to 

\begin{equation*}
     \ep^{s-n} \widetilde{\mathcal E}_{\widetilde B^+_{\ep}}(U)  =\ep^{s-n}\left(\frac{\beta_{s}}{2}\int_{\widetilde B^+_{\ep}(q,0)}z^{1-s}|\nabla U|^2 \,dVdz + \ep^{-s}\int_{B_{\ep}(q)} W(u_\ep) \, dV \right) \geq c_n .
\end{equation*}
{\bf Step 2.} Monotonicity and almost-stable annulus.\\
By the monotonicity formula of Theorem \ref{monfor}, for $\lambda \in (0,1)$, we deduce that
\begin{equation*}
   \frac{\widetilde{\mathcal E}_{\widetilde B^+_{\rho}}(U)}{\rho^{n-s}} \ge  c \frac{\widetilde{\mathcal E}_{\widetilde B^+_{\lambda \rho}}(U)}{( \lambda \rho)^{n-s}} , \s \s \forall \, \rho\in (0, R_{\rm mon}) , 
\end{equation*}
where $R_{\rm mon}$ is the radius given by the monotonicity formula and can be taken to be $R_{\rm mon} = \inj_M(q)/4$---see Remark \ref{injradrmk}---and thus by hypothesis is $R_{\rm mon} \ge 1/8 $. In particular, subtracting $\rho^{s-n}\widetilde{\mathcal E}_{\widetilde B^+_{\lambda \rho}}(U)$ in both sides and using Step 1 gives
\begin{equation}\label{eq: mon for annulus lower bound}
   \frac{\widetilde{\mathcal E}_{\widetilde B^+_{\rho} \setminus \widetilde B^+_{\lambda \rho}}(U)}{\rho^{n-s}} \ge  c \big(1-c^{-1}\lambda^{n-s} \big) \frac{\widetilde{\mathcal E}_{\widetilde B^+_{\lambda \rho}}(U)}{( \lambda \rho)^{n-s}} \ge \frac{c c_n}{2} , 
\end{equation}
provided $\ep \le \lambda \rho$, $\rho <1/8$, and $c^{-1}\lambda^{n-s}\le  \frac{1}{2}$. Define $\lambda_0=\min\{(c/2)^{4/(n-s)},\frac{1}{4}\}$; in particular, $c^{-1}\lambda_0^{(n-s)/4}\le  \frac{1}{2}$.\\
\vsp

In other words, we also have a lower density bound on (appropriate) annuli instead of balls. We want to find an almost stable one: Let $\ell \ge 0$ be an integer that will be chosen later sufficiently large, depending only on $m$, $n$, and $s$. For $k  \in \{ \ell, \ell+1 \dotsc , \ell +  m \}$, consider the annuli 
\begin{equation*}
    A_k := B_{\lambda_0^{2k+3}} (q) \setminus B_{\lambda_0^{2k+4}} (q) , \s \mbox{ and } \s  \widetilde A_k^+ := \widetilde B^+_{\lambda_0^{2k+3}} (q,0) \setminus \widetilde B^+_{\lambda_0^{2k+4}} (q,0) . 
\end{equation*}
These are $(m+1)$ disjoint annuli at pairwise distance at least $\lambda_0^{2(m+\ell)+3}$. Since $u$ has Morse index at most $ m$ by hypothesis, by Lemma \ref{asineq} there is one of these annuli, say $A_k$, where $u$ is $\Lambda$-almost stable in $A_k$ with 
\begin{equation*}
    \Lambda = m \max_{i\neq j} \sup_{A_i \times A_j} \frac{C}{d(x,y)^{n+s}} \le C(m,\lambda_0).
\end{equation*}
{\bf Step 3.} Conclusion.\\
Set $\rho_k:= \lambda_0^{2k+3} $, i.e. the outer radius of $A_k$. Observe that ${B_{\lambda_0^{1/4} \rho_k } \setminus B_{\lambda_0^{1/2} \rho_k }}\subset A_k$. By Lemma \ref{sob-contr-pot}---applied to a suitable finite cover of ${B_{\lambda_0^{1/4} \rho_k } \setminus B_{\lambda_0^{1/2} \rho_k }}$ by balls--- for every $a \in (0,1)$ we have 
\begin{align*}
    \rho_k^{s-n} \mathcal{E}^{\rm Pot}_{B_{\lambda_0^{1/4} \rho_k } \setminus B_{\lambda_0^{1/2} \rho_k } }(u_\ep)  & \le \frac{C}{a} \rho_k^{s-n} \int_{\widetilde A^+_k}z^{1-s}|\nabla U|^2 \,dVdz + Ca + C\ep^s \rho_k^{-n} \mathcal{E}^{\rm Pot}_{A_k}(u_\ep) \\ & \le \frac{C}{a} \rho_k^{s-n} \int_{\widetilde A^+_k}z^{1-s}|\nabla U|^2 \,dVdz + Ca + C\ep^s . 
\end{align*}
By \eqref{eq: mon for annulus lower bound} applied with $\lambda = \lambda_0^{1/4}$ and $\rho=\lambda_0^{1/4}\rho_k$, and using that 
\begin{equation*}
    \widetilde B^+_{\lambda_0^{1/4}\rho_k } \setminus \widetilde B^+_{\lambda_0^{1/2}\rho_k }  \subset \widetilde A_k^+ , 
\end{equation*}
this implies
\begin{align*}
    c & \le  \big( \lambda_0^{1/4}\rho_k \big) ^{s-n} \widetilde{\mathcal E}_{\widetilde B^+_{\lambda_0^{1/4}\rho_k } \setminus \widetilde B^+_{\lambda_0^{1/2}\rho_k }} (U) \\ & \le C \rho_k^{s-n} \int_{\widetilde A^+_k}z^{1-s}|\nabla U|^2 \,dVdz  + C  \rho_k^{s-n} \mathcal{E}^{\rm Pot}_{B_{\lambda_0^{1/4}\rho_k } \setminus B_{\lambda_0^{1/2}\rho_k} }(u_\ep) \\ &  \le C \rho_k^{s-n} \int_{\widetilde A^+_k}z^{1-s}|\nabla U|^2 \,dVdz +  \frac{C}{a} \rho_k^{s-n} \int_{\widetilde A^+_k}z^{1-s}|\nabla U|^2 \,dVdz + Ca + C\ep^s  \\ &  \le C  \rho_k^{s-n} \int_{\widetilde A^+_k}z^{1-s}|\nabla U|^2 \,dVdz + \frac{c}{4} +  \frac{c}{4} , 
\end{align*}
provided we take $a$ and $\ep$ sufficiently small (depending only on $m,n$ and $s$). Hence, absorbing the last two terms to the left and using Lemma \ref{lem:whtorwohh} with $k=1$, we get 
\begin{align*}
    \frac{c}{2} & \le C  \rho_k^{s-n} \int_{ \widetilde B^+_{\rho_k} } z^{1-s}|\nabla U|^2 \,dVdz \\ & \le  \frac{C}{r^s} + C \bigg(\rho_k^{-n}\int_{B_{  r \rho_k}(q)} |1+u_\ep| \bigg)^{1-s}  \bigg( \rho^{1-n}\int_{B_{  r \rho_k}(q)} |\nabla u_\ep|\bigg)^s , 
\end{align*}
for all $r\ge 1$ provided $ 2 r \rho_k \le 1/2$.

\vsp 
Choosing $r = (c/4C)^{-1/s} $, we can absorb the first term to the left in the last inequality. After doing so, here we have to choose $\ell$ large (depending only on $m,n$ and $s$) in order to have, with the previous choice of $r$, that
\begin{equation*}
    r \rho_k \le   r\lambda_0^{2l+3} \le 1/2 . 
\end{equation*}
With these choices 
\begin{equation*}
    \frac{c}{4 } \le C \omega_0^{1-s} \bigg( \int_{B_{1/2}(q)} |\nabla u_\ep|\bigg)^s ,
\end{equation*}
and by the BV-estimate of Theorem \ref{BVEst} we reach a contradiction if the density $\omega_0$ is too small. This concludes the proof.  
\end{proof}

From the proof we can extract an additional auxiliary result.
\begin{proposition}
    Let $u:M\to (-1,1)$ be a solution of (\ref{restrictedeq}) in $B_{R}(p)\subset M$ with Morse index $m_{B_{R}(p)}(u)\leq m$, and suppose that $M$ satisfies the flatness assumption ${\rm FA}_2(M,g,R,p,\varphi)$. Then, there exist positive constants $C_0$ and $\ep_0$, depending only on $n$, $s$, and $m$, such that the following holds: whenever $\ep \le \ep_0$ and $R \ge C_0\ep$, if for some $q\in B_{R/2}(p)$ we have $|u_\ep(q)|\le \tfrac{9}{10}$ then
\begin{equation}\label{revBV}
    \int_{B_{R/2}(q)} |\nabla u_\ep| \, dV \ge c_0 R^{n-1} \,,
\end{equation}
for some $c_0=c_0(n,s,m)$. This fact will be useful in the proof of Proposition \ref{PotBound} below.
\end{proposition}
\begin{proof}
    It follows by simply repeating the proof of Proposition \ref{DensEst} above, from when we found a point $q\in B_{1/2}(p)$ with $|u_\ep(q)|\le \tfrac{9}{10}$ to the very last line, and using that $|u_\ep|\le 1 $ to estimate the density $\int_{B_{1/2}(q)} |1+u_\ep|$ from above instead of using the bound $\omega_0$.
\end{proof}

\subsubsection{Decay of \texorpdfstring{$\mathcal{E}^{\rm Pot}$}{} -- Proof of Theorem \ref{A-C-energy-est}}\label{PotDecSection}

\begin{lemma}\label{lemdecay}
Let $s \in (0, 1)$, $p \in M$, and assume that ${\rm FA}_2(M,g, p,R,\varphi)$ holds; recall Definition \eqref{flatnessassup}. Let $u_\ep:M \to (-1,1)$ be a solution of (\ref{restrictedeq}) in $B_R(p)$. Then, there exist positive constants $C=C(n,s)$ such that, if $ \ep < 1 $ and $1-|u_\ep| \le \tfrac{1}{10} $ in $\varphi(\B_R(0))$, then
\[
0\le 1-|u_\ep| \le C (\ep/R)^s \quad \mbox{in } \, \varphi (\B_{R/2}(0)).
\] 
\end{lemma}

\begin{proof}
Since the statement is scaling-invariant, we assume $R=1$. Suppose in addition that $\tfrac{9}{10} \le u_\ep \le 1$ in $\varphi(\B_1(0))$; the case $-1 \le u_\ep \le -\tfrac{9}{10}$ can be reduced to the previous by the even symmetry of $W$ (i.e., replacing $u_\ep$ by $-u_\ep$). Then, since $u_\ep$ solves \eqref{restrictedeq} we see that $v: = 1-u_\ep$ satisfies 
\begin{equation*}
    L v := (-\Delta)^{s/2} v + \frac{1}{2\ep^s} v \le 0 , \s \mbox{in }\,\varphi(\B_{1}(0)).
\end{equation*}

Now we simply build a barrier from above for $v$. Fix a smooth function $\xi_\circ \in C^{\infty}(\R^n) $ such that $\chi_{\B_{1/2}(0)} \le 1-\xi_\circ \le \chi_{\B_{3/4}(0)}$ and consider the function $ \xi := \xi_\circ \circ \varphi^{-1}$ defined on $M$, considered to be identically $1$ outside $\varphi(\B_{1}(0))$. Since ${\rm FA}_2(M,g, p,1,\varphi)$ holds, we have that ${\rm FA}_2(M,g, q,1/10, \varphi_{\varphi^{-1}(q), 1/10})$ holds for every $q\in \varphi(\B_{3/4}(0))$ (see $(c)$ in Remark \ref{flatscalingrmk}). Then, for every $q \in \varphi(\B_{3/4}(0))$ we have
\begin{align*}
\big|(-\Delta)^{s/2} \xi   \big|(q) & \le \int_{M} | \xi (q) - \xi(p)  | K_s(p,q) \, dV_p   \\ & \le C \int_{B_{1/10}(q)} {\rm dist}(p,q) K_s(p,q) \, dV_p + 2 \int_{M\setminus B_{1/10}(q)}  K_s(p,q) \, dV_p \le C_0 ,
\end{align*}
for some $C_0 >0$ that depends only on $n$ and $s$. The last estimate follows, respectively, from Lemma \ref{loccomparability} and Theorem \ref{prop:kern1} (in particular, from \eqref{remaining3}). Then, using that $\xi  \ge 0$, we have 
\begin{equation*}
    L (\xi  +2C_0 \ep^s) = (-\Delta)^{s/2} \xi  + \frac{1}{2\ep^s}(\xi  + 2C_0 \ep^s ) \ge -C_0+C_0=0 , \s \mbox{in } \, \varphi(\B_{3/4}(0)).
\end{equation*}
Since $\xi  +2C_0\ep^s > 1 \ge v$ in 
$M\setminus \varphi(\B_{3/4}(0))$ we get, by the maximum principle, that $v \le \xi  +2C_0\ep^s $ in $\varphi(\B_{3/4}(0)) $. Hence, using that $ \xi \equiv 0$ in $\varphi(\B_{1/2}(0))$, we have shown that $ 1-u_\ep \le 2C_0 \ep^s  $ in $\varphi(\B_{1/2}(0))$, as desired.
\end{proof}

The following proposition shows the quantitative convergence to zero, as $\ep \searrow 0$, of the potential energy of finite index solutions to the A-C equation \eqref{restrictedeq}. The statement and proof are inspired by the ones of Proposition 6.2 in \cite{CCS}, which deals with stable solutions of the fractional Allen-Cahn equation in $\R^n$. We moreover simplify the proof in \cite{CCS}, by using the lower bound \eqref{revBV} on the $BV$ norm that we have obtained as a byproduct of (the proof of) Proposition \ref{DensEst}.

\begin{proposition}
\label{PotBound} Let $s \in (0, 1)$, $p \in M$, and assume that the flatness assumption ${\rm FA}_2(M,g, p, R,\varphi)$ holds. Let $u_\ep:M\to(-1,1)$ be a solution of \eqref{restrictedeq} in $B_{R}(p)\subset M$ with Morse index $m_{B_{R}(p)}(u_\ep) \leq m$. Then, there exist constants $C$ and $\ep_0$, depending only on $n$, $s$, and $m$, such that for all $\ep \le \ep_0$:
\[
 \ep^{-s}\int_{B_{R/2}(p)} W(u_\ep)\,dV  \le C R^{n-s} (\ep/R)^{\beta} \,,
\]
where  $\beta := \min\big(\frac{1-s}{2}, s \big) >0$.
\end{proposition}

\begin{proof}
Given $q\in B_{R/2}(p)$, let 
\[ 
r_q : = \max ( \min\big( \tfrac{R}{16} , \tfrac 1 {2} {\rm dist} (q, \{|u|\le \tfrac {9}{10}\})),  C_0\ep),
\] 
where $C_0>0$ is a large constant, depending only on $n$ and $s$, to be chosen later. 

Observe that, if $ \tfrac{R}{16} \le C_0\ep$ then
\[
(\ep/R)^{-s} \int_{B_{R/2}(p)}  W(u_\ep)\,dV  \le (16 C_0)^s  ({\textstyle \max_{[-1,1]} }W )   |B_{R/2}(p)|   \le  C R^n  \le C  C_0^{\beta}   (\ep/R)^{\beta}R^n.
\]
Thus, we may (and do) assume that $ \tfrac{R}{16} > C_0\ep$. In particular, $r_q\in \big[C_0\ep, \tfrac{R}{16} \big]$ for all $q\in B_{R/2}(p)$. 

\vspace{3pt}

\textit{Claim.} For some constant $c=c(n,s,m)>0$, there holds 
\begin{equation} \label{whtoiwhoewih}
\int_{B_{4r_q} (q)} |\nabla u_\ep| \ge c  (r_q)^{n-1} \quad \mbox{whenever } r_q< \tfrac{R}{16}  .
\end{equation}
Indeed, if $r_q < \tfrac{R}{16}$ then necessarily $ \tfrac{R}{16} \ge  \tfrac 1 {2} {\rm dist} (q, \{|u|\le \tfrac {9}{10}\}) $, otherwise we would obtain 
\begin{equation*}
   r_q  = \max ( \min\big( \tfrac{R}{16} , \tfrac 1 {2} {\rm dist} (q, \{|u|\le \tfrac {9}{10}\})),  C_0\ep) = \max ( \tfrac{R}{16},  C_0\ep) \ge \tfrac{R}{16} , 
\end{equation*}
which contradicts that $r_q < \tfrac{R}{16} $. Hence $ \tfrac{R}{16} \ge  \tfrac 1 {2} {\rm dist} (q, \{|u|\le \tfrac {9}{10}\}) $ and 
\begin{equation*}
   r_q  =  \max ( \tfrac 1 {2} {\rm dist} (q, \{|u|\le \tfrac {9}{10}\})),  C_0\ep) \ge \tfrac 1 {2} {\rm dist} (q, \{|u|\le \tfrac {9}{10}\})) . 
\end{equation*}
Thus, since also $q\in B_{R/2}(p)$, there exists $q' \in \{|u_\ep| \le \tfrac{9}{10} \} \cap B_{\frac{R}{2} + \frac{R}{16}}(p)$ such that ${\rm dist}(q,q') \le 2r_q$. Then, choosing $\ep_0$ small and $C_0$ big (depending only on $n,s$ and $m$) according to the constants in Proposition \ref{DensEst}, by \eqref{revBV} we have 
\begin{equation*}
    \int_{B_{4r_q}(q)} |\nabla u_\ep| \ge \int_{B_{2r_q}(q')} |\nabla u_\ep| \ge c(2r_q)^{n-1} ,
\end{equation*}
and the claim is proved.

\vsp 
We now produce a covering of $B_{R/2}(p)$ by some of the balls  $\{B_{r_q}(q) \}_{q \in B_{R/2}(p)}$  as follows. Given $k \le -5$, let $X_k := \{q\in B_{R/2}(p) \,: \,r_q \in (2^kR, 2^{k+1}R]\}$ and let $ \mathcal J_k $ be a discrete index set such that $ \{ q_j^k \}_{j \in \mathcal J_k} $ forms a maximal subset of $ X_k $ with the property that the balls $ B_{r(q_j^k)/4}(q_j^k) $ are pairwise disjoint, where we denote $ r(q_j^k) := r_{q_j^k}$. By construction of $ X_k $, it follows that  
\begin{equation*}
    X_k \subset \bigcup_{j\in \mathcal J_k} B_{r(q_j^k)}(q_j^k) , 
\end{equation*}
and that the family of enlarged balls
\begin{equation*}
    \big\{ B_{4r(q_j^k)}(q_j^k) \big\}_{j\in \mathcal J_k}
\end{equation*} 
has  (dimensional) finite overlapping.\footnote{That is, every point $ q \in \big\{ B_{4r(q_j^k)}(q_j^k) \big\}_{j\in \mathcal J_k} $ belongs to at most $N=N(n)$ of these balls. This is easy to check: if $q$ belongs to $N$ of such balls, we would have the existence of $N$ points $q_j^k$ in $B_{4 \cdot R 2^{k+1}}(p)$ such that the balls $B_{\frac{1}{4} R 2^k}(q_j^k)$ are disjoint and contained in $B_{9\cdot R 2^{k}}(p)$. Then, comparing the volumes and using that ${\rm FA}_2(M,g, p, R,\varphi)$ holds gives a dimension bound on $N$.} 
Note also that since $R/16 > C_0\ep $ we have $\lfloor\log_2 (C_0\ep/R)\rfloor \le -5$ and, by construction, the union of the sets $X_k$ when $k$ runs on  $\{\lfloor\log_2 (C_0\ep/R)\rfloor \le k \le -5\}$ covers all of $B_{R/2}(p)$.

\vsp
Now, on the one hand,  by the BV-estimate of Theorem \ref{BVEst} we have $\int_{B_{3R/4}(p)} |\nabla u_\ep| \le CR^{n-1}$, and this yelds 
\begin{equation} \label{whtoiwhoewih2}
\#\mathcal J_k \le C 2^{-k(n-1)} \,,
\end{equation}
for all $k \le -5$. Indeed, this follows using the fact that the balls  $ \big\{ B_{ 4 r(q_j^k)}(q_j^k) \big\}_{j \in \mathcal{J}_k}$  have finite overlap and are all contained in $B_{3/4R}(p)$. Indeed, when $k<-5$  then $r(q_j^k) < \frac{R}{16}$ and hence all the balls satisfy \eqref{whtoiwhoewih} and are strictly contained in $B_{3/4R}(p)$ by construction, while for $k=-5$ the radius of the balls is at least $\tfrac{R}{16}$ so their number must be bounded.

\vsp
On the other hand, for any given $\alpha \in [0,2s]$, we claim that Lemma \ref{lemdecay} yields
\[
\fint_{B_{r_q}(q)}  W(u_\ep) \,dV = \fint_{B_{r_q}(q)}  \frac{1}{4}(1-u_\ep^2)^2 \,dV \le C  \Big(\frac{\ep}{r_q}\Big)^\alpha  , 
\]
where $\fint $ denotes the integral average. Indeed, note that if $r_q= C_0 \ep$ the previous estimate is trivial, while if $r_q> C_0\ep$ then $r_q = \tfrac 1 2 {\rm dist} (q, \{|u|\le \tfrac {9}{10}\})$ and hence we may apply Lemma \ref{lemdecay}
(recall that $r_q\ge C_0\ep\ge \ep$) in $B_{2r_q}(q)$ to get the desired bound.
 
\vsp 
Therefore, choosing $\alpha := \min\big( \frac{1+s}{2}, 2s\big)\in (0,1)$ we obtain---using \eqref{whtoiwhoewih2}---that
\[
\begin{split}
(\ep/R)^{-s} \int_{B_{R/2}(p)}  W(u_\ep)\,dV &\le C \sum_{k= \lfloor\log_2 (C_0\ep/R)\rfloor}^{-5}  \sum_{j\in\mathcal J_k} (\ep/R)^{-s} \int_{B_{r(q_j^k)}(q_j^k)} W(u_\ep) \,dV
\\
& \le C  \sum_{k= \lfloor\log_2 (C_0\ep/R)\rfloor}^{-5}  \sum_{j\in\mathcal J_k} (\ep/R)^{-s} \Big(\frac{\ep}{r_{q_j^k}}\Big)^\alpha r_{q_j^k}^n
\\
& \le C \sum_{k= \lfloor\log_2 (C_0\ep/R)\rfloor}^{-5} (\ep/R)^{-s} \Big(\frac{\ep}{2^kR}\Big)^\alpha (2^{k+1}R)^n (\#\mathcal J_k)
\\
& \le C \sum_{k= \lfloor \log_2 (C_0\ep/R)\rfloor}^{-5} (\ep/R)^{\alpha-s} R^n 2^{k(n-\alpha)} 2^{-k(n-1)}
\\
&\le CR^n (\ep/R)^{\alpha-s}\sum_{k= -\infty}^{-5}  (2^k)^{1-\alpha} 
\\
& \le CR^n(\ep/R)^{\beta}, 
\end{split}
\]
as we wanted to show.
\end{proof}

With the above estimates at hand, the proof of Theorem \ref{A-C-energy-est} is straightforward.

\begin{proof}[Proof of Theorem \ref{A-C-energy-est}] The bound on the Sobolev part of the energy is a direct consequence of the $BV$ estimate of Theorem \ref{BVEst} and the interpolation inequality Proposition \ref{interprop}, and the bound on the Potential part is exactly the statement of Proposition \ref{PotBound} above.
\end{proof}

\subsection{Strong convergence to a limit interface -- Proof of Theorem \ref{StrongConv}}\label{ConvSection}

With the estimates for Allen-Cahn solutions of Section \ref{EstimatesSection} at hand, we can finally prove Theorem \ref{StrongConv}.
\begin{proof}[Proof of Theorem \ref{StrongConv}] We split the proof according to the different statements in the Theorem.

\vsp \textbf{Step 1.} \label{fasdfad} Convergence in $H^{s/2}(M)$.

\vsp \noindent
    Since $M$ is compact, there is a small radius $R=R(M)>0$ so that the flatness assumption ${\rm FA}_2(M,g,R,p,\varphi_p)$ holds for every $p \in M$; see Remark \ref{fbsvdg}. We can then apply the $BV$ estimate of Theorem \ref{BVEst} to get a bound on the BV norm $[u_{\ep_j}]_{BV(B_{R/2}(p))}$ independently of $p\in M$. For any $\sigma\in (0,1)$, the interpolation result of Proposition \ref{interprop} together with the comparability between $K_{\sigma}(\varphi_p(x),\varphi_p(y))$ and $\frac{1}{|x-y|^{n+\sigma}}$ (see Lemma \ref{loccomparability}) gives then the bound
    \begin{equation*}
        \iint_{B_{R/2}(p)\times B_{R/2}(p)} |u_{\ep_j}(p)-u_{\ep_j}(q)|^2K_{\sigma}(p,q)\,dV_p\,dV_q \leq C(n,\sigma)\, ,
    \end{equation*}
    valid for any $p\in M$. Combining this with \eqref{remaining3} of Theorem \ref{prop:kern1} (with $\alpha=0$), we see that
    \begin{equation*}
        \iint_{B_{R/4}(p)\times M} |u_{\ep_j}(p)-u_{\ep_j}(q)|^2K_\sigma(p,q)\,dV_p\,dV_q \leq C(n,\sigma)\, ,
    \end{equation*}
    which after covering $M$ with finitely many such balls of radius $R/4$ shows that 
    \begin{equation*}
    \| u_{\ep_j} \|_{H^{\sigma/2}(M)} \le C(M, \sigma)\, .
    \end{equation*}
    In particular, we can choose some fixed $\sigma>s$. Then, the (standard) compactness of the inclusion\footnote{The compactness of this inclusion is well known on balls of $\R^n$. This immediately gives one way of showing it for compact manifolds as well, after covering them with a finite number of small coordinate balls and using the same estimations for the kernel as in the present proof.} $H^{\sigma/2}(M) \hookrightarrow H^{s/2}(M)$ shows that a subsequence converges strongly in $H^{s/2}(M)$ to a limit function $u_0 \in H^{s/2}(M) $. Moreover, after extracting a further subsequence (that we do not relabel), we also assume that the convergence holds almost everywhere on $M$.

\vsp \textbf{Step 2.} Convergence of the Potential energies $\mathcal{E}_{M}^{\rm Pot}(u_{\ep_j})$ to zero and structure of $u_0$.

\vsp \noindent
    Again as in Step 1, covering $M$ with a finite number of balls of radius $R$ so that ${\rm FA}_2(M,g,R,p,\varphi_p)$ holds for all $p \in M$, applying Proposition \ref{PotBound} to each ball of the covering we get (for $j$ large) that
    \begin{equation}\label{egest22}
    \mathcal E^{\rm Pot}_{M}(u_{\ep_j}) \le C(M,s,m) \ep_j^\beta \, ,
    \end{equation}
    which shows that $\mathcal{E}^{\rm Pot}_M(u_{\ep_j}) \to 0 $ as $j\to \infty$ (since then $\ep_j\to 0$).
    
    \vsp
    The fact that the limit function is of the form $u_0=\chi_{E}-\chi_{E^c}$ for a set $E\subset M$ follows: since we just proved that $ \ep_j^{-s}\int_{M} W(u_{\ep_j})\to 0$ as $j\to \infty$, of course $\int_{M} W(u_{\ep_j})\to 0$ as well. By Fatou's Lemma, we deduce that $\int_{M} W(u_0) = 0$, which shows that the limit $u_0$ can only take the values $\pm 1$. Hence $u_0=\chi_{E}-\chi_{E^c}$ for some measurable set $E \subset M$, which is actually a set of finite perimeter since the $u_{\ep_j}$ satisfy uniform BV estimates. The fact that  \eqref{cdcd}-\eqref{ababab} hold, after choosing the representative of  $E$ for which every point of $E$ with density $1$ belongs to its interior and every point of density $0$ belongs to its complement, follows from the convergence in $L^1(M)$ and the density estimate of Proposition \ref{DensEst}.
    
    \vsp \textbf{Step 3.}\label{asdfvasdf} Convergence of the level sets to $\partial E$ in the Hausdorff distance.

\vsp
    
   \vsp \noindent 
   This is a direct consequence of Lemma \ref{lemdecay} and the density estimate in Proposition \ref{DensEst}. Fix $c \in (-1,1)$. Arguing by contradiction, assume one could find points $p_j \in \{ u_{\ep_j} \ge  c \} $ and $ q_j \in E $ with $d(p_j, q_j) \ge r > 0 $ and $B_{r/2}(p_j) \cap E =\emptyset $, for some small $r >0$. By compactness for a subsequence there is $p_\circ $ such that $p_{j}\to p_\circ$, and in particular $B_{r/4}(p_\circ) \subset E^c$. This implies (up to subsequences, that we do not relabel) that $\lim_{j \to \infty} u_{\ep_j} =-1$ a.e. in $B_{r/4}(p_\circ)$. By the density estimate of Proposition \ref{DensEst} then $u_{\ep_j} \le -\tfrac{9}{10}$ in $B_{r/8}(p_\circ)$ for all $j$ sufficiently large, and with Lemma \ref{lemdecay} this implies $u_{\ep_j} \le -1 +  C(\ep_j/r)^s$. This contradicts $ u_{\ep_j}(p_j) \ge c >-1 $ for $j$ large, and concludes the proof.
    
    \vsp \textbf{Step 4.}  The limit set $E$ is stationary for the fractional perimeter.

    \vsp
    \textit{Claim.} Let $X$ be a vector field of class $C^\infty$ on $M$, and let $ \psi^t := \psi_X^t$ denote its flow at time $t$. Putting $u_{\ep_j,t}(p):=u_{\ep_j}(\psi^{-t}(p))$, then
    \begin{equation}\label{q5etgqwa4e5}
        \frac{d^\ell}{dt^\ell}\mathcal E^{\rm Sob}_M(u_{\ep_j,t})\to \frac{d^\ell}{dt^\ell}\text{Per}_s(\psi^t(E)) \s\, \textnormal{and} \s\, \frac{d^\ell}{dt^\ell}\mathcal E^{\rm Pot}_M(u_{\ep_j,t})\to 0 \,, \quad \mbox{as } j\to\infty\, .
    \end{equation}
    \begin{proof}[Proof of the Claim]
        Changing variables with the flow $\psi^t$, and denoting its Jacobian at time $t$ as $J_t$, gives that
        \begin{align}
        \frac{d^\ell}{d t^\ell}\mathcal E^{\rm Sob}_M(u_{\ep_j,t})&=\frac{d^\ell}{d t^\ell}\iint |u_{\ep_j}(\psi^{-t}(p))-u_{\ep_j}(\psi^{-t}(q))|^2 K_s(p,q)   dV_p\, dV_q \nonumber\\
        &=\iint |u_{\ep_j}(p)-u_{\ep_j}(q)|^2 \frac{d^\ell}{d t^\ell}\Big[K_s(\psi^{ t}(p),\psi^{ t}(q)) J_{ t}(p)J_{ t}(q)\Big]  dV_p\, dV_q \label{eqcvthm}\, ,
        \end{align}
        and likewise
        \begin{equation}\label{eq: per derivative under integral}
        \frac{d^\ell}{d t^\ell}\text{Per}_s(\psi^t(E))=\iint |u_{0}(p)-u_{0}(q)|^2\frac{d^\ell}{d t^\ell}\Big[K_s(\psi^{ t}(p),\psi^{ t}(q)) J_{ t}(p)J_{ t}(q)\Big]  dV_p\, dV_q\, .
        \end{equation}
    The passage of the derivatives under the integral sign in these lines can be justified as follows. For every $p\neq q$, which is a set of full measure in $M\times M$, the derivative $ \frac{d^\ell}{d t^\ell}\Big[K_s(\psi^{ t}(p),\psi^{ t}(q)) J_{ t}(p)J_{ t}(q)\Big]$ exists for all $t>0$. Moreover, by \eqref{rderest} we see that 
    \begin{equation}\label{eq: bound K flowed}
        \sup_{t\in (0,T)} \left| \frac{d^\ell}{d t^\ell}\Big[K_s(\psi^{ t}(p),\psi^{ t}(q)) J_{ t}(p)J_{ t}(q)\Big] \right| \le C_T K_s(p,q), \s \forall \, T>0 . 
    \end{equation}
Hence, uniformly for $t\in (0,T)$, the integrand in \eqref{eqcvthm} is bounded by (a constant times) $|u_{\ep_j}(p)-u_{\ep_j}(q)|^2 K_s(p,q)$, which is a function in $L^1(M\times M)$ since $u_{\ep_j} \in H^{s/2}(M)$. The dominated convergence theorem easily implies then that we can pass the derivative under the integral sign. The same argument applies to \eqref{eq: per derivative under integral}.

        \vsp 
        Now, we can rewrite the first expression as
        $$
        \frac{d^\ell}{d t^\ell}\mathcal E^{\rm Sob}_M(u_{\ep_j,t})=\iint |u_{\ep_j}(p)-u_{\ep_j}(q)|^2K_s(p,q)\frac{\frac{d^\ell}{d t^\ell}\Big[K_s(\psi^{ t}(p),\psi^{ t}(q)) J_{ t}(p)J_{ t}(q)\Big]}{K_s(p,q)}  dV_p\, dV_q\, .
        $$
        Since $u_{\ep_j}\to u_{0}$ in $H^{s/2}(M)$ by Step 1, we immediately see that
        $$
        A_j:=|u_{\ep_j}(p)-u_{\ep_j}(q)|^2K_s(p,q)\to |u_{0}(p)-u_{0}(q)|^2K_s(p,q)=:A ,  \quad \mbox{in } L^1(M\times M)\, .$$ 
        On the other hand, by the bound in \eqref{eq: bound K flowed} the fixed function $B:=\frac{1}{K_s(p,q)} \cdot \frac{d^\ell}{d t^\ell}\Big[K_s(\psi^{ t}(p),\psi^{ t}(q)) \,J_{ t}(p)J_{ t}(q)\Big]$ belongs to $L^\infty(M\times M)$. Therefore, $A_jB\to AB$ in $L^1(M\times M)$ as well, which gives the first part of the claim.

\vsp
        For the second part of the claim, which regards the derivatives of the potential energy, we change variables once again with the flow $\psi^t$, finding that 
        \begin{equation*}
\frac{d^\ell}{d t^\ell}\mathcal{E}^{\rm Pot}_{M}({v_t}) =\frac{d^\ell}{d t^\ell}\int \ep^{-s} W(v(\psi^{- t}(p)) dV_p\, =  \int \ep^{-s} W(v(p))\frac{d^\ell}{d t^\ell}\,J_{ t}(p)  dV_p\, .
\end{equation*}
Bounding the derivatives of the Jacobian (in absolute value) by a constant, we deduce that
\[
\left| \frac{d^\ell}{d t^\ell}\mathcal{E}^{\rm Pot}_{M}({v_t}) \right| \le C\mathcal{E}^{\rm Pot}_{M}({v})\, .
\]
Combining this with \eqref{egest22}, we conclude that
        $$\frac{d^\ell}{dt^\ell}\mathcal E^{\rm Pot}_M(u_{\ep_j,t})\to 0 \quad \mbox{as } j\to\infty$$
as desired.
\end{proof}
    Now that we have shown the claim, the fact that $E$ is stationary for the fractional perimeter follows from considering a vector field $X$ of class $C^\infty$ on $M$, and applying the claim (with $ \ell=1$ and $t=0$) and the stationarity of $u_{\ep_j}$ for the Allen-Cahn energy:
    \begin{align*}
        \frac{d}{dt}\bigg|_{t=0}\text{Per}_s(\psi^t(E))=\lim_{j\to\infty}\frac{d}{dt}\bigg|_{t=0}\Big[\mathcal E^{\rm Sob}_M(u_{\ep_j,t})+\mathcal E^{\rm Pot}_M(u_{\ep_j,t})\Big]=0\, .
    \end{align*}

    \vsp
\textbf{Step 5.} $E$ has Morse index at most $m$ (recall Definition \ref{WeakMorseDef}).

\vsp
\vsp
To check this, consider $(m+1)$ vector fields $X_0,...,X_m$ of class $C^\infty$ on $M$. Letting $a:=(a_0,a_1,...,a_m)\in\R^{m+1}$ and $X[a]=a_0X_0+...+a_mX_m$, we can define the quadratic form $Q_{\ep_j}(a):=\frac{d^2}{dt^2}\Big|_{t=0}\mathcal E (u_{\ep_j}\circ \psi^{-t}_{X[a]})$, which we can write as $Q_{\ep_j}(a)=Q_{\ep_j}^{kl}a_ka_l$ for some coefficients $Q_{\ep_j}^{kl}$. From \eqref{q5etgqwa4e5} and the polarization identity for a quadratic form, it is immediate to see that $Q_{\ep_j}^{kl}\to Q_{0}^{kl}$ as $j\to \infty$, where 
\begin{equation*}
    Q_0(a):=\frac{d^2}{dt^2}\Big|_{t=0}\textnormal{Per}_s(\psi^{t}_{X[a]}(E))=Q_{0}^{kl}a_ka_l . 
\end{equation*}

Now, since $u_{\ep_j}$ have Morse index $\leq m$ (with respect to the standard notion of index in Definition \ref{MorseDef}), by definition we know that for every $j$ there must exist some $a^{j}\in\Sp^{m}$ such that \begin{equation*}
    Q_{\ep_j}(a^{j})=\frac{d^2}{dt^2}\Big|_{t=0}\mathcal E (u_{\ep_j}\circ \psi^{-t}_{X[a^j]})\geq 0 . 
\end{equation*}
Then, the convergence of the coefficients $Q_{\ep_j}^{kl} \to Q_0^{kl}$ immediately shows that $Q_0(a)\geq 0$ for some $a\in\Sp^m$ as well.
\end{proof}

The rest of this section is devoted to proving that the sets constructed as limits of solutions to the Allen-Cahn equation (and which were shown to be critical points of the fractional perimeter under inner variations) are actually viscosity solutions to the NMS equation.

\begin{lemma}[Palatucci-Savin-Valdinoci \cite{PSV}]\label{lem:hoihwrhiohtiow}
There exists a unique increasing function $v_{\circ}: \R \to (-1,1)$ with $v_\circ(0)=0$ that solves $(-\Delta)^s(v_\circ) + W'(v_\circ)=0$ in $\R$. 
\end{lemma}

\begin{remark}\label{rem:shioghsoihs}
The following fact will be useful: Let $A$ be any symmetric positive definite matrix with $A_{in} = \delta_{in}$, $1\le i\le n$.
Defining $v_{\ep,\tau}(x) = v_\circ\big( \ep^{-1}(x_n-\tau)\big)$  (where $v_\circ$ is the function of Lemma \ref{lem:hoihwrhiohtiow}) we still have that
\[
\alpha_{n,s} \, \int_{\R^n} \frac{\big(v_{\ep,\tau}(x)- v_{\ep,\tau}(y)\big)}{ |A(x-y)|^{n+s}}  |A| dy  + \ep_j^{-s} W'(v_{\ep,\tau}) =0,
\]
where $|A|$ denotes the determinant of $A$. In other words, $v_{\ep,\tau}$ is also an Allen--Cahn solution ``in this new metric''.
\end{remark}

\begin{remark}
   We will implicitly use the following fact many times. Let $\varphi : \B_{r_\circ}(0) \to M $ be a diffeomorphism onto its image with $\varphi(0)=p$, and let $F \subset M$ be a measurable set. Then, for $s\in (0,1)$ the limit
   \begin{equation*}
\lim_{r\downarrow 0} \int_{M\setminus B_r(p)} (\chi_F -\chi_{F^c})(q) K_s(p,q) \,dV_q 
\end{equation*}
exists if and only if the limit
\begin{equation*}
\lim_{r\downarrow 0} \int_{M\setminus \varphi(\B_r(0))} (\chi_F -\chi_{F^c})(q) K_s(p,q) \,dV_q 
\end{equation*}
exists, and if they do exist they coincide. This is not due to cancellations and can be seen as follows: for $r$ sufficiently small (so that ${\rm FA}_1(M,g,p,r,\varphi)$ holds), by Lemma \ref{loccomparability} we can estimate 
\begin{align*}
    \bigg|\int_{M\setminus B_r(p)} (\chi_F -\chi_{F^c})(q) &K_s(p,q) \,dV_q - \int_{M\setminus \varphi(\B_r(0))} (\chi_F -\chi_{F^c})(q) K_s(p,q) \,dV_q  \bigg| \\ & \le C \int_{ B_r(p) \Delta \varphi(\B_r(0))} \frac{1}{d(q,p)^{n+s}} \, dV_q \le \frac{C}{r^{n+s}} \vol \big(B_r(p)  \triangle   \varphi(\B_r(0)) \big) = O(r^{1-s}) \to 0 , 
\end{align*}
as $r \to 0^+$, since $ \vol \big(B_r(p)  \triangle   \varphi(\B_r(0)) \big) = O(r^{n+1})$ for small $r$. 
\end{remark}

\begin{proposition}
\label{prop:viscosity}
Assume that $u_{\ep_j}$ are solutions to the  A-C equation \eqref{restrictedeq} on $M$, with parameters $\ep_j\to 0$ and Morse index $m(u_{\ep_j})\leq m$, and that moreover
$u_{\ep_j}\to u_0:
=\chi_E-\chi_{E^c}$ in $H^{s/2}(M)$. Then $\partial E$ is a viscosity solution of the NMS equation in the following sense: whenever  $p\in \partial E$, and $\varphi: \B_{\varrho_\circ}(0) \to V$ is a diffeomorphism from $\B_{\varrho_\circ}$ to an open neighborhood $V\subset M$ of $p$ 
satisfying  $\varphi(0)=p$ and $V^+:= \varphi(\B_{\varrho_\circ}^+)\subset E$ (where we denote $\B_{r}^+: = \B_{r}\cap \{x_n>0\}$) we have
\begin{equation*}
\lim_{r\downarrow 0} \int_{M\setminus B_r(p)} (\chi_F -\chi_{F^c})(q) K_s(p,q) \,dV_q\le 0, 
\qquad \mbox{for } F= V^+ \cup (E\setminus V).
\end{equation*}
\end{proposition}

\begin{proof}
We suppose by contradiction that for some $p$ and  $\varphi: \B_{\varrho_\circ}\to V$ as in the statement of the proposition we had
\begin{equation}\label{whothiowhw}
\lim_{r\downarrow 0} \int_{M\setminus B_r(p)} (\chi_F -\chi_{F^c})(q) K_s(p,q) \,dV_q\ge  2\delta >0, 
\qquad \mbox{for } F= V^+ \cup (E\setminus V).
\end{equation}
Our goal is now to obtain a contradiction. 

\vsp
Let us make the following useful observation that we will use several times throughout the proof. 
Let $\psi: = \B_{\varrho} \to W\subset V$ be another diffeomorphism with $\psi(0)=p$ such that $\varphi( \B_{\varrho_\circ}^+) \cap W \subset \psi(\B_\rho^+)$. Put $G=  \psi(\B_{\varrho}^+) \cup (E\setminus \psi( \B_\varrho))$; then, 
\[
(\chi_G -\chi_{G^c})(q)\ge(\chi_F -\chi_{F^c})(q)\quad \mbox{for all }q\in M.
\]
Hence, the integral \eqref{whothiowhw} only grows when replacing $F$ by $G$.
In particular, this applies to ``restrictions of domain'', such as $\psi = \varphi|_{\B_\varrho}$ for any $\varrho < \varrho_\circ$.

\vsp
\textbf{Step 1.} Setting $x=(x',x_n)$, we claim that we can replace $F$ by 
\[
F_t: = \varphi(\{x\in \B_{\varrho_\circ} : x_n>t|x'|^2\}) \cup (E\setminus V)
\]
in  \eqref{whothiowhw}, for $t>0$ sufficiently small, provided that we also replace $2\delta$ by $\delta$. In fact, this is a consequence of
\[
f(t) := \lim_{r\downarrow 0} \int_{M\setminus B_r(p)} (\chi_{F_t} -\chi_{F_t^c})(q) K_s(p,q) \,dV_q
\]
being continuous in $t$: Since $f(0)\ge 2\delta$, for $t >0$ sufficiently small, we will still have $f(t)\ge \delta >0$. We now prove that $f$ is continuous indeed: 

Fix $0 < \sigma < t $, $\varrho$ sufficiently small so that ${\rm FA}_1(M,g,p, 2\varrho, \varphi) $ holds (here we use the observation at the beginning of the proof regarding the domain restriction) and $r \ll \varrho$. Let 
\begin{equation*}
    S:= F_\sigma \setminus F_t \subset V = \varphi(\B_{\varrho}). 
\end{equation*}
We have 
\begin{align}
    \Bigg| \int_{M\setminus B_r(p)} & (\chi_{F_\sigma} -\chi_{{F_\sigma}^c})(q) K_s(p,q) \,dV_q  - \int_{M\setminus B_r(p)} (\chi_{F_t} -\chi_{{F_t}^c})(q) K_s(p,q) \,dV_q \Bigg| \nonumber \\ &= 2 \int_{V\setminus B_r(p)} \chi_{S}(q) K_s(p,q) \, dV_q \nonumber \\ &= 2  \int_{ \B_{\varrho} \setminus \varphi^{-1}(B_r(p))} \chi_{\varphi^{-1}(S)}(z) K_s(p,\varphi(z))  |J \varphi| \, dz   \nonumber \\ & \le C  \int_{ \B_{\varrho} \setminus \B_{r/2} }  \frac{\chi_{\varphi^{-1}(S)}(z)}{|z|^{n+s}} \, dz , \label{eq: f cont in t 1} 
\end{align}
where we have computed the integral in coordinates $\varphi^{-1}$ and we have used Lemma \ref{loccomparability} to estimate the kernel $K_s(p,\varphi(z))=K_s(\varphi(0),\varphi(z))$. By the very definition of $S$, for $0<R<\varrho$, it follows that $\mathcal{H}^{n-1}(\varphi^{-1}(S) \cap \partial \B_R) \le CR^{n-2} \cdot C|t-\sigma| R^2 = C|t-\sigma|R^n$. Hence, by polar coordinates
\begin{align*}
    \int_{ \B_{\varrho} \setminus \B_{r/2} }  \frac{\chi_{\varphi^{-1}(S)}(z)}{|z|^{n+s}} \, dz & = \int_{r/2}^\varrho \frac{1}{R^{n+s}} \mathcal{H}^{n-1}(\varphi^{-1}(S) \cap \partial \B_R) \, dR \\ & \le C|t-\sigma| \int_{r/2}^\varrho R^{-s} \, dR = \frac{C}{1-s}|t-\sigma| (\varrho^{1-s} - (r/2)^{1-s}) .
\end{align*}
Thus, letting $r \to 0^+$ in \eqref{eq: f cont in t 1} gives
\begin{equation*}
    |f(t)-f(\sigma)| \le C|t-\sigma| \varrho^{1-s} . 
\end{equation*}
In particular, $f$ is continuous and this concludes Step 1.

\vspace{0.3 cm}

Next, fix $t=t_\circ>0$ small and choose ``Fermi coordinates'' adapted to the hypersurface $\Gamma: = \varphi(\{x_n= t_\circ |x'|^2\})$ around $p$. More precisely: there exists a diffeomorphism $\psi : \B_{\varrho_1}\to W = \psi(\B_{\varrho_1})$,  with $\psi(0)=p$ and $W\subset V$  open neighborhood of $p$, such that for all $x\in \B_{\varrho_1}$,
\[
d(\psi(x) , \Gamma) = 
\begin{cases}
 x_n \quad &\mbox{if }x_n\ge 0\\
-x_n \quad &\mbox{if }x_n\le 0
\end{cases} 
\]
and $\psi(\B_{\varrho_1}^+)=W\cap \varphi( \B_{\varrho_\circ} \cap \{ x_n>t|x'|^2\})$.

Moreover, since  $ G:= \psi( \B_{\varrho_1}^+ )\cup (E\setminus W)$ contains $F_t$ we have 
\begin{equation}\label{whothiowhw2}
\lim_{r\downarrow 0} \int_{M\setminus B_r(p)} (\chi_G -\chi_{G^c})(q) K_s(p,q) \,dV_q\ge  \delta >0\, .
\end{equation}
Also, by construction we have
\begin{equation*}
\psi( \B_{\varrho_1} \cap \{x_n \ge -c|x'|^2\})\subset E,
\end{equation*}
where $c>0$ depends on $t_\circ$.

\vspace{3pt}
\textbf{Step 2.} We now perform a key computation in coordinates. 
Let us now choose a smooth cutoff function $\eta: \R^n \to \R_+$ satisfying $\chi_{\B_1}\le \eta \le \chi_{\B_2}$ and put:

\begin{equation}\label{whtiohwiohw}
\eta_\varrho(x)  := \eta(x/\varrho) \quad \mbox{and}\quad 
\bar \eta_\varrho : = \eta_\varrho \circ\psi^{-1}.
\end{equation}

Let $K(x,y) = K_s(\varphi(x), \varphi(y))$ be the expression in coordinates $\psi^{-1}$ of the kernel $K_s(p,q)$ for $p, q\in W$, that is, for $x,y\in \B_1$. Let $g_{ij}: \B_{\varrho_1}\to \R^{n^2}$ denote the components of the metric in the coordinates $\psi^{-1}$.  Since $\psi^{-1}$ are Fermi coordinates, we have 
\begin{equation}\label{wjhtiohwohwioh}
g_{ni} = g_{in} = \delta_{ni}, \quad 1\le i \le n . 
\end{equation}
This will be crucially used later.

Fix $\varrho\in (0,\varrho_1/2)$ small to be chosen later. By \eqref{whothiowhw2}, we have for $G_\varrho:= \psi(\B_{\varrho}^+) \cup (E\setminus \psi(\B_\varrho))$ and $H : = \{x_n>0\}\subset \R^n$ 
\[
\begin{split}
\lim_{r\downarrow 0}  \int_{\B_{2\varrho}\setminus \B_r} (\chi_H -\chi_{H^c})(y) K(0,y)\eta_\varrho(y) \sqrt{|g|}(y)\,dy &= \lim_{r\downarrow 0} \int_{M\setminus B_r(p)} (\chi_{G_\varrho} -\chi_{G_\varrho^c})(q) K_s(p,q)\bar \eta_\varrho(q) \,dV_q
\\
&\ge \delta - \int_{M\setminus \psi(\B_{\varrho})} (\chi_E -\chi_{E^c})(q) K_s(p,q)( 1-\bar \eta_\varrho)(q) \,dV_q.
\end{split}
\]
Notice that $\B_r(0)$ is not the same as $\varphi^{-1}(B_r(p))$, however as it will become clear from the proof below, the limits as $r\downarrow 0$ of the corresponding integrals give the same value.

Let us also write 
\[
K(x,y)\sqrt{|g|}(y) = \frac{\alpha_{n,s}}{|A(x)(x-y)|^{n+s}} \sqrt{|g|}(x) + \widehat K(x,y),
\]
where  $A(x)$ is the nonnegative definite symmetric square root of the matrix $(g_{ij}(x))$. Notice that, thanks to  \eqref{wjhtiohwohwioh}, we have $A_{ni}(x) = A_{in}(x) = \delta_{ni}$ for all $1\le i\le n$.
Also, by Proposition \ref{prop:kern1} the kernel $ \widehat  K(x,y)$ is {\rm not} singular, in the sense that
\[ \big| \widehat  K(x,y)\big| \le C ( 1+ |x-y|^{-n -s+1}).\]

We thus have
\begin{equation}\label{wgeuitguiwgw}
\begin{split}
\alpha_{n,s} \lim_{r\downarrow 0}  \int_{\B_{\varrho_1}\setminus \B_r} \frac{(\chi_H -\chi_{H^c})\eta_\varrho (y)}{ |A(0)(0-y)|^{n+s}}  \sqrt{|g|}(0) dy  &\ge \delta - \int_{M\setminus \psi(\B_{\varrho})} (\chi_E -\chi_{E^c})(q) K_s(p,q)( 1-\bar \eta_\varrho)(q) \,dV_q \\
& \hspace{25mm} - \int_{\B_{\varrho_1}} (\chi_H -\chi_{H^c})(y)  \widehat  K(0,y) \eta_\varrho(y) \, dy
\end{split}
\end{equation}

Let us now recall the assumption that $u_{\ep_j} \to u_0 = \chi_E-\chi_{E^c}$,
and let us define for $x$ in a neighbourhood of $0$
\[
f_j(x) = -\int_{M\setminus \psi(\B_{\varrho})} u_{\ep_j}(q) K(\psi(x),q)( 1-\bar \eta_\varrho)(q) \,dV_q - \int_{ \B_{2\varrho}} (u_{\ep_j}\circ \psi)(y)  \widehat  K(x ,y)\eta_\varrho(y)  \, dy
\]
and 
\[
f_\infty (x) = -\int_{M\setminus \psi( \B_{\varrho})} (\chi_E-\chi_{E^c})(q) K(\psi(x),q)( 1-\bar \eta_\varrho)(q) \,dV_q - \int_{\B_{2\varrho}} (\chi_E -\chi_{E^c})(\psi(y))  \widehat  K(x ,y) \eta_\varrho(y) \, dy
\]
Define also
\[
I(\varrho) := \int_{ \B_{2\varrho}} (\chi_E -\chi_{E^c})(\psi(y)) \widehat K(0,y) \eta_\varrho(y) \, dy - \int_{ \B_{2\varrho}} (\chi_H -\chi_{H^c})(y)  \widehat K(0 ,y) \eta_\varrho(y) \, dy
\]

Fixing $\varrho>0$ small enough, we will have $|I(\varrho)|\le \delta /4$.
Then, is not difficult to show (using the kernel bounds of Proposition \ref{prop:kern1} and of Lemma \ref{loccomparability}) that $f_j(x)\to f_\infty(x)$ uniformly for all $x$ in a neighborhood of $0$, and that $f_\infty$ is continuous in a neighborhood of $0$. 
As a consequence, we have $|f_j(x)-f_\infty (0)| < \delta/4$ for all $x\in \B_{r_\circ}(0)$ and $j \ge j_\circ$, for some $j_\circ$.

On the other hand, recall that  $(-\Delta)^s u_{\ep_j} +\ep_j^{-s} W'(u_{\ep_j}) =0$ in $M$. Hence, in particular
\[
\lim_{r\downarrow 0} \int_{M\setminus B_r(\psi(x))} \big(u_{\ep_j}(\psi(x)) -u_{\ep_j}(q)\big) K_s(p,q) \,dV_q + \ep_j^{-s}W'(u_{\ep_j}(\psi(x))) =0
\]
for all $x\in \B_{r_\circ}(0)$. Proceeding similarly the previous equation can be rewritten as 
\begin{equation}\label{wjtiohwoiwh1}
\alpha_{n,s}  \int_{ \B_{\varrho}} \frac{\big((u_{\ep_j}\circ \psi)(x)- (u_{\ep_j}\circ \psi)(y)\big) \eta_\varrho(y)}{|A(x)(x-y)|^{n+s}}  \sqrt{|g|}(x) \, dy  + \ep_j^{-s}W'(u_{\ep_j}(\psi(x))) = f_j(x) \,.
\end{equation}
Notice also that \eqref{wgeuitguiwgw} can be rewritten as
\[
\begin{split}
\alpha_{n,s} \lim_{r\downarrow 0}  \int_{ \B_{2\varrho}\setminus \B_r} \frac{(\chi_H -\chi_{H^c}) \eta_\varrho (y)}{ |A(0)(0-y)|^{n+s}}  \sqrt{|g|}(0) \, dy  &\ge \delta + f_\infty(0) + I(\varrho)\, .
\end{split}
\]

We now define $v_{\ep,\tau}(x) = v_\circ\big( \ep^{-1}(x_n-\tau)\big)$, where $v_\circ: \R \to (-1,1)$ is the function from  Lemma \ref{lem:hoihwrhiohtiow}. In view of Remark \ref{rem:shioghsoihs}, we have for $x\in \B_{r_\circ}(0)$,  $j$ large, and $|\tau|$ sufficiently small,
\[
\begin{split}
\alpha_{n,s} \, \int_{ \B_{2\varrho}} \frac{\big(v_{\ep_j,\tau}(x)- v_{\ep_j,\tau}(y)\big)\eta_\varrho (y)}{ |A(x)(x-y)|^{n+s}}  \sqrt{|g|}(x) \, dy  + \ep_j^{-s} W'(v_{\ep_j,\tau}) &\le \frac{\delta}{4}.
\end{split}
\]

This implies that  whenever $x\in \B_{r_\circ}$, $j$ sufficiently large, and $|\tau|$ sufficiently small
\begin{equation}\label{wjtiohwoiwh2}
\begin{split}
\alpha_{n,s} \, \int_{ \B_{2\varrho}} \frac{\big(v_{\ep_j,\tau}(x)- v_{\ep_j,\tau}(y)\big)\eta_\varrho (y)}{ |A(x)(x-y)|^{n+s}}  \sqrt{|g|}(x) \, dy  + \ep_j^{-s} W'(v_{\ep_j,\tau}) &\le f_j(x) - \frac{\delta}{4}.
\end{split}
\end{equation}
In other words, we have shown that $v_{\ep_j,\tau}$ is a strict subsolution of \eqref{wjtiohwoiwh1}.

\vspace{3pt}
\textbf{Step 3.} We now reach the desired contradiction. 
Fix now $\theta\in \left( 0, \tfrac{1}{100} \right)$ sufficiently small (to be chosen) and let 
\[
\xi_\theta (t) : = 
\begin{cases}
-1+\theta\quad & \mbox{if } t\in [-1,-1+\theta] ,\\
t& \mbox{if } t\in [-1+\theta,1-\theta] , \\
1-\theta\quad & \mbox{if } t\in [1-\theta,1].
\end{cases} 
\]

By the Hausdorff convergence of the level sets of $u_{\ep_j}$ which we have proved in Step 3 at page \pageref{asdfvasdf}, for any $t\in [-1+\theta, 1-\theta]$ the set $\{x\in \B_{2\varrho} : (u_{\ep_j}\circ \varphi) \ge t\}$ converges in Hausdorff distance towards  $\psi^{-1}(E)\supset \{x\in \B_{2\varrho} \,\, : \,\,  x_n \le -c|x'|^2\}$. 
 
Hence,  for every fixed $\tau >0$ we have, for all $j$ sufficiently large,
\begin{equation}\label{wtiowhiopwjh}
\xi_\theta\circ u_{\ep_j}\circ\psi \ge \xi_\theta\circ v_{\ep_j,\tau}  \quad \mbox{ in }  \overline{\B_{2\varrho}} \,.
\end{equation}
Let us define
\[
\tau_j : = \min \big\{\tau\in \R \ | \ \mbox{\eqref{wtiowhiopwjh} holds for } j \,\big\}\, .
\]
Notice that by definition of $\tau_j$ there is $x_j\in \overline{ \B_{2\varrho}} \cap \{|u_{\ep_j}|\le 1-\theta\}\cap\{|v_{\ep_j,\tau}|\le 1-\theta\}$. By the previous Hausdorff convergence property of level sets, it must be $x_j\to 0$ and $\tau_j \to 0$ as $j\to \infty$. 
 
\vsp
Let us show that, if $\theta$ is chosen sufficiently small, we have 
\begin{equation}\label{wjthiowhoiughw246}
u_{\ep_j}\circ\psi\ge v_{\ep_j,\tau_j}\quad \mbox{in } \B_{r_\circ/2}\, .
\end{equation}

Indeed, thanks to \eqref{wjtiohwoiwh1}-\eqref{wjtiohwoiwh2} the difference 
\[w: = u_{\ep_j}\circ\psi- v_{\ep_j,\tau_j}\] 
satisfies
\begin{equation}\label{whiowhoiwh1}
\mathcal L w(x) : = \alpha_{n,s}  \int_{ \B_{\varrho}} \frac{\big(w(x)- w(y)\big) \eta_\varrho(y)}{|A(x)(x-y)|^{n+s}}  \sqrt{|g|}(x) \, dy  \ge \frac{\delta}{4} +  \ep_j^{-s}\big(W'(v_{\ep_j,\tau_j})- W'(u_{\ep_j}\circ\psi)\big)(x) \quad \mbox{in } \B_{r_\circ}.
\end{equation}
Notice that  since \eqref{wtiowhiopwjh} holds for $\tau=\tau_j$  we have $w = v_{\ep_j,\tau_j} -(u_{\ep_j}\circ\varphi) \ge -\theta$ in $\B_{2\varrho}$.

\vsp
Assume now by contradiction that $\inf_{\B_{r_\circ/2}} w <0$. 
Recall \eqref{whtiohwiohw} and define  
\[
\overline \eta_t =  -\theta + t\eta_{r_\circ/2} \,. 
\]
and let $t_*\in [0,\theta)$ be the supremum of the $t\ge 0$ such that $w\ge \overline \eta_t$ in $\B_{2\varrho}$. By construction there exists $x_* \in \B_{r_\circ}$ such that 
\[
(w- \overline \eta_{t_*}) (x_*) =0 \quad \mbox{while}\quad    w- \overline \eta_{t_*}\ge 0 \quad \mbox{in } \B_{2\varrho}. 
\]

Now evaluating the integro-differential operator $\mathcal L$ (whose kernel is supported in $B_{2\varrho}$; see \eqref{whiowhoiwh1}) at the point $x_*$ we obtain 
\[
C\theta \ge \mathcal L \overline \eta_{t_*}(x_*) \ge \mathcal L w(x_\circ) \ge \frac{\delta}{4} +  \ep_j^{-s}\big(W'(v_{\ep_j,\tau_j})- W'(u_{\ep_j}\circ\varphi)\big)(x_\circ) \ge \frac{\delta}{4},
\]
Notice that  $W''>0$ in the interval $[u_{\ep_j}\circ\varphi(x_\circ), v_{\ep_j,\tau_j}(x_\circ)]$ because \eqref{wtiowhiopwjh} holds for $\tau=\tau_j$, and hence either $u_{\ep_j}\circ\varphi(x_\circ)\ge 1-\theta$ or $v_{\ep_j,\tau_j}(x_\circ)\le -1+\theta$. Therefore, choosing $\theta>0$ sufficiently small  so that $C\theta < \delta/4$ we reach a contradiction. Hence, we have proved that $w\ge 0$ and \eqref{wjthiowhoiughw246} holds.

\vsp
Finally, take $j$ large so that $x_j\in \B_{r_\circ/4}$ (recall that $x_j\to 0$ as $j\to \infty$).
Using that $w\ge 0$ in $ \B_{r_\circ/2}$, $w(x_j)=0$, and $w\ge -\theta$ in $ \B_{2\varrho}\setminus \B_{r_\circ}$ and evaluating $\mathcal Lw$ at the point $x_j\in \B_{r_\circ/4}$ we obtain, similarly as before
\[
C(r_\circ) \theta \ge - \alpha_{n,s}  \int_{ \B_{\varrho}} \frac{w(y) \eta_\varrho(y)}{|A(x_j)(x_j-y)|^{n+s}}  \sqrt{|g|}(x_j) \, dy  = \mathcal L w(x_j) \ge \frac{\delta}{4}.
\]
Choosing $\theta>0$ sufficiently small, we obtain a contradiction, and this completes the proof.
\end{proof}

Theorem \ref{StrongConv} motivates the definition of $\mathcal{A}_m(M)$ given in the introduction, see Definition \ref{AClimits}.

\vsp
The ``surfaces" $\Sigma$ belonging to the class $\mathcal{A}_m(M)$ enjoy some properties additionally to those already described in Theorem \ref{StrongConv} and Proposition \ref{prop:viscosity}. We record them in the remark below.

\begin{remark}\label{remprop}
    Every $\Sigma = \partial E \in \mathcal{A}_m(M)$ also satisfies that if ${\rm FA}_2(M,g,R,p,\varphi)$ is satisfied, then the following hold:
\begin{itemize}
    \item[$(1)$] \textbf{BV and energy estimate.} For some $C=C (n,s,m) > 0 $ there holds
\begin{equation*}
    \textnormal{Per}(E;B_{R/2}(p)) \leq CR^{n-1} \quad \text{and}\quad  \textnormal{Per}_s(E;B_{R/2}(p)) \leq CR^{n-s}.
\end{equation*}
    \item[$(2)$] \textbf{Density estimate.} For some positive constant $\omega_0$, which depends only on $n$, $s$ and $m$, we have that if $R^{-n}|E \cap B_R(p)| \leq \omega_0$ then $|E \cap B_{R/2}(p)|=0$.
\end{itemize}

   Indeed, by Definition \ref{AClimits} of $\mathcal{A}_m(M)$ we can find a sequence $u_{\ep_j}$, made of A-C solutions with Morse index $\leq m$ and parameters $\ep_j\to 0$, converging to $E$ in $L^1(M)$, and also in $H^{s/2}(M)$ thanks to Theorem \ref{StrongConv}. Then, property $(1)$ follows from the lower semicontinuity of the $BV$ norm under $L^1$ convergence and the convergence of Sobolev energies under strong $H^{s/2}$ convergence, together with the fact that the $u_{\ep_j}$ satisfy uniform $BV$ and Sobolev estimates themselves by Theorems \ref{BVEst} and \ref{A-C-energy-est}. Similarly, property $(2)$ follows from the $L^1$ convergence and the density estimates of Proposition \ref{DensEst} satisfied by the $u_{\ep_j}$ themselves.

\end{remark}

\subsection{The Yau conjecture for nonlocal minimal surfaces -- Proof of Theorem \ref{FracYau3}}\label{YauSection}

We can now combine the existence and convergence results in the previous sections to prove the Yau conjecture for nonlocal minimal surfaces.
    
\begin{proof}[Proof of Theorem \ref{FracYau3}]
Fix $\p\in \mathbb{N}$. Theorem \ref{Existence1} gives the existence, for all $\ep\in(0,\ep_\p)$, of a solution $u_{\ep,\p}$ to the fractional Allen-Cahn equation with Morse index $m(u_{\ep,\p}) \leq \p$ and energy bounds
\begin{equation}\label{ACgrowth}
    C^{-1} \p^{s/n} \leq (1-s)\mathcal{E}_M^\ep(u_{\ep,\p}) \leq C \p^{s/n}.
\end{equation}
Thanks to the convergence result in Theorem \ref{StrongConv}, we can find a subsequence $\{\ep_j\}_{j }$ such that the $u_{\ep_j, \p}$ converge in $H^{s/2}(M)$ to a limit function
\begin{equation*}
u_{0,\p}=\chi_{E^\p}-\chi_{ M \setminus E^\p }  ,   
\end{equation*}
where $\partial E^\p$ is an $s$-minimal surface and $\partial E^\p\in\mathcal{A}_\p(M)$ by definition. Moreover, by \eqref{ACgrowth} and the strong convergence of the Allen-Cahn energies stated in Theorem \ref{StrongConv}, we deduce that the fractional perimeter of $E^\p$ satisfies the bounds
\begin{equation*}
    C^{-1} \p^{s/n} \leq (1-s)\textnormal{Per}_s(E^\p) \leq C \p^{s/n}.
\end{equation*}
In particular, the fractional perimeter of the $E^\p$ goes to infinity as $\p\to\infty$, thus we conclude that the family $\{E^\p \}_{\p \in \N}$ is infinite. 
\end{proof}

\begin{remark}
    We emphasise that, unlike in \cite{GG1} or \cite{MN}, due to the {\bf strong} convergence as $\ep\to 0$ there is no multiplicity phenomenon. We have used the following in the previous proof: Consider two sets $E^\p$ and $E^{\p'}$ as in the proof of Theorem \ref{FracYau3} above, corresponding to respective limits of the sequences $u_{\ep,\p}$ and $u_{\ep,\p'}$ as $\ep \to 0$, and assume that 
\begin{equation*}
    \lim_{\ep \to 0}(1-s)\mathcal{E}_M^\ep(u_{\ep,\p}) \neq \lim_{\ep \to 0}(1-s)\mathcal{E}_M^\ep(u_{\ep,{\p'}})\,; 
\end{equation*}
then, they are necessarily distinct sets. Hence, their boundaries correspond to geometrically distinct $s$-minimal surfaces. We note that this does not prevent, however, that $E^\p = E^{\p+1}$ for some value of $\p$, since it could be that $(1-s)\textnormal{Per}_s(E^\p) = (1-s)\textnormal{Per}_s(E^{\p+1})$; nevertheless, as $\p\to\infty$ 
\end{remark}

\section{Regularity and rigidity results}

\subsection{Blow-up procedure}\label{BlowUpSection}
The goal of this subsection is to explicitly show how to perform blow-ups of (sequences of) $s$-minimal surfaces around points with flatness assumptions, proving strong convergence results for the blow-up sequence to a Euclidean limit surface in a manner similar to Section \ref{ConvSection}.
\begin{definition}[\textbf{Blow-up sequence}]\label{blowupdef}
Let $(M_j,g^{(j)})$ be a sequence of closed manifolds of dimension $n$, and let $p_j\in M_j$ be points such that $M_j$ satisfies the flatness assumption ${\rm FA}_3(M_j,g^{(j)},1,p_j,\varphi_j)$. Suppose in addition that $g_{kl}^{(j)}(0) = \delta_{kl}$, i.e. that the metric of $M_j$ with respect to the chart $\varphi_j^{-1}$ at the point $0$ is the Euclidean metric.

For each $j$, let $\partial E_j$ be an $s$-minimal surface in $M_j$, satisfying uniform BV estimates in the sense that there is some $C_0$ independent of $j$ such that
$$
\text{Per}(\varphi_j^{-1} (E_j); B_{r}(x))\leq C_0 r^{n-1} \quad \mbox{for all } x\in \B_{1/2} \mbox{ and } r\in (0,1/4)\, ,
$$
where we put $\varphi_j^{-1} (E_j):=\{y\in \B_1: \varphi_j(y)\in E_j\}$.\\
Given $r_j\searrow 0$, a sequence of subsets of $\R^n$ of the form
\[
F_j : =  \frac{1}{r_j}\varphi_j^{-1} (E_j)  \subset \B_{1/r_j} \subset \R^n
\]
(for some $M_j, p_j, E_j$ as above) is called a blow-up sequence.
\end{definition}

\begin{remark}\label{rescblowdef}
    $F_j$ is a blow-up sequence if and only if there exist $(\widehat M_j,\widehat g^{(j)})$, $\widehat p_j\in\widehat M_j$, and $R_j\nearrow\infty$ such that
    \begin{itemize}
        \item $\textnormal{Per}(F_j;B_r(x))\leq C_0 r^{n-1} \quad \mbox{for all } x\in \B_{R_j/2} \mbox{ and } r\in (0,R_j/4)$;
        \item ${\rm FA}_3(\widehat M_j,\widehat g^{(j)},R_j,\widehat p_j,\widehat\varphi_j)$ holds and $\widehat g_{kl}^{(j)}(0)=\delta_{kl}$, where $\widehat g_{kl}^{(j)}=\widehat g^{(j)}((\widehat\varphi_j)_*(e_k),(\widehat\varphi_j)_*(e_l))$ denotes the metric in coordinates;
        \item For each $j$ there is an $s$-minimal surface $\partial \widehat E_j$ in $\widehat M_j$ such that $F_j=\widehat\varphi_j^{-1}(E_j)$.
    \end{itemize}
\end{remark}
\begin{proof}[Proof of the remark]
    This follows from putting $\widehat M_j= M_j$, $\widehat p_j = p_j$, $\widehat g^{(j)}=\frac{1}{r_j^2}g^{(j)}$ and $R_j=\frac{1}{r_j}$ in Definition \ref{blowupdef} and considering the scaling properties stated in Remark \ref{flatscalingrmk}.
\end{proof}

We record some auxiliary results. The notation $K_{\widehat M_j}$ will be used instead of $K_s$ when we want to explicit which manifold the kernel $K_s$ is being considered on.
\begin{proposition}\label{eucproperties1}
    Let $F_j\subset\R^n$ be a blow-up sequence, with 
    associated $(\widehat{M}_j,\widehat g^{(j)})$, $\widehat E_j\subset \widehat M_j$ and $R_j\to \infty$ as in Remark \ref{rescblowdef}. The following hold:
    \begin{itemize}
        \item[\textit{(i)}] The components $\widehat g_{kl}^{(j)}$ of the metric of $\widehat M_j$ (using the chart parametrization $\widehat \varphi_j$) converge locally uniformly to the Euclidean ones, in the sense that given $R_0>0$,
        $$\sup_{x\in\B_{R_0}} \big|\widehat g_{kl}^{(j)}(x)-\delta_{kl} \big|\to 0 \quad \mbox{as } j\to\infty.
        $$
        \item[\textit{(ii)}] The kernel $K_{\widehat M_j}$ converges locally uniformly to the Euclidean one, in the sense that given $R_0>0$,
        $$\sup_{(x,y)\in\B_{R_0}\times\B_{R_0}}\left|\frac{K_{\widehat M_j}(\widehat \varphi_j(x),\widehat\varphi_j(y))}{\frac{\alpha_{n,s}}{|x-y|^{n+s}}}-1\right|\to 0 \quad \mbox{as } j\to\infty.
        $$ 

    \end{itemize} 
\end{proposition}
\begin{proof} The first part follows from the definition of the flatness assumptions and the fact that  $R_j\to\infty$.

\vsp
As for part \textit{(ii)}, it is a consequence of Proposition \ref{prop:kern1}. Precisely, it follows from putting $R=R_j$ and $z=y-x$ in \eqref{remaining0} of Proposition \ref{prop:kern1}.
\end{proof}

\begin{lemma}\label{lemrescen}
    Let $F_j\subset\R^n$ be a blow-up sequence, with 
    associated $(\widehat{M}_j,\widehat g^{(j)})$, $\widehat E_j\subset \widehat M_j$ and $R_j\to \infty$ as in Remark \ref{rescblowdef}. Put
    $$
    K_j(x,y):=K_{\widehat M_j}(\widehat\varphi_j(x),\widehat\varphi_j(y))
    $$
    and (for a fixed $\rho<R_j/4$)
    $$
    \textnormal{Per}_s^{(j)}(F_j;\B_\rho):=\frac{1}{4}\iint_{(B_{R_j}\times B_{R_j})\setminus (B_\rho^c\times B_\rho^c)} |u_j(x)-u_j(y)|^2 K_j(x,y)\sqrt{g^{(j)}}(x)\sqrt{g^{(j)}}(y)\,dx\,dy\, ,
    $$
    where $u_j:=\chi_{F_j}-\chi_{F_j^c}$.\\
    Given a vector field $X\in C_c^\infty(B_\rho;\R^n)$, define $X_j:=(\widehat \varphi_j)_* X$ and extend it by zero to a vector field on $\widehat M_j$. The following hold:
    \begin{itemize}
        \item[$(1)$] Let $ 0\le\ell\le 3$. Then
        $$\bigg|\frac{d^\ell}{dt^\ell} \Big(\textnormal{Per}_s^{\widehat M_j}(\psi_{X_j}^t(\widehat E_j);\widehat\varphi_j(\B_\rho))-\textnormal{Per}_s^{(j)}(\psi_{X}^t(F_j);\B_\rho)\Big)\bigg|\leq \frac{C_X}{R_j^s}\, .$$

        \item[$(2)$] If $\chi_{F_j}\to\chi_{F}$ in $H^{s/2}_{\rm loc}(\R^n)$, then for $0\le\ell\le 2$
        $$
        \frac{d^\ell}{dt^\ell}\textnormal{Per}_s^{(j)}(\psi_{X}^t(F_j);\B_\rho)\to \frac{d^\ell}{dt^\ell}\textnormal{Per}_s^{\R^n}(\psi_{X}^t(F);\B_\rho)\, .
        $$
    \end{itemize}
\end{lemma}
\begin{proof}
    We begin by proving (1). Let $v_j^t:=\chi_{\psi_{X_j}^t(\widehat E_j)}-\chi_{\psi_{X_j}^t(\widehat E_j^c)}$ and $u_j^t:=\chi_{\psi_{X}^t(F_j)}-\chi_{\psi_{X}^t(F_j^c)}$.\\
    By splitting the domain of the corresponding integral and then passing to coordinates, we can write
\begin{align*}
    \textnormal{Per}_s^{\widehat M_j}(\psi_{X_j}^t(\widehat E_j);\widehat\varphi_j(\B_\rho))&= \frac{1}{4}\iint_{(\widehat\varphi_j(\B_{R_j})\times \widehat\varphi_j(\B_{R_j})) \setminus (\widehat\varphi_j(\B_{\rho})^c\times \widehat\varphi_j(\B_{\rho})^c)} |v_j^t(p)-v_j^t(q)|^2\, K_{\widehat M_j}(p,q)\,dV_p\,dV_q \\
    &\hspace{0.6 cm} + \frac{1}{2}\iint_{\widehat\varphi_j(\B_{\rho})\times \widehat\varphi_j(\B_{R_j})^c} |v_j^t(p)-v_j^t(q)|^2\, K_{\widehat M_j}(p,q)\,dV_p\,dV_q\\
    &=\frac{1}{4}\iint_{(\B_{R_j}\times \B_{R_j}) \setminus (\B_{\rho}^c\times \B_{\rho}^c)}|u_j^t(x)-u_j^t(y)|^2 K_j(x,y)\sqrt{g^{(j)}(x))}\sqrt{g^{(j)}(y))}\,dx\,dy \\
    &\hspace{0.6 cm} + \frac{1}{2}\iint_{\widehat\varphi_j(\B_{\rho})\times \widehat\varphi_j(\B_{R_j})^c} |v_j^t(p)-v_j^t(q)|^2\, K_{\widehat M_j}(p,q)\,dV_p\,dV_q\\
    &=\textnormal{Per}_s^{(j)}(\psi_X^t(F_j);\B_\rho)+ \frac{1}{2}\iint_{\widehat\varphi_j(\B_{\rho})\times \widehat\varphi_j(\B_{R_j})^c} |v_j^t(p)-v_j^t(q)|^2\, K_{\widehat M_j}(p,q)\,dV_p\,dV_q\, .
\end{align*}
From this computation, changing variables with the flow as in \eqref{eqcvthm} and then passing to coordinates in the first variable we can compute

\begin{align*}
    \bigg|\frac{d^\ell}{dt^\ell} \Big(\textnormal{Per}_s^{\widehat M_j}&(\psi_{X_j}^t(\widehat E_j);\widehat\varphi_j(\B_\rho))-\textnormal{Per}_s^{(j)}(\psi_{X}^t(F_j);\B_\rho)\Big)\bigg|=\\
    &=\frac{1}{2}\bigg|\frac{d^\ell}{dt^\ell}\iint_{\widehat\varphi_j(\B_{\rho})\times \widehat\varphi_j(\B_{R_j})^c} |v_j^t(p)-v_j^t(q)|^2\, K_{\widehat M_j}(p,q)\,dV_p\,dV_q\bigg|\\
    &=\frac{1}{2}\bigg|\iint_{\widehat\varphi_j(\B_{\rho})\times \widehat\varphi_j(\B_{R_j})^c} |v_j(p)-v_j(q)|^2\, \frac{d^\ell}{d t^\ell}\Big[K_{\widehat M_j}(\psi_{X_j}^{ t}(p),q) \,J_{ t}(p)\Big]\,dV_p\,dV_q\bigg|\\
    &\leq C\iint_{\B_{\rho}\times \widehat\varphi_j(\B_{R_j})^c} \bigg|\frac{d^\ell}{d t^\ell}\Big[K_{\widehat M_j}(\widehat\varphi_j(\psi_{X}^{t}(x)),q) \,J_{ t}(p)\Big]\bigg|\,dx\,dV_q\, .
\end{align*}
Bounding the derivatives in time of the Jacobian $J_{ t}(p)$ by a constant, and using \eqref{remaining3} with $R=R_j$ to bound the integral in $q$, we conclude the result in (1).

\vsp
To see (2), let $R>\rho$ and put $f_j(t):=\textnormal{Per}_s^{(j)}(\psi_X^t(F_j);\B_\rho)$. Changing variables with the flow (as above) and splitting the domain of the integral, for $R<R_j$ we can write 

\begin{align*}
    &\frac{d^\ell}{dt^\ell}\textnormal{Per}_s^{(j)}(\psi_X^t(F_j);\B_\rho) \\
    &=\frac{1}{4}\iint_{(\B_{R}\times \B_{R}) \setminus (\B_{\rho}^c\times \B_{\rho}^c)}|u_j(x)-u_j(y)|^2 \frac{d^\ell}{dt^\ell}\Big[K_j(\psi_X^t(x),\psi_X^t(y))\sqrt{g^{(j)}}(\psi_X^t(x))\sqrt{g^{(j)}}(\psi_X^t(y))J_t(x)J_t(y)\Big]\,dx\,dy\\
    &\hspace{0.6cm} +\frac{1}{2}\iint_{\B_\rho\times(\widehat\varphi_j(\B_{R_j})\setminus \widehat\varphi_j(\B_{R}))}|u_j(x)-u_j(\varphi^{-1}(q))|^2\frac{d^\ell}{dt^\ell}\Big[K_{\widehat M_j}(\widehat\varphi_j(\psi_X^t(x)),q)\sqrt{g^{(j)}}(\psi_X^t(x))J_t(x)\Big]\,dx\,dV_q\, .
\end{align*}

Let $0\leq \ell \leq 3$. Thanks to the flatness assumptions and \eqref{remaining3} of Proposition \ref{prop:kern1}, we can bound

\begin{align}
\bigg|\iint_{\B_\rho\times(\widehat\varphi_j(\B_{R_j})\setminus \widehat\varphi_j(\B_{R}))} & |u_j(x)-u_j(\varphi^{-1}(q))|^2\frac{d^\ell}{dt^\ell}\Big[K_{\widehat M_j}(\widehat\varphi_j(\psi_X^t(x)),q)\sqrt{g^{(j)}}(\psi_X^t(x))J_t(x)\Big]\,dx\,dV_q\bigg| \nonumber \\
&\leq C\iint_{\B_\rho\times(\widehat M_j\setminus \widehat\varphi_j(\B_{R}))}\bigg|\frac{d^\ell}{dt^\ell}\Big[K_{\widehat M_j}(\widehat\varphi_j(\psi_X^t(x)),q)\sqrt{g^{(j)}}(\psi_X^t(x))J_t(x)\Big]\bigg|\,dx\,dV_q \nonumber \\
&\leq \frac{C}{R^s} \, . \label{iobound}
\end{align}

On the other hand, by the flatness assumptions and \eqref{cdsacsdc2} in Proposition \ref{sdfgsdrgs} we have that, for $t\in(-T,T)$ and $j$ large enough so that $R<R_j/4$,
{\fontsize{10pt}{10pt}\selectfont
\begin{align*}
     \iint_{(\B_{R}\times \B_{R}) \setminus (\B_{\rho}^c\times \B_{\rho}^c)} & |u_j(x)-u_j(y)|^2 \bigg|\frac{d^\ell}{dt^\ell}\Big[K_j(\psi_X^t(x),\psi_X^t(y))\sqrt{g^{(j)}}(\psi_X^t(x))\sqrt{g^{(j)}}(\psi_X^t(y))J_t(x)J_t(y)\Big]\bigg|\,dx\,dy \\
    &\leq C_T\iint_{\B_{R}\times \B_{R}}|u_j(x)-u_j(y)|^2 \frac{\alpha_{n,s}}{|x-y|^{n+s}}\,dx\,dy\\
    &\leq C_T\, ,
\end{align*}
}
where in the last line we combined the fact that $F_j$ has bounded classical perimeter in $B_{R_j/4}$ with the interpolation result in Proposition \ref{interprop}.

\vsp
This shows that the functions $\frac{d^\ell}{dt^\ell}f_j(t)$ are locally uniformly bounded for $0\leq \ell \leq 3$; in particular, for $0\leq \ell \leq 2$ we deduce that the $\frac{d^\ell}{dt^\ell}f_j(t)$ are locally uniformly bounded and moreover have a uniform modulus of continuity, thus by Arzel\`a-Ascoli they subsequentially converge locally uniformly. By standard single-variable calculus, to conclude our desired result it then suffices to show that $f_j(t)$ converges pointwise to $g(t):=\textnormal{Per}_s^{\R^n}(\psi_X^t(F);\B_\rho)$, since then the first two derivatives of the limit function $g(t)$ will be the limits of the derivatives of the $f_j(t)$. We shall now prove the pointwise convergence result.

\vsp
Denote $u^t:=\chi_{\psi_X^t(F)}-\chi_{\psi_X^t(F^c)}$. We can then write
\begin{align*}
    g(t)&=\textnormal{Per}_s^{\R^n}(\psi_X^t(F);\B_\rho)\\
    &=\frac{1}{4}\iint_{(\B_{R}\times \B_{R}) \setminus (\B_{\rho}^c\times \B_{\rho}^c)}|u(x)-u(y)|^2\frac{\alpha_{n,s}}{|\psi_X^t(x)-\psi_X^t(y)|^{n+s}}J_t(x)J_t(y)\,dx\,dy\\
    &\hspace{0.6cm} +\frac{1}{2}\iint_{\B_\rho\times\B_{R}^c}|u(x)-u(y)|^2\frac{\alpha_{n,s}}{|\psi_X^t(x)-y|^{n+s}}J_t(x)\,dx\,dy\, .
\end{align*}

Clearly
$$\iint_{\B_\rho\times\B_{R}^c}|u(x)-u(y)|^2\frac{\alpha_{n,s}}{|\psi_X^t(x)-y|^{n+s}}J_t(x)\,dx\,dy\to 0 \quad \mbox{as } R\to\infty\, ,$$
since the integrand is absolutely integrable by \eqref{remaining3} in Proposition \ref{prop:kern1}. Together with \eqref{iobound}, given $\ep>0$, we deduce that there exists an $R>\rho$ (depending only on $\rho$ and $\ep$) such that the aforementioned terms are both smaller than $\ep/2$ for all $j$ large enough. From this fact and a simple triangle inequality, we find that 
\begin{align*}
    \bigg|&\textnormal{Per}_s^{(j)}(\psi_X^t(F_j);\B_\rho)-\textnormal{Per}_s^{\R^n}(\psi_X^t(F);\B_\rho)\bigg| \\
    & \leq \frac{1}{4}\iint_{(\B_{R}\times \B_{R}) \setminus (\B_{\rho}^c\times \B_{\rho}^c)}|u_j(x)-u_j(y)|^2 \cdot \\ & \hspace{2cm} \cdot  \left|K_j(\psi_X^t(x),\psi_X^t(y))\sqrt{g^{(j)}}(\psi_X^t(x))\sqrt{g^{(j)}}(\psi_X^t(y)) -\frac{\alpha_{n,s}}{|\psi_X^t(x)-\psi_X^t(y)|^{n+s}}\right|J_t(x)J_t(y)\,dx\,dy \\
    &\hspace{0.4cm} +\frac{1}{4}\Big|\iint_{(\B_{R}\times \B_{R}) \setminus (\B_{\rho}^c\times \B_{\rho}^c)}\Big(|u_j(x)-u_j(y)|^2-|u(x)-u(y)|^2\Big) \frac{\alpha_{n,s}}{|\psi_X^t(x)-\psi_X^t(y)|^{n+s}}J_t(x)J_t(y)\,dx\,dy\Big|\\
    &\hspace{0.4 cm}+\ep\, .
\end{align*}
Regarding the first term, thanks to Proposition \ref{eucproperties1} and \eqref{cdsacsdc2} in Proposition \ref{sdfgsdrgs} it can be bounded as follows:
\begin{align*}
    &\iint_{(\B_{R}\times \B_{R}) \setminus (\B_{\rho}^c\times \B_{\rho}^c)}|u_j(x)-u_j(y)|^2 \cdot \\ & \hspace{2cm} \cdot \left| K_j(\psi_X^t(x),\psi_X^t(y))\sqrt{g^{(j)}}(\psi_X^t(x))\sqrt{g^{(j)}}(\psi_X^t(y))-\frac{\alpha_{n,s}}{|\psi_X^t(x)-\psi_X^t(y)|^{n+s}}\right| J_t(x)J_t(y)\,dx\,dy\\
    &= \iint_{(\B_{R}\times \B_{R}) \setminus (\B_{\rho}^c\times \B_{\rho}^c)}|u_j(x)-u_j(y)|^2 \frac{\alpha_{n,s}}{|\psi_X^t(x)-\psi_X^t(y)|^{n+s}} \, \cdot \\ & \hspace{4cm} \cdot \left|\frac{K_j(\psi_X^t(x),\psi_X^t(y))\sqrt{g^{(j)}}(x)\sqrt{g^{(j)}}(y)}{\frac{\alpha_{n,s}}{|\psi_X^t(x)-\psi_X^t(y)|^{n+s}}}-1\right|J_t(x)J_t(y)\,dx\,dy\\
    &\leq o_j(1) \iint_{B_{R}\times \B_{R}}|u_j(x)-u_j(y)|^2 \frac{\alpha_{n,s}}{|\psi_X^t(x)-\psi_X^t(y)|^{n+s}}\,dx\,dy\\
    &\leq o_j(1) \, C_T\iint_{\B_{R}\times \B_{R}}|u_j(x)-u_j(y)|^2 \frac{\alpha_{n,s}}{|x-y|^{n+s}}\,dx\,dy\, ,
\end{align*}
where $o_j(1)\to 0$ as $j\to\infty$. This implies that the whole expression goes to zero since the factor $\iint_{\B_{R}\times \B_{R}}|u_j(x)-u_j(y)|^2 \frac{\alpha_{n,s}}{|x-y|^{n+s}}\,dx\,dy$ can be bounded by a constant independent of $j$: indeed, for $j$ large enough so that $R<R_j/4$, the $F_j$ satisfy uniform perimeter estimates in $\B_R$ by assumption (see Remark \ref{rescblowdef}), and thus also uniform fractional energy estimates by interpolation (see Proposition \ref{interprop}).

\vsp
As for the second term, we can write
\begin{align}
&\bigg|\iint_{(\B_{R}\times \B_{R}) \setminus (\B_{\rho}^c\times \B_{\rho}^c)}\Big(|u_j(x)-u_j(y)|^2-|u(x)-u(y)|^2\Big) \frac{\alpha_{n,s}}{|\psi_X^t(x)-\psi_X^t(y)|^{n+s}}J_t(x)J_t(y)\,dx\,dy\bigg|= \nonumber \\ \nonumber 
&=\bigg|\iint_{(\B_{R}\times \B_{R}) \setminus (\B_{\rho}^c\times \B_{\rho}^c)}\Big(|u_j(x)-u_j(y)|^2 - |u(x)-u(y)|^2\Big) \frac{\alpha_{n,s}}{|x-y|^{n+s}} \, \cdot \\ & \hspace{7cm} \cdot \frac{\frac{\alpha_{n,s}}{|\psi_X^t(x)-\psi_X^t(y)|^{n+s}}J_t(x)J_t(y)}{\frac{\alpha_{n,s}}{|x-y|^{n+s}}}\,dx\,dy\bigg|\, . \label{2termrew}
\end{align}
Since $u_j\to u$ in $H^{s/2}_{\rm loc}(\R^n)$ by assumption, one immediately sees that
$$A_j(x,y):=\Big(|u_j(x)-u_j(y)|^2-|u(x)-u(y)|^2\Big)\frac{\alpha_{n,s}}{|x-y|^{n+s}}\to 0 \quad \mbox{in } L^1_{\rm loc}\, .$$
On the other hand,
$$B(x,y):=\frac{\frac{\alpha_{n,s}}{|\psi_X^t(x)-\psi_X^t(y)|^{n+s}}J_t(x)J_t(y)}{\frac{\alpha_{n,s}}{|x-y|^{n+s}}}$$
is a fixed function in $L^\infty_{\rm loc}$ by \eqref{cdsacsdc2} in Proposition \ref{sdfgsdrgs}. Thus $A_jB\to 0$ in $L^1_{ \rm loc}$, and this means that \eqref{2termrew} goes to $0$ as $j\to\infty$ as well.

\vsp
Putting everything together, we deduce that
\begin{align*}
    \limsup_{j\to\infty}\Big|\textnormal{Per}_s^{(j)}(\psi_X^t(F_j);\B_\rho)-\textnormal{Per}_s^{\R^n}(\psi_X^t(F);\B_\rho)\Big|&\leq \ep\, ;
\end{align*}
since $\ep$ was arbitrary, we conclude that 
$$f_j(t)=\textnormal{Per}_s^{(j)}(\psi_X^t(F_j);\B_\rho)\to\textnormal{Per}_s^{\R^n}(\psi_X^t(F);\B_\rho)=g(t)\, .$$
As explained before, this gives the desired result.
\end{proof}

The main result of this section is the following:
\begin{theorem}[Convergence to a limit]\label{BlowUpConv}
Let $F_j\subset\R^n$ be a blow-up sequence. Then, there exists a Euclidean $s$-minimal surface $F\subset \R^n$ such that a subsequence of the $v_j:=\chi_{F_j}-\chi_{(\R^n\setminus F_j)}$ converges to $v:=\chi_F-\chi_{(\R^n\setminus F)}$ in $H^{s/2}_{\rm loc}(\R^n)$.
\end{theorem}

\begin{proof} We divide the proof in two steps.

    \vsp
    \textbf{Step 1.} Convergence to a limit set $F$.

    \vsp \noindent
    Fix a radius $R$. For $j$ large enough so that $R<R_j/4$, the $F_j$ satisfy a uniform BV estimate in $\B_R$, as indicated in the third bullet of Proposition \ref{eucproperties1}. As in Step $1$ of the proof of Theorem \ref{StrongConv} (see page \pageref{fasdfad}), a bound on the BV norm implies that a subsequence of the $v_j=\chi_{F_j}-\chi_{F_j^c}$ converges strongly in $H^{s/2}(\B_R)$ norm. Iterating the same reasoning on increasingly large balls and using a diagonal selection argument, we can find a subsequence (still denoted by $v_j$) converging in each of the norms $H^{s/2}(\B_k)$, $k\in \N$, to a limit function $v=\chi_F-\chi_{F^c}$.

    \vsp
\textbf{Step 2.} Proof that $F$ is stationary for the fractional perimeter.

 \vsp \noindent
Fix an arbitrary Euclidean vector field $X \in C_c^\infty(\B_\rho)$, for some $\rho>0$, and let $\psi_X^t$ denote its flow at time $t$.\\
Since the $F_j$ are a blow-up sequence, let $\widehat E_j\subset \widehat M_j$ and $R_j\to\infty$ be those given by Remark \ref{rescblowdef}. For $j$ large enough so that $\rho<R_j$, define $X_j=(\widehat\varphi_j)_*(X)$; extending it by $0$, we obtain a vector field $X_j$ defined on all of $\widehat M_j$. Since $\partial\widehat E_j$ is an $s$-minimal surface in $\widehat M_j$, $\frac{d}{d  t}\big|_{t=0}\textnormal{Per}_s^{\widehat M_j}(\psi_{X_j}^{t}(\widehat E_j);\widehat\varphi_j(\B_{\rho}))=0$. Lemma \ref{lemrescen} gives then that 
$$\left|\frac{d}{d  t}\bigg|_{t=0}\textnormal{Per}_s^{\R^n}(\psi_{X}^t(F);\B_\rho)\right|=\lim_{j\to\infty}\left|\frac{d}{d  t}\bigg|_{t=0}\textnormal{Per}_s^{(j)}(\psi_{X}^t(F_j);\B_\rho)\right|\leq \lim_{j\to\infty}\frac{C_X}{R_j^s}=0$$
as desired.
\end{proof}

We will next prove that the convergence in the theorem also holds in the Hausdorff distance sense. First, we show that the assumptions in Definition \ref{blowupdef} imply uniform density estimates.
\begin{lemma}\label{densityest222}
Let $(M,g)$ be a closed manifold of dimension $n$, satisfying the flatness assumption ${\rm FA}_3(M,g,R,p,\varphi)$. Suppose in addition that $g_{kl}^{(j)}(0) = \delta_{kl}$, i.e. the metric of $M$ with respect to the chart $\varphi^{-1}$ at the point $0$ is the Euclidean metric. Let $E$ be an $s$-minimal surface in $M$, satisfying a uniform BV estimate in the sense that there is some $C_0$ such that
$$
\textnormal{Per}(\varphi^{-1} (E); \B_{r}(x))\leq C_0 r^{n-1} \quad \mbox{for all } x\in \B_{R/2} \mbox{ and } r\in (0,R/4)\, .
$$
Then there exists a positive constant $\omega_0=\omega_0(n,s, C_0)$ such that if
$$
r^{-n}|E\cap B_r(q)|\leq \omega_0
$$
for some $q\in \varphi(\B_{R/2})$ and $r\in (0,R/8)$, then
$$
|E\cap B_{r/2}(q)|=0\, .
$$
\end{lemma}
\begin{proof}
    Notice that since the statement is scaling-invariant, it suffices to prove it for $R=1$. We also recall that the stationarity of $E$ implies that it satisfies the monotonicity formula of Theorem \ref{monfor} (with potential $F\equiv 0$). Observe also that up to modifying $E$ on a set of measure zero we can assume that its topological boundary coincides with its essential boundary. We then proceed as follows:
    
    \vsp
    \textbf{Step 1.} Positive density of the extended energy at every boundary point.

     \vsp \noindent
    Since $\varphi^{-1}(E)$ is a set of finite perimeter in $\B_{1/2}$, De Giorgi’s structure theorem for sets of finite perimeter gives that if $x\in\partial \varphi^{-1}(E)\cap \B_{1/2}$ is in the reduced boundary, given $r_j\to 0$ the sequence of sets $H_j=\frac{1}{r_j}(\varphi^{-1}(E)-x)$ converges in $L^1_{\rm loc}(\R^n)$ to a half-space $H$ passing through $0$.

    \vspace{1pt}
    For a fixed $x$ as above, defining $M_j=M$, $E_j=E$, $p_j=\varphi(x)$, $r_j=\frac{1}{j}$, and $\varphi_j(y)=\varphi(x+A(x)y)$, where $A(x)$ is a matrix chosen so that the metric of $M$ is the identity at $0$ in the coordinates given by $\varphi_j$, the associated $F_j:=\frac{1}{r_j}\varphi_j^{-1}(E)\subset\R^n$ are a blow-up sequence (in the sense of Definition \ref{blowupdef}). Thus, by Theorem \ref{BlowUpConv} they converge in $L^1_{\rm loc}$ to a limit $F$. On the other hand, $F_j=A(x)^{-1}\frac{1}{r_j}(\varphi^{-1}(E)-x)=A(x)^{-1}H_j$, so that in fact $F=A(x)^{-1}H$ and thus it is also a hyperplane passing through $0$.

    \vsp
    Let $N_j$ denote the rescaled manifold $(M,\frac{1}{r_j^2}g)$, and write $u_j=\chi_{E}-\chi_{E^c}$, viewed as a function on $N_j$. Write $U_j$ for its Caffarelli-Silvestre extension to $N_j\times \R_+$, and $V$ for the Caffarelli-Silvestre extension of $\chi_F-\chi_{F^c}$ to $\R^n\times\R_+$. By the lower semicontinuity of the extended Sobolev energy under a blow-up, seen for example arguing as in Step 2 in the proof of Lemma \ref{2Dcone},
    \begin{equation}\label{lbext}
    \liminf_{j \to \infty}\int_{\widetilde B_1^{N_j}(p_j,0)}|\widetilde \nabla U_j(p,z)|^2\, dV_p\,z^{1-s}dz\geq \int_{\widetilde \B_1} |\widetilde D V(x,z)|^2dx\,z^{1-s}dz=c(n,s)>0\, .
    \end{equation}
    If $U$ denotes the Caffarelli-Silvestre extension of $u=\chi_E-\chi_{E^c}$ (viewed as a function on $M$) to $M\times\R_+$, by scaling (recall that $N_j$ is just $(M,\frac{1}{r_j^2}g)$) the inequality \eqref{lbext} can be written as
    $$
    \liminf_{j \to \infty} \frac{1}{r_j^{n-s}}\int_{\widetilde B_{r_j}^{M}(\varphi(x),0)}|\widetilde \nabla U(p,z)|^2dV_p\,z^{1-s}dz\geq c(n,s)>0\, .
    $$
    In words, we have found that $E$ has extended energy density uniformly bounded from below by a constant $c(n,s)$ at $p=\varphi(x)$, for every reduced boundary point $p$ as above. On the other hand, the reduced boundary is dense in the essential boundary, as one can see, for example, by the isoperimetric inequality. Then, since we have shown that the above lower bound holds at all reduced boundary points $p\in\partial E$, by the upper semicontinuity of the extended energy density (proved as in case of classical minimal surfaces, using the monotonicity formula of Theorem \ref{monfor}) it actually holds at every $p\in\partial E$.

    \vsp
    \textbf{Step 2.} Conclusion.

     \vsp \noindent
    Assume that there are $q\in\varphi(\B_{1/2})$ and $r\in(0,1/8)$ such that $r^{-n}|E\cap B_r(q)|\leq \omega_0$ but $|E\cap B_{r/2}(q)|>0$; if $\omega_0$ is small enough, then automatically also $|E^c\cap B_{r/2}(q)|>0$. By the isoperimetric inequality, this implies that $\partial E\cap B_{r/2}(q)\neq \emptyset$ as well.

    \vspace{2pt}
    Let $p\in\partial E\cap B_{r/2}(q)$; we can now argue as in the proof of Proposition \ref{DensEst}, which showed density estimates in the case of solutions of the Allen-Cahn equation. First, a uniform lower bound on the density holds in our case for all $0\leq\rho\leq R_{\rm mon}$, thanks to combining Step 1 and the monotonicity formula. We then apply the interpolation Lemma \ref{lem:whtorwohh}, after which the $BV$ estimate assumption allows us to conclude the argument as in the proof of Proposition \ref{DensEst}.
\end{proof}

\begin{proposition}\label{hausconvprop}
    In the conclusions of Theorem \ref{BlowUpConv}, the convergence also holds locally in the Hausdorff distance sense.
\end{proposition}
\begin{proof}
    Let $E_j$ be as in Remark \ref{rescblowdef}, so that $F_j=\widehat\varphi_j^{-1}(E_j)$. Applying Lemma \ref{densityest222} and the flatness assumption on the metric we find that the $F_j$ satisfy density estimates in $\B_{R_j/16}$, with $R_j\to\infty$. The local convergence in the Hausdorff distance follows then arguing by contradiction, simply due to local $L^1$ convergence to $F$ and the density estimates.
\end{proof}

\subsection{Properties of blow-ups of Allen-Cahn limits}\label{ACpropsect}
We define the class of all surfaces which are blow-up limits of sets in $\mathcal{A}_m$ (recall Definitions \ref{AClimits} and \ref{blowupdef}).
\begin{definition}\label{BlowUpClass}
    A set $\partial F\subset\R^n$ is said to be in the class $\mathcal{A}_{m}^{Blow-up}$ if it is a blow-up limit of sets in $\mathcal{A}_m$. This means that there exist $\Sigma_j=\partial E_j\in\mathcal{A}_m(M_j)$ and $r_j\to 0$ such that the associated $F_j= r_j^{-1}\varphi_j^{-1}(E_j)$ are a blow-up sequence converging to $F$ in $L^1_{\rm loc}(\R^n)$ as $j\to\infty$ (by Theorem \ref{BlowUpConv} and Proposition \ref{hausconvprop}, the convergence can then be upgraded to be in $H^{s/2}_{\rm loc}(\R^n)$ and locally in the Hausdorff distance sense).
\end{definition}
\begin{remark}
Since $\partial E_j\in\mathcal{A}_m(M_j)$, the assumption in Definition \ref{blowupdef} that the sets $F_j$ satisfy the classical perimeter estimates is automatically satisfied if the rest of assumptions are, thanks to (1) in Remark \ref{remprop}.
\end{remark}

We now prove a precise almost-stability inequality for sets in $\mathcal{A}_{m}^{Blow-up}$, which will be used in the next section. We begin by showing its counterpart for Allen-Cahn solutions.
\begin{lemma}\label{improvineq1} Assume that $M$ satisfies the flatness assumptions ${\rm FA}_1(M,g,1, p, \varphi)$, and let $u_\ep:M\to\R$ be a solution to the Allen-Cahn equation in $\varphi(\B_1)$ with Morse index at most $m$. Let $A_1,...,A_{m+1}\subset \varphi(\B_{1/2})$ be $(m+1)$ open sets, with pairwise distances denoted by $D_{ij}:=\textnormal{dist}(A_i,A_j)$, and for every $1\leq i<j \leq m+1$ choose any positive weights $\lambda_{ij}>0$.  Then, in at least one of the $A_i$ there holds that
\begin{equation*}
\begin{split}
    \mathcal{E}^{''}(u_{\ep})[\xi ,\xi ]\geq -C\|\xi \|_{L^1(A_i)}^2 \left( \sum_{j<i} \frac{1}{\lambda_{ji}} D_{ij}^{-(n+s)}
    +\sum_{j>i} \lambda_{ij} D_{ij}^{-(n+s)} \right) \s \forall \, \xi \in C_c^1(A_i) \,,
\end{split}
\end{equation*}
for some $C=C(n,s,m)$. 
\end{lemma}
\begin{proof}
The statement is a more precise version of Lemma \ref{asineq}, and the proof proceeds similarly. Using \eqref{2ndvar}, we compute the second variation at $u_\ep$ for linear combinations of $m+1$ test functions $\xi_i$, supported each in the corresponding $A_i$, getting 
\begin{equation*}
\begin{split}
    \mathcal{E}^{''}(u_\ep) &[a_1\xi_1+a_2\xi_2+\dotsc+a_{m+1}\xi_{m+1},a_1\xi_1+a_2\xi_2+\dotsc+a_{m+1}\xi_{m+1}]   \\ 
    & = a_1^2\mathcal{E}^{''}(u)[\xi_1,\xi_1] + \dotsc+a_{m+1}^2\mathcal{E}^{''}(u)[\xi_{m+1},\xi_{m+1}] \\
    & \hspace{0.2cm} +2 a_1a_2\iint_{A_1 \times A_2}(\xi_1(p) - \xi_1(q))(\xi_2(p) - \xi_2(q)) K_s(p,q)\,dV_p\,dV_q\\
    & \hspace{0.2cm}+\dotsc\\
    & \hspace{0.2cm}+2 a_ma_{m+1}\iint_{A_m \times A_{m+1}}(\xi_m(p) - \xi_m(q))(\xi_{m+1}(p) - \xi_{m+1}(q)) K_s(p,q)\,dV_p\,dV_q\,.
\end{split}
\end{equation*}
Thanks to the flatness assumptions and Lemma \ref{loccomparability}, we have that $K_s(\varphi(x),\varphi(y))\leq \frac{C}{|x-y|^{n+s}}$, for some $C=C(n,s)$ and for $(\varphi(x),\varphi(y))\in A_i\times A_j$. Recall that the supports of $\xi_i$ and $\xi_j$ are the disjoint subsets $A_i,A_j\subset\varphi_j(\B_{1/2})$. Then, the term containing the double integral over $A_i \times A_j$ with $i<j$ can be bounded as follows:
\begin{equation*}
\begin{split}
    2a_ia_j \iint_{A_i \times A_j} (\xi_i(p)-\xi_i(q))&(\xi_j(p)-\xi_j(q)) K_s(p,q) \, dV_p dV_q \\ &= -2a_ia_j\iint_{A_i \times A_j}\xi_i(p) \xi_j(q) K(p,q)\,dV_p\,dV_q \\
    &\leq 2|a_ia_j| CD_{ij}^{-(n+s)}\|\xi_i\|_{L^1(A_i)}\|\xi_j\|_{L^1(A_j)} \\
     &\leq \lambda_{ij} a_i^2 C D_{ij}^{-(n+s)}\|\xi_i\|_{L^1(A_i)}^2+\frac{C}{\lambda_{i j}} a_j^2D_{i j}^{-(n+s)}\|\xi_j\|_{L^1(A_j)}^2\, ,
\end{split}
\end{equation*}
where we have applied Young's inequality in the last line. Substituting this into the second variation expression gives
\begin{align*}
    \mathcal{E}^{''}(u) & [a_1\xi_1+a_2\xi_2+\dotsc+a_{m+1}\xi_{m+1},a_1\xi_1+a_2\xi_2+\dotsc+a_{m+1}\xi_{m+1}]   \\  \le
    &  \sum_{i=1}^{m+1} a_i^2\left[\mathcal{E}^{''}(u)[\xi_i,\xi_i] +C\|\xi_i\|_{L^1(A_i)}^2 \left(\sum_{j<i} \frac{1}{\lambda_{ji}} D_{ij}^{-(n+s)}
    +\sum_{j>i} \lambda_{ij} D_{ij}^{-(n+s)} \right)\right].
\end{align*}

The condition that the Morse index is at most $m$ implies that the expression cannot be $<0$ for all $(a_1,\dotsc,a_{m+1})\neq 0$. Hence, we find that there must exist some $i$ such that 
\begin{equation*}
\begin{split}
    \mathcal{E}^{''}(u)[\xi_i,\xi_i]\geq -C\|\xi_i\|_{L^1(A_i)}^2 \left(\sum_{j<i} \frac{1}{\lambda_{ji}} D_{ij}^{-(n+s)}
    +\sum_{j>i} \lambda_{ij} D_{ij}^{-(n+s)} \right)
\end{split}
\end{equation*}
holds for all $\xi_i \in C_c^1(A_i)$, and this concludes the proof.
\end{proof}
From this, we will obtain the desired almost-stability inequality for blow-up sets.
\begin{lemma}\label{lemasbup}
Let $F\in\mathcal{A}_m^{Blow-up}$. Let $X_1,X_2,...,X_{m+1}$ be smooth vector fields on $\R^n$ with disjoint compact supports $A_1,A_2,...,A_{m+1}$, and denote $D_{k\ell}:=\textnormal{dist}(A_k,A_\ell)$.
For $1\leq i<\ell \leq m+1$, choose positive weights $\lambda_{i\ell}>0$. Then, for at least one of the $i$ (depending on $F$) we have that
\begin{equation}\label{asbup}
\frac{d^2}{dt^2}\bigg|_{t=0}\textnormal{Per}_s(\psi_{X_{i}}^t(F);A_i)\geq -C \|X_i\|_{L^\infty}^2\textnormal{diam}(A_i)^{2(n-1)}\left(\sum_{\ell<i} \frac{1}{\lambda_{\ell i}} D_{i \ell}^{-(n+s)}
    +\sum_{\ell>i} \lambda_{i \ell} D_{i\ell}^{-(n+s)}\right)\, ,
\end{equation}
where $C=C(n,s,m)$.
\end{lemma}

\begin{proof}
Since $F\in\mathcal{A}_m^{Blow-up}$, from the definition and proceeding as in Remark \ref{rescblowdef} there exist $(\widehat M_j,\widehat g^{(j)})$, $\widehat p_j\in\widehat M_j$, and $R_j\nearrow\infty$ satisfying the assumptions in the Remark and $\widehat E_j\in \mathcal{A}_m(\widehat M_j)$ such that the associated $F_j=\widehat\varphi_j^{-1}(\widehat E_j)$ converge to $F$ in the appropriate sense. Fix one such $j$; since $\widehat E_j\in \mathcal{A}_m(\widehat M_j)$, by definition there exists a sequence $\{u_k\}_{k\in\N}$ of solutions to Allen-Cahn on $\widehat M_j$, with parameters $\ep_k\to 0$, converging to $\chi_{\widehat E_j}-\chi_{\widehat E_j^c}$ in $L^1(\widehat M_j)$ as $k\to \infty$. By Lemma \ref{improvineq1}, given $k$ we can find an index $i(k)$, $1\leq i(k)\leq m+1$, such that the inequality in the Lemma is true for $u_{k}$ on $\widehat\varphi_j(A_{i(k)})\subset\widehat M_j$. We select an index $i$ so that the inequality is valid for a whole subsequence of the $u_k$ (which we do not relabel), so that 
\begin{equation*}
    \mathcal{E}^{''}(u_{k})[\xi_i,\xi_i]\geq -C\|\xi_i\|_{L^1(\widehat\varphi_j(A_i))}^2\left(\sum_{\ell<i} \frac{1}{\lambda_{\ell i}} \widehat D_{i\ell}^{-(n+s)}
    +\sum_{\ell>i} \lambda_{i\ell} \widehat D_{i \ell}^{-(n+s)}\right)\,
\end{equation*}
for all $\xi_i\in C_c^1(\widehat\varphi_j(A_i))$ and $k\in\N$. Here $\widehat D_{i\ell}=\text{dist}(\widehat\varphi_j(A_i),\widehat\varphi_j(A_\ell))$.\\ 
Put $X_{i,j}:=(\widehat\varphi_j)_*X_i$, and extend it by zero outside its domain of definition to a vector field on all of $\widehat M_j$. Selecting $\xi_i=\nabla_{X_{i,j}}u_{k}$, we arrive at
\begin{equation}\label{2acas}
\begin{split}
    \frac{d^2}{dt^2}\bigg|_{t=0}\mathcal{E}(u_k\circ\psi_{X_{i,j}}^{-t})&=\mathcal{E}^{''}(u_k)[\nabla_{X_{i,j}}u_k,\nabla_{X_{i,j}}u_k]\\
    &\geq -C\|\nabla_{X_{i,j}}u_k\|_{L^1(\widehat\varphi_j(A_i))}^2\left(\sum_{\ell<i} \frac{1}{\lambda_{\ell i}} \widehat D_{i \ell}^{-(n+s)}
    +\sum_{\ell>i} \lambda_{i \ell} \widehat D_{i \ell}^{-(n+s)}\right)\, .
\end{split}
\end{equation}

Thanks to the BV estimate of Theorem \ref{BVEst} and the flatness assumption on the metric, we can bound 
$$
\|\nabla_{X_{i,j}}u_{k}\|_{L^1(\widehat\varphi_j(A_i))}^2\leq C \|X_{i,j}\|_{L^\infty}^2\text{diam}(\widehat\varphi_j(A_i))^{2(n-1)}\leq C\|X_{i}\|_{L^\infty}^2\text{diam}(A_i)^{2(n-1)}\, ,
$$
and also
$$
\widehat D_{i,l}=\text{dist}(\widehat\varphi_j(A_i),\widehat\varphi_j(A_l))\leq C\text{dist}(A_i,A_l)=C D_{i,l}\, .
$$
Substituting this into \eqref{2acas}, and using that by \eqref{q5etgqwa4e5} there holds
$$\frac{d^2}{dt^2}\bigg|_{t=0}\mathcal{E}(u_k\circ\psi_{X_{i,j}}^{-t})\to \frac{d^2}{dt^2}\bigg|_{t=0}\textnormal{Per}_s^{\widehat M_j}(\psi_{X_{i,j}}^t(\widehat E_j);\widehat\varphi_j(A_i)) \quad \mbox{as } k\to\infty\, ,$$
we obtain that
\begin{equation*}
\frac{d^2}{dt^2}\bigg|_{t=0}\textnormal{Per}_s^{\widehat M_j}(\psi_{X_{i,j}}^t(\widehat E_j);\widehat\varphi_j(A_i))\geq -C \|X_i\|_{L^\infty}^2\textnormal{diam}(A_i)^{2(n-1)}\left(\sum_{\ell<i} \frac{1}{\lambda_{\ell i}}  D_{i \ell}^{-(n+s)}
    +\sum_{\ell>i} \lambda_{i \ell} D_{i \ell}^{-(n+s)}\right)\, .
\end{equation*}

On the other hand, by Lemma \ref{lemrescen} we have that 
$$
\frac{d^2}{dt^2}\bigg|_{t=0}\textnormal{Per}_s^{\R^n}(\psi_{X_i}^t(F);A_i)=\lim_j \frac{d^2}{dt^2}\bigg|_{t=0}\Big[\textnormal{Per}_s^{(j)}(\psi_{X_i}^t(F_j);A_i)\Big]=\lim_j\frac{d^2}{dt^2}\bigg|_{t=0}\Big[\textnormal{Per}_s^{\widehat M_j}(\psi_{X_{i,j}}^t(\widehat E_j);\widehat\varphi_j(A_i))\Big]\, ,
$$
which then proves \eqref{asbup}.

\end{proof}

\subsection{Classification of blow-up limits}\label{BlowupLimSection}

The main result of this section is the following classification result:

\begin{theorem}[\textbf{Classification result}] \label{A-CBernstein} Let $s\in(0,1)$ and $3\leq n < n_s^*$. Let $\mathcal{F}$ be any family of sets of $\R^n$ satisfying the following properties:

\vspace{5pt}
(1) \textbf{Stationarity.} Every set $E\in \mathcal F$ is an $s$-minimal surface, in the sense of Definition \ref{def-fracminsurface}.
 
  \vspace{5pt}
(2) \textbf{BV estimate.} There is $C_\circ$ such that for every $E\in \mathcal F$, $x\in \R^n$, and $R>0$ we have
\begin{equation*}
    \textnormal{Per}(E; B_R(x)) \leq C_\circ R^{n-1}\, .
\end{equation*}

 \vspace{5pt}
(3) \textbf{Viscosity solution of the NMS\footnote{Meaning ``nonlocal minimal surface".} equation.} If $x_0 \in \partial E$ and $E$ admits an interior (resp. exterior) tangent ball at $x_0$, then $\int \frac{\chi_E(y)-\chi_{E^c}(y)}{|x_0-y|^{n+s}}dy \leq 0$ \ (resp. $\geq 0$).

 \vspace{5pt}
(4) \textbf{Almost-stability in one out of $(m+1)$ disjoint sets.} There exists some (fixed) $m\in \N$ such that the following holds. Let $X_1,X_2,...,X_{m+1}$ be smooth vector fields with disjoint compact supports $A_1,A_2,...,A_{m+1}$, and denote $D_{kl}:=\textnormal{dist}(A_k,A_l)$.
For $1\leq i<l \leq m+1$, choose positive weights $\lambda_{il}>0$. Then, given $E\in\mathcal F$, for at least one of the $i$ (depending on $E$) we have that
\begin{equation}\label{asper}
\frac{d^2}{dt^2}\Big|_{t=0}\textnormal{Per}_s(\psi_{X_{i}}^t(E);A_i)\geq -C \|X_i\|_{L^\infty}^2\textnormal{diam}(A_i)^{2(n-1)}(\sum_{\ell<i} \frac{1}{\lambda_{\ell i}} D_{i \ell}^{-(n+s)}
    +\sum_{\ell>i} \lambda_{i \ell} D_{i\ell}^{-(n+s)})\, ,
\end{equation}
where $C=C(\mathcal{F})$.
 
 \vspace{5pt}
(5) \textbf{Completeness under scalings and $L^1_{\rm loc}(\R^n)$ limits}.
If $E\in\mathcal F$, then any translation, dilation and rotation of $E$ is in $\mathcal F$ as well. Moreover, if $E_i$ is a sequence of elements of $\mathcal F$ and $E_i\to E_\infty$ in $L_{\rm loc}^1(\R^n)$, then $E_\infty\in\mathcal F$ as well.

 \vspace{5pt}
(6) \textbf{Cones with $n-2$ translation-invariant directions are half-spaces.} If $E\in \mathcal F $ is a cone and there is a linear $(n-2)$-dimensional subspace $L\subset \R^n$ such that $E+x = E$ for all $x\in L$, then $\partial E$ must be a hyperplane.

\vspace{5pt}
Then, every  $E\in\mathcal{F}$ which is not equal (up to null sets) to $\R^n$ or $\varnothing$ must be a half-space.
\end{theorem}

An important property follows from (1) and (2) above:
\begin{lemma}\label{lemmondens}
Let $\mathcal{F}$ be a family of sets of $\R^n$ satisfying properties (1) and (2) in Theorem \ref{A-CBernstein}. Then, any set $E\in \mathcal F$ also satisfies density estimates, meaning that
there exists a positive constant $\omega_0=\omega_0(n,s, C_0)$ such that if
$$
R^{-n}|E\cap B_R(q)|\leq \omega_0
$$
for some $q\in \varphi(\B_{1/2})$ and $R\in (0,1/8)$, then
$$
|E\cap B_{R/2}(q)|=0\, .
$$
Moreover, if $E_i\in \mathcal F$ and $E_i\to E_\infty$ in $L_{\rm loc}^1(\R^n)$, then they also converge to $E_\infty$ locally in the Hausdorff distance sense.
\end{lemma}
\begin{proof}
    Same as for Lemma \ref{densityest222} and Proposition \ref{hausconvprop}.
\end{proof}

\vsp 

We will need the following result, which is obtained by combining the $C^{1,\alpha}$ improvement of flatness theorem in \cite{CRS} and the $C^{1,\alpha}$-to-$C^\infty$ bootstrap result for nonlocal minimal graphs in \cite{BFV}.
\begin{theorem}[\cite{BFV, CRS}]\label{improvflatRn}
 Let $s \in (0,1)$. Then, there exists $\sigma > 0$, depending on $n$ and $s$, such that the following holds: let $E \subset \R^n$ and $ x \in \partial E$, and assume that
\begin{itemize}
    \item[\textit{(i)}] The set $E$ is a viscosity solution of the NMS equation in $\B_r(x)$, in the sense of Proposition \ref{prop:viscosity}.
    \item[\textit{(ii)}] The boundary $\partial E$ is included in a $\sigma$-flat cylinder in $\B_{r}(x)$, that is 
    \begin{equation*}
        \partial E \cap \B_{r}(x) \subset \{y\in\R^n: |e \cdot (y-x)| \le \sigma r \},
    \end{equation*}
    for some direction $e\in\Sp^{n-1}$. 
\end{itemize}
Then $\partial E$ is a $C^{\infty}$ graph in the direction $e$ in $ \B_{r/2}(x)$, with uniform estimates. In particular, its second fundamental form $ {\rm II}_{\partial E} $ satisfies
\begin{equation}\label{2formbound}
    \sup_{y\in \partial E\cap\B_{r/2}(x)}| {\rm II}_{\partial E} |(y) \le \frac{C}{r}\, ,
\end{equation}
with $C=C(n,s)$.
\end{theorem}

We will also need the following intuitive lemma, to be read as ``cones with finite Morse index are stable outside the origin'', and which will be proved after Theorem \ref{A-CBernstein}.

\begin{lemma}\label{StableCone} Let $E \subset \R^n$ be a cone with $\textnormal{Per}_s(E; \B_1(0)) <+\infty $. Assume that $E$ is stationary for the $s$-perimeter, in the sense of Definition \ref{def-fracminsurface}, and that it satisfies property (4) in the statement of Theorem \ref{A-CBernstein}. Then $E$ is stable in $\R^n\setminus \{0\}$.
\end{lemma}

We will now prove Theorem \ref{A-CBernstein}.

\begin{proof}[Proof of Theorem \ref{A-CBernstein}]

Let $E_\infty$ be a blow-down limit of $E$, i.e., a limit of a sequence $E_i=\frac{1}{r_i}E$, with $r_i\to\infty$ (by property (5), such a limit exists, and it is also a member of $\mathcal{F}$). Then $E_\infty$ is a cone: Let $U$, $U_\infty$ and $U_i$ denote the Caffarelli-Silvestre extensions of $u:=\chi_E-\chi_{E^c}$, $u_\infty:=\chi_{E_\infty}-\chi_{E_\infty^c}$ and 
$u_i:=\chi_{E_i}-\chi_{E_i^c}$, respectively. Using the notation $\Phi_V(r):=r^{s-n}\int_{\tilde B_r^+(0,0)}z^{1-s}|\nabla V(x,z)|^2 dx\,dz$, by convergence of the extended energies\footnote{Follows easily from convergence in $H^{s/2}_{\rm loc}(\R^n)$. The latter is proved, thanks to property (2), as in Step 1 in the proof of Theorem \ref{BlowUpConv}.} and scaling we have that
$$
\Phi_{U_\infty}(r)=\lim_i \Phi_{U_i}(r)=\lim_i \Phi_{U}(rr_i)\, .
$$
By the monotonicity of $\Phi_E$, which we know since $E$ is an $s$-minimal surface by property (1) and thus satisfies Theorem \ref{monfor}, the limit $\lim_{R\to\infty} \Phi_{U}(R)$ exists, and by property (2) and the interpolation result in Lemma \ref{lem:whtorwohh} it is a finite constant. The equality above then shows that $\Phi_{U_\infty}(r)$ is equal to this constant independently of $r$. Since $E_\infty$ is also an $s$-minimal surface (by properties (1) and (5)), the last paragraph in Theorem \ref{monfor} gives that $E_\infty$ is a cone.

\vsp
We will now prove that $E_\infty$ is in fact a hyperplane; by the local Hausdorff convergence of the $E_i=\frac{1}{r_i}E$ to $E_\infty$ (see Lemma \ref{lemmondens} above), $E$ then satisfies the hypotheses of Theorem \ref{improvflatRn} for every $r>0$, and therefore by \eqref{2formbound} $E$ then needs to be a hyperplane as well (since ${\rm II}_{\partial E}$ vanishes).

\vsp
First, Lemma \ref{StableCone} states that $E_\infty$ is stable outside the origin. If it also were smooth outside the origin, the assumption that $3\leq n < n_s^*$ would imply that $\partial E_\infty$ is a hyperplane and we would finish the proof.\\
If, arguing by contradiction, there is instead some point $x_1\neq 0$ where $E_\infty$ is not smooth, we need to apply a dimension reduction argument: blowing up around $x_1$, we obtain a new cone $E_1\in\mathcal{F}$ which is now translation invariant along some direction; after a rotation, we can write $E_1=\widetilde E_1\times\R$, and this is allowed by property (5).

\vsp
We claim that $E_1$ cannot be smooth outside the origin. First, if that were the case, $E_1$ would be a hyperplane, since $3\leq n < n_s^*$. Now, the blow-up rescalings of $E_\infty$ around $x_1$ converge locally in the Hausdorff distance sense to $E_1$ by Lemma \ref{lemmondens}, and they are viscosity solutions of the NMS equation by property (3). If $E_1$ were indeed a hyperplane, the assumption in the improvement of flatness Theorem \ref{improvflatRn} would be satisfied for the blow-up rescalings of $E_\infty$ around $x_1$ (for large enough indices in the sequence). Thus $E_\infty$ would be smooth in a neighborhood of $x_1$, a contradiction.

\vsp
Now that we know that $E_1$ is not smooth outside the origin, we can iterate the argument with this new cone: Since $E_1=\widetilde E_1\times\R$ is not smooth outside the origin, there is some point $x_2\in\R^{n-1}\setminus\{0\}$ where $\widetilde E_1$ is not smooth. Hence, we can blow up again around $(x_2,0)$ and obtain a new cone $E_2$, which is now translation invariant with respect to two orthogonal directions. After a rotation, $E_2=\widetilde E_2 \times \R^2$. Moreover, $E_2$ cannot be smooth outside the origin, by the same improvement of flatness argument we applied to $E_1$.\\
Iterating this reasoning $n-2$ times, we end up with a cone that is translation invariant in $n-2$ orthogonal directions, i.e., of the form $\widetilde E \times \R^{n-2}$ after a rotation, and which is not smooth outside the origin. This is not possible by property (6), and therefore we reach a contradiction.
\end{proof}

We now give the proof of Lemma \ref{StableCone}.
\begin{proof}[Proof of Lemma \ref{StableCone}]
Consider an annular region of the form $A_0=B_{1}\setminus B_{R_0}$ with $ 0< R_0 < 1$, centered at the origin. It suffices to show that $E$ is stable in $A_0$, by the arbitrariness of $R_0$ and the dilation invariance of $E$.

\vsp
The strategy is the following. Let $X$ be a vector field supported on the annulus $A_0$. Let $A_1,...,A_{m+1}$ be $(m+1)$ rescaled copies of $A_0$ of the form $A_{i}=RA_{i-1}=R^iA_0$, with $R>0$ sufficiently large so that they are disjoint. Likewise, consider the $(m+1)$ rescaled vector fields $X_{i}:=R^iX(x/R^i)$, which are supported in the respective $A_i$. Since $E$ satisfies property (4) in the statement of Theorem \ref{A-CBernstein}, we know that the almost-stability inequality (\ref{asper}) will hold in at least one of the $A_i$. Moreover, since $E$ is dilation invariant, we will be able to translate this information back into $A_0$, and taking $R$ arbitrarily large we will find that $E$ is actually stable on $A_0$ and conclude the proof.

\vsp
Define $ u :=\chi_E-\chi_{E^c}$, and let $\psi_X^t$ denote the flow of the vector field $X$ at time $t$. Observe that $u_{i}^t:=\chi_{\psi_{X_i}^t(E)}-\chi_{\psi_{X_i}^t(E)^c}$ is the composition of $u =\chi_E-\chi_{E^c}$ with the flow of $X_{i}=R^iX(x/R^i)$, which is given by $\psi_{X_i}^t=R^i\psi_X^t(x/R^i)$. By the dilation-invariance of the cone $E$, and hence of $u_\infty$, we have that

\begin{equation*} 
u_{i}^t(x)=u(R^i\psi_X^t(x/R^i))=u(\psi_X^t(x/R^i))=u^t(x/R^i);
\end{equation*}
the scaling property of the fractional Sobolev energy then gives

\begin{equation*}
    \text{Per}_s(\psi_{X_i}^t(E); A_i)=\mathcal{E}_{A_{i}}^{\rm Sob}(u_{i}^t)=\mathcal{E}_{A_{i}}^{\rm Sob}(u^t(x/R^i))=R^{i(n-s)}\mathcal{E}_{A_0}^{\rm Sob}(u^t)=R^{i(n-s)}\text{Per}_s(\psi_{X}^t(E); A_0) \, ,
\end{equation*}
so that in particular
\begin{equation}\label{2dresrel}
    \frac{d^2}{dt^2}\bigg|_{t=0}\text{Per}_s(\psi_{X_i}^t(E); A_i)=R^{i(n-s)}\frac{d^2}{dt^2}\bigg|_{t=0}\text{Per}_s(\psi_{X}^t(E); A_0) \, .
\end{equation}
Now, by assumption, we know that the almost-stability inequality \eqref{asper} will be satisfied in one of the $A_i$. Combined with \eqref{2dresrel}, we obtain that
\begin{equation*}
\begin{split}
    R^{i(n-s)}\frac{d^2}{dt^2}\bigg|_{t=0}\text{Per}_s(\psi_{X}^t(E);A_0) &\geq -C \|X_i\|_{L^\infty}^2\text{diam}(A_i)^2(\sum_{\ell<i} \frac{1}{\lambda_{\ell i}} D_{i\ell}^{-(n+s)}
    +\sum_{\ell>i} \lambda_{i\ell} D_{i\ell}^{-(n+s)})\, ,
\end{split}
\end{equation*}
where $\lambda_{ij}$ are positive weights and $D_{ij}= \textnormal{dist}(A_i, A_j)$. We can bound
$$
\|X_i\|_{L^\infty}^2\text{diam}(A_i)^{2(n-1)}= \|X\|_{L^\infty}^2(R^i)^2(R^i\text{diam}(A_0))^{2(n-1)}\leq C_{X,A_0}R^{2ni}\, .
$$ 
We also observe that 
$$
D_{i,l} = \text{dist}(R^{i}A_0,R^{l}A_0) \geq cR^{\max\{i,l\}}
$$
for some small $c$, for all $R$ sufficiently large depending on $A_0$.

\vsp
Substituting into the inequality we obtained and dividing both sides by $R^{i(n-s)}$, we get
\begin{align*}
    \frac{d^2}{dt^2}\bigg|_{t=0}\text{Per}_s(\psi_{X}^t(E);A_0) &\geq -C_{X,A_0} R^{i(n+s)}(\sum_{\ell<i} \frac{1}{\lambda_{\ell i}} R^{-i(n+s)}
    +\sum_{\ell>i} \lambda_{i\ell} R^{-\ell(n+s)})\\
    &= -C_{X,A_0} (\sum_{\ell<i} \frac{1}{\lambda_{\ell i}}
    +\sum_{\ell>i} \lambda_{i\ell} R^{-(\ell-i)(n+s)})\, .
\end{align*}
Now, choosing the positive weights as $\lambda_{ij}=R^{\frac{n+s}{2}}$ for every pair $i<j$, we obtain
\begin{align*}
    \frac{d^2}{dt^2}\bigg|_{t=0}\text{Per}_s(\psi_{X}^t(E);A_0)
    &\geq -C_{X,A_0} (\sum_{\ell<i} R^{-\frac{n+s}{2}}
    +\sum_{\ell>i} R^{-(\ell-i-\frac{1}{2})(n+s)})\, ,
\end{align*}
so that all the powers of $R$ become strictly negative. Letting $R\to\infty$, we deduce that
\begin{align*}
    \frac{d^2}{dt^2}\bigg|_{t=0}\text{Per}_s(\psi_{X}^t(E);A_0)
    &\geq 0
\end{align*}
as desired.
\end{proof}

We will, in particular, apply Theorem \ref{A-CBernstein} to the class $\mathcal{A}_m^{Blow-up}$. The next lemma proves that property (6) in the assumptions holds for this class. This is the (nonlocal and finite index) analogue of a result of Schoen--Simon \cite[p. 785-787]{SchoenSimon}; see also \cite[Proposition 3.2]{TW10}.

\begin{lemma}\label{2Dcone} Let $n\geq 3$. Assume that some nontrivial cone $E\subset \mathbb{R}^n$
belongs to $\mathcal{A}_m^{Blow-up}$ and is of the form $\widetilde E \times \mathbb{R}^{n-2}$
for some cone $\widetilde{E} \subset\mathbb{R}^2$.
Then, $\partial E$ is a hyperplane.
\end{lemma}
\begin{proof} We divide the proof into two steps. 

    \vsp
    \textbf{Step 1.} 
Let us show the following claim: assume that ${\rm FA}_1(M,g,1,p,\varphi)$ holds, and  $u:M\to (-1,1)$ is a solution of Allen-Cahn with parameter $\ep\in (0,1)$ in $B_1(p)$ (equivalently a critical point of $\mathcal E_{B_1(p)}$) that is $\Lambda$-almost stable in $B_1(p)$ (see Definition \ref{almoststab}).
Let $U: M\times  \R_+\to (-1,1)$ be the Caffarelli-Silvestre extension of $u$. Then, for some constant $C=C(n,s,\Lambda) >0$ we have:
\[
\int_{\widetilde B_{1/2}^+(p,0)} \mathcal A^2(U) \, dV \,z^{1-s} dz  \le C\, ,
\]
where 
\[
\mathcal A^2(U) : = \big(|\nabla^2 U|^2 -|\nabla|\nabla U||^2\big)\chi_{\{|\nabla u|>0 \}} =\Big(|\nabla^2 U|^2-\nabla^2 U\Big(\tfrac{\nabla U}{|\nabla U|}, \tfrac{\nabla U}{|\nabla U|}\Big)\Big) \chi_{\{|\nabla u|>0\}} \ge 0\, .
\]
Here $\nabla^2 U$ denotes the ``horizontal'' Hessian of $U(\cdot , z)$ ---i.e. for $z$ fixed--- with respect to $g$.

    \vsp 
Indeed since $u$ is  $\Lambda$-almost stable, for all $\xi\in C^1( \widetilde B^+_{1}(p,0))$ with support contained in $\overline{ \widetilde B^+_{3/4}(p,0)}$ and trace $\xi_0$ on $z=0$, we have 
\begin{equation*}
\begin{split}
\widetilde{\mathcal E}_{1}''(U)[\xi,\xi] &=\beta_{s}\int_{\widetilde B^+_{1}}z^{1-s}|\nabla \xi|^2 \,dVdz + \ep^{-s}\int_{B_{1}} W''(u) \xi_0^2\, dV  
\\
&\geq  \mathcal{E}_{B_1}''(u)[\xi_0,\xi_0] \ge -\Lambda\|\xi_0 \|^2_{L^1(B_1)}\,,
\end{split}
\end{equation*}
where $\widetilde B^+_{1}$ and $B_{1}$ are brief notations for $\widetilde B^+_{1}(p,0)$ and $B_{1}(p)$.

\vsp
Thus, testing the above almost stability inequality with a test function that is product $\xi=c\eta$ we obtain (with a simple integration by parts similar to \cite[Proof of Theorem 1.3]{Cabre-Sanz})
\begin{align}
\int_{B_{1}} c\big( \beta_s(z^{1-s}\partial_z) & c(\,\cdot\,,0^+) -\ep^{-s} W''(u)\big) \eta^2dV \nonumber \\ & \le \beta_s \int_{\widetilde B_1^+}\big(c^2 z^{1-s}|\widetilde \nabla \eta|^2-c\widetilde{\rm div}(z^{1-s} \widetilde\nabla c)\eta^2\big)dVdz + \Lambda \Big(\int_{B_1} |c\eta|dV\Big)^2\, . \label{whtiohoew22}
\end{align}

Taking the horizontal gradient $\nabla$ of 
$\beta_s(z^{1-s}\partial_z) U(\,\cdot\,,0^+) -\ep^{-s} W'(u)=0$ at $z=0$, and computing the scalar product with $\nabla u$, we obtain 
\[
\beta_s(z^{1-s} \partial_z)|_{z=0^+} (\nabla U)\cdot \nabla u -\ep^{-s} W''(u) |\nabla u|^2=0\, .
\]
Using $\beta_s(z^{1-s}\partial_z)|_{z=0^+} (\nabla U)\cdot \nabla u =\frac{\beta_s}{2}( z^{1-s}\partial_z)|_{z=0^+} |\nabla U|^2 =
\frac{\beta_s}{2}|\nabla U| (\beta_s z^{1-s}\partial_z)|_{z=0^+} |\nabla U|$ we obtain that 
$c= |\nabla U|$ makes the left hand side of \eqref{whtiohoew22} vanish. 
Hence, for this choice of $c$ we obtain 

\begin{equation}\label{whtiohoew}
0\le \beta_s \int_{\widetilde {B}_1^+}\big(c^2 z^{1-s}|\widetilde \nabla \eta|^2-c\widetilde{\rm div}(z^{1-s} \widetilde\nabla c)\eta^2\big)dVd + \Lambda \Big(\int_{B_1} |c\eta|dV\Big)^2\,.
\end{equation}

Notice that 
\[
c\widetilde{\rm div}(z^{1-s} \widetilde\nabla c) = z^{1-s} c\Delta c + c\partial_z(z^{1-s}\partial_z c)  = \big( \tfrac 1 2\Delta (c^2) - |\nabla c|^2)\big) z^{1-s} +c\partial_z(z^{1-s}\partial_z c)\, . 
\]
Now, since $c=|\nabla U|$, the Bochner identity ---applied to each ``horizontal slice'' $M\times\{z\}$ of $\widetilde{M}$--- yields
\[
\tfrac1 2\Delta (c^2) = \nabla U\cdot \nabla (\Delta U) + |\nabla^2 U|^2 +  {\rm Ric}(\nabla u, \nabla u)\, .
\]
Since, by the equation defining the extension, $z^{1-s}\Delta U = - \partial_z(z^{1-s}\partial_z U)$, we obtain
\[
z^{1-s}\nabla U\cdot \nabla (\Delta U) = - \nabla U\cdot \partial_z(z^{1-s}\partial_z \nabla U)\, .
\]
But explicit computation shows that
\[
\begin{split}
|\nabla U|\partial_z(z^{1-s}\partial_z |\nabla U|)
 &= |\nabla U|\partial_z\bigg(\frac{z^{1-s}\partial_z(\tfrac 1 2 |\nabla U|^2)}{|\nabla U|}\bigg) \\
 &= \nabla U\partial_z\big(z^{1-s} \cdot\partial_z\nabla U)\big) +z^{1-s} \Big(|\partial_z\nabla U|^2 -\big(\partial_z|\nabla U|\big)^2\Big) \\
& \ge \nabla U\cdot \partial_z(z^{1-s}\partial_z \nabla U)\, .
\end{split}
\]
Hence, estimating ${\rm Ric}(\nabla U, \nabla U) \ge -C|\nabla U|^2$, we deduce that
\[
c\widetilde{\rm div}(z^{1-s} \widetilde\nabla c) \ge z^{1-s}\big(|\nabla^2 U|^2 -|\nabla|\nabla U||^2 -C |\nabla U|^2\big)\chi_{\{|\nabla u|>0\}}\, .
\]
Inserting this in \eqref{whtiohoew}, we reach 
\[
\int_{\widetilde {B}_1^+} z^{1-s}\mathcal A^2(u) \eta^2 dVdz \le \beta_s \int_{\widetilde {B}_1^+} |\nabla U|^2 z^{1-s} ( |\widetilde \nabla \eta|^2 +C\eta^2)dVdz + \Lambda \Big(\int_{B_1} |\nabla u||\eta |dV\Big)^2 .
\]

From this we conclude the claim in Step 1, fixing a cutoff satisfying $\chi_{\widetilde B_{1/2}^+}  \le \eta \le \chi_{\widetilde B_{3/4}^+}$ and using the estimates for $ \beta_s \int_{\widetilde B_{3/4}^+} z^{1-s}|\nabla U|^2 dVdz$  and$\int_{B_{3/4}} |\nabla u| dV$ proved in Section \ref{BVSection}. In particular, Lemma \ref{lem:whtorwohh} with $R=1, k=0$ and the fact that $\Lambda$-almost stability implies a uniform $BV$ estimates, that is Proposition \ref{BVest}.

\vsp
\textbf{Step 2.} Recall now that, as recorded in Remark \ref{rescblowdef}, if $E$ belongs to  $\mathcal{A}_m^{Blow-up}$ we have sequences of:
\begin{itemize}
\item closed manifolds $(M_j, g^{M_j})$;
\item  points $p_j\in M_j$ and scales $R_j\uparrow \infty$ for which 
${\rm FA}_3( M_j, g^{M_j}, R_j, p_j, \varphi_j)$ holds and $g^{M_j}_{k\ell}(0) = \delta_{k\ell}$.
\item solutions of Allen-Cahn $u_j:M_j \to (-1,1)$ with parameters $\ep_j\downarrow 0$ and Morse index bounded by $m$ such that 
$(u_j\circ \varphi_j) \to u_\circ : = \chi_{E}-\chi_{E^c}$ in $L^1_{\rm loc}(\R^n)$. 
\end{itemize}

Let $U_j: \widetilde M_j \to (-1,1)$ be the extensions of the $u_j$ and and observe that $U_j\rightharpoonup U_\circ$ in weakly in $L^{1}_{\rm loc}(\R^{n+1}_+)$, where $U_\circ$ is the (unique, bounded) Caffarelli-Silvestre extension of $u_\circ$ to $\R^{n+1}_{+}$. Actually, thanks to Theorem \ref{BlowUpConv}, one could prove  local strong convergence in the weighted Hilbert space $H^1_{\rm loc}(\R^{n+1}_+; |z|^{1-s}dxdz)$, although (much rougher) weak $L^{1}_{\rm loc}$ will suffice here.

\vsp
Notice also that in the local coordinates $\varphi_j^{-1}$ we will have
$g^{M_j}_{k\ell} \to \delta_{k\ell}$ in $C^2_{\rm loc}(\R^n) $), since $R_j \to \infty$. Hence by standard elliptic estimates $U_j\circ \widetilde\varphi_j \to U_\circ$ in $C^2_{\rm loc} (\R^n\times (0,+\infty))$  (up to subsequence), where $\widetilde\varphi_j(x,z) = (\varphi_j(x),z) $. 

\vsp
Now, for all $j \gg 1$ sufficiently large, take the  $m+1$ balls $ \{ B_1 (\varphi_j(3i e_3)) \}_{i\le m}$, $i=0,1,\dots, m$. By property (4) ---i.e. almost stability in one out of $(m+1)$ disjoint sets--- $U_j$ will be almost stable in one of them. We may assume without loss of generality (up to translation and subsequence) that it actually is $B_1(\varphi_1(0))$, and then Step 1 gives \[
\int_{\varphi_j(\B_{1/2}^+)} \mathcal A^2(U_j) dV_j \,z^{1-s}dz  \le C\, .
\]

After passing to the limit (using that $U_j\circ \varphi_j\to U_\circ $ in $C^2_{\rm loc}$) we obtain, for every $\delta>0$\,:
\begin{equation}\label{eq: bdd integral in delta}
    \int_{\B^+_{1/2-\delta} \cap \{z>\delta \}} \Big(|D^2 U_\circ|^2-D^2 U_\circ\Big(\tfrac{D U_\circ }{|D U_\circ|}, \tfrac{D U_\circ }{|D U_\circ|}\Big)\Big)  \chi_{\{|D U_\circ|>\delta\}}\, dx \,z^{1-s}dz  \le C. 
\end{equation}

On the other hand, if $ E = \widetilde{E} \times \mathbb{R}^{n-2} $ for some nontrivial cone $ \widetilde{E} \subset \mathbb{R}^2 $, then the associated extension $ U_\circ $ of $u_\circ = \chi_E-\chi_{E^c}$ depends only on the first two variables and is $0$-homogeneous. This implies that $ \mathcal{A}(U_\circ) $ is homogeneous of degree $-2$, leading (since also $ \mathcal{A}(U_\circ) $ is not indentically zero as $E$ is not flat) to a blow-up of the integral 
\begin{equation*}
    \int \mathcal{A}^2(U_\circ) \, dx z^{1-s}  dz
\end{equation*}
around the origin. This would contradict \eqref{eq: bdd integral in delta} and thus $E$ is flat.
\end{proof}

The properties proved so far for the family $\mathcal F = \mathcal{A}_m^{Blow-up}$ (see Definition \ref{BlowUpClass}) show that it satisfies the hypotheses of Theorem \ref{A-CBernstein}, whence we deduce
\begin{corollary}[\textbf{Blow-ups of limit surfaces of Allen-Cahn are hyperplanes}]\label{adfhgsfhd} Let $s\in(0,1)$ and $3\leq n < n_s^*$. Then, any nonempty $\partial F$ in $\mathcal{A}_m^{Blow-up}$ is a hyperplane.
\end{corollary}
\begin{proof}
    It follows from applying Theorem \ref{A-CBernstein} to the class $\mathcal F=\mathcal{A}_m^{Blow-up}$. Properties (1) and (5) in the assumptions of Theorem \ref{A-CBernstein} follow immediately from (the proof of) Theorem \ref{BlowUpConv}. Property (2) follows from (1) in Remark \ref{remprop} and the lower semicontinuity of the BV seminorm under $L^1$-convergence. Properties (4) and (6) have been proved, respectively, in Lemma \ref{lemasbup} and Lemma \ref{2Dcone}. Finally, property (3) ---namely that blow-ups are viscosity solutions of the NMS equation in $\R^n$--- follows easily from Proposition \ref{prop:viscosity} and the convergence of boundaries in Hausdorff distance under a blow-up (Proposition \ref{hausconvprop}), using the convergence of the kernels in Proposition \ref{eucproperties1} part $(ii)$. See \cite{CRS} and \cite{CSapprox}.
\end{proof}

\subsection{Uniform regularity and separation in low dimensions -- Proof of Theorem \ref{UnifReg}}\label{UnifRegSection}
In this section we will prove Theorem \ref{UnifReg}, which stated that sets in $\mathcal{A}_m(M)$, i.e. the limits of Allen-Cahn solutions on $M$ with index at most $m$, are smooth with \textit{uniform regularity} and \textit{separation} estimates in low dimensions.

\vsp
We will need the following improvement of flatness theorem for sets which are viscosity solutions of the NMS equation in a Riemannian manifold, proved in \cite{Moy} more generally assuming boundedness of the nonlocal mean curvature, and which extends the result in \cite{CRS} to the setting of ambient Riemannian manifolds.
\begin{theorem}[\cite{Moy}]\label{improvflat}
 Let $s \in (0,1)$ and $0 < \alpha < s$. Then, there exists $\sigma > 0$, depending on $n$,$s$ and
$\alpha$, such that the following holds. Let $(M
, g)$ be an $n$-dimensional Riemannian manifold.
Take $p\in M$, and assume that the flatness assumption ${\rm FA}_1(M, g, r, p, \varphi)$ holds. Let $E \subset M$ with $p \in \partial E$, and
assume that
\begin{itemize}
    \item[\textit{(i)}] The set $E$ is a viscosity solution of the NMS equation in $\varphi(B_r(0))$, in the sense of Proposition \ref{prop:viscosity}.
    \item[\textit{(ii)}] The boundary $\varphi^{-1}(\partial E )$ is included in a $\sigma$-flat cylinder in $\B_{r}(0)$, that is 
    \begin{equation*}
        \varphi^{-1}(\partial E) \cap \B_{r}(0) \subset \{|e \cdot x| \le \sigma r \},
    \end{equation*}
    for some direction $e\in\Sp^{n-1}$. 
\end{itemize}
    Then $\varphi^{-1}(\partial E)$ is a single $C^{1,\alpha}$ graph in the direction $e$ in $ \B_{r/2}(0)$, with uniform estimates.
\end{theorem}
\begin{proof}[Proof of Theorem \ref{UnifReg}]
We will first show that $E$ is trapped (in the coordinates given by $\varphi^{-1}$) in a very flat cylinder, as recorded in the next claim:

\vsp
\textbf{Claim.} Let $\sigma>0$. Then there exists a uniform constant $R_\sigma=R_{\sigma}(m,s,\sigma)$ and a unit vector $e\in\Sp^{n-1}$ such that
\begin{gather}\label{SigmaFlat1}
\scalebox{1}{$
    -\sigma R_{\sigma} \leq y \cdot e \leq \sigma R_{\sigma} \text{\ \ \ for all } y\in \varphi^{-1}(\partial E) \cap \B_{R_\sigma}.
$}
\end{gather}

\begin{proof}[Proof of the Claim:]
Fix $\sigma>0$; the proof will be by contradiction and blow-up. Let $R_j=1/j$. If the Claim were false, then for every $j\in\N$ there would exist closed manifolds $M_j$ satisfying the flatness assumptions ${\rm FA}_3(M_j,g,1, p_j, \varphi_j)$, and some sets $E_j\in\mathcal{A}_m({M_j})$ so that $p_j\in\partial E_j$ but such that (\ref{SigmaFlat1}) is not satisfied for any unit vector (with $E_j$, $R_j$ and $\varphi_j$ in place of $E$, $R_\sigma$ and $\varphi$).

\vsp
Consider, then, the blow-up sequence $F_j=\frac{1}{R_j}\varphi_j^{-1}(E_j)$. By Proposition \ref{hausconvprop}, a subsequence of the $F_j$ converges (in particular) locally in the Hausdorff distance sense to a limit set $F\in\mathcal{A}_m^{Blow-up}$. Moreover, since $0\in F_j$, we see that $0\in F$ as well.

\vsp
Now, from the classification result of Corollary \ref{adfhgsfhd}, we know that $\partial F$ is, in fact, a hyperplane passing through the origin. The local Hausdorff convergence of the $\partial F_j=\frac{1}{R_j}\varphi_j^{-1}(\partial E_j)$ to the hyperplane $\partial F$ implies then that the condition
\begin{equation*}
\scalebox{1}{$
    -\sigma \leq y \cdot e \leq \sigma  \text{\ \ \ for all } y\in \frac{1}{R_j}\varphi_j^{-1}(\partial E_j) \cap \B_{1}
$}
\end{equation*}
will be satisfied for all $j$ large enough in the subsequence and for $e$ the normal vector to the limit hyperplane. Rescaling this condition by a factor $R_j$, we obtain exactly that (for $j$ large) the $E_j$ satisfy (\ref{SigmaFlat1}) with $E_j, R_j$ and $\varphi_j$, contradiction. This finishes the proof of the claim.
\end{proof}
Now that the claim is known to be true, choosing $\sigma$ in it to be the constant in Theorem \ref{improvflat} (recall that sets in $\mathcal{A}_m(M)$ are viscosity solutions of the NMS equation by Proposition \ref{prop:viscosity}, and that our notion of viscosity solution in Proposition \ref{prop:viscosity} is equivalent to the one used in \cite{Moy} to obtain Theorem \ref{improvflat}), we obtain that $\varphi^{-1}(\partial E)\cap \B_{R_\sigma/2}$ is a single graph with uniform $C^{1,\alpha}$ estimates.
\end{proof}

\subsection{Dimension reduction -- Proof of Theorem \ref{FracYau2}}\label{dimredsection}
This section proves Theorem \ref{FracYau2}. We will call singular points those points $x \in \partial E $ where $\partial E$ cannot be described as a $C^{1, \alpha}$ graph around $x$, and we will denote by $\sing(\partial E)$ or $\sing(E)$ the (closed) set of all the singular points of $\partial E$. We state here a more general result about regularity for $s$-minimal surfaces, which are limits of Allen-Cahn, and immediately show how it proves Theorem \ref{FracYau2}.

\begin{theorem}\label{dimred}
    Let $s\in(0,1)$. Let $(M,g)$ be a closed Riemannian manifold of dimension $n\geq 3$, and let $\partial E\in\mathcal{A}_m(M)$. Then:
    \begin{itemize}
        \item If $n<n_s^*$, then $\partial E$ is a $C^{\infty}$ hypersurface.
        \item If $n=n_s^*$, $\partial E$ is a $C^{\infty}$ hypersurface outside of a discrete set.
        \item If $n>n_s^*$, then $\partial E$ is a $C^{\infty}$ hypersurface outside of a closed set of Hausdorff dimension at most $n-n_s^*$.
    \end{itemize}
\end{theorem}
We readily deduce:
\begin{proof}[Proof of Theorem \ref{FracYau2}]
    The surfaces $\Sigma^{\p}=\partial E^{\p}$ in Theorem \ref{FracYau3} belong to $\mathcal{A}_\p(M)$ by construction (see Section \ref{YauSection} for the proof of Theorem \ref{FracYau3}). Therefore, Theorem \ref{dimred} applies to them, which gives Theorem \ref{FracYau2}.
\end{proof}
Theorem \ref{dimred} will be proved after two preliminary lemmas.

\begin{lemma}\label{singcone}
    Let $\partial E\in\mathcal{A}_m^{Blow-up}$ and let $x\in\sing(\partial E)$. Choose $r_j\to 0$; then, the blow-up sequence $\frac{1}{r_j}(E-x)$ converges in $L^1_{\rm loc}$ and locally in the Hausdorff distance sense to a singular cone $C_\infty\in\mathcal{A}_m^{Blow-up}$ which is stable in $\R^n\setminus\{0\}$. If, moreover, $x$ is an accumulation point of $\sing(\partial E)$, then $r_j$ can be chosen so that $C_\infty$ has a singular point on $\partial B_1$ (thus an entire line of singular points).
\end{lemma}
\begin{proof}
    Recall that $\mathcal{A}_m^{Blow-up}$ satisfies the properties in the statement of Theorem \ref{A-CBernstein}, see the proof of Corollary \ref{adfhgsfhd}. The convergence of the $F_j:=\frac{1}{r_j}(E-x)$ to a cone $C_\infty$ in the appropriate sense follows then as in the beginning of the proof of Theorem \ref{A-CBernstein}, and Lemma \ref{StableCone} gives the stability of $C_\infty$ outside the origin.

    \vsp
    If $C_\infty$ were non-singular (i.e. if $C_\infty$ were a half-space), then the Hausdorff convergence of the $F_j=\frac{1}{r_j}(E-x)$ to $C_\infty$ on $\B_1$ would imply that the assumption of the improvement of flatness result of Theorem \ref{improvflatRn} is satisfied by $E$ on a small ball centered at $x$. Hence $\partial E$ would be a $C^{1,\alpha}$ hypersurface around $x$, and this would contradict the assumption that $x$ is a singular point. Moreover, in case $x$ is a limit point of a sequence $x_j\in\sing(\partial E)$, choosing $r_j:=\text{dist}(x,x_j)$ the $F_j=\frac{1}{r_j}(E-x)$ have singular points at $0$ and at $\frac{1}{r_j}(x_j-x)\in\partial\B_1$. Selecting a subsequence $j_k$ such that the $x_{j_k}$ converge to a limit point $x'\in\partial\B_1$, the improvement of flatness argument above shows that the limit cone of the $F_{j_k}$ must have a singular point at $x'\in\partial\B_1$.
\end{proof}
\begin{lemma}\label{singcone2} Let $C\subset\R^n$ be a cone in $\mathcal{A}_m^{Blow-up}$, with $n\geq n_s^*$. Then $\mathcal{H}^t(\sing(C))=0$ for all $t>n-n_s^*$. Moreover, in the case $n=n_s^*$, $C$ is smooth outside the origin.
\end{lemma}
\begin{proof}
    Fix $t>n-n_s^*$, and assume for contradiction that $\mathcal{H}^t(\sing(C))>0$ (or that $C$ is not smooth outside the origin in the case $n=n_s^*$).

    \vsp
    \textbf{Claim.} If $n>n_s^*$, there exists $x\in \sing(C)\cap\partial \B_1$ such that, blowing up around $x$, we find a cone of the form $\widetilde C \times \R$ (up to a rotation) with $\mathcal{H}^{t-1}(\sing(\widetilde C))>0$.
    \begin{proof}[Proof of the claim]
        Since we are assuming that $\mathcal{H}^t(\sing(C))>0$, there must exist some point $x\in \sing(C)\cap\partial \B_1$ of positive $\mathcal{H}^t_\infty$-density, in the sense that (with the appropriate constant normalization) there exists a sequence $r_j\to 0$ such that $\mathcal{H}^t_\infty(\sing(C)\cap \B_{r_j}(x))\geq r_j^t$ for all $j$. Consider the blow-up sequence $C_j=\frac{1}{r_j}(C-x)$; by Lemma \ref{singcone}, a subsequence will converge locally in the Hausdorff distance sense to a limit cone $C_\infty$ which (since $C$ is itself already a cone) is of the form $C_\infty=\widetilde C_\infty \times \R$, up to performing a rotation. Assume by contradiction that $\mathcal{H}^{t-1}(\sing(\widetilde C_\infty))=0$, or equivalently that $\mathcal{H}^{t}(\sing(C_\infty))=0$.\\
        Now, given any fixed finite cover by open sets of $\sing(C_\infty)\cap \bar\B_1$, for $j$ large enough the $\sing(C_j)\cap \B_1$ are also contained in the cover: otherwise, we would have a subsequence $y_j\in \sing(C_j)$ converging to some $y\in (C_\infty\setminus \sing(C_\infty))\cap \bar\B_1$, so that by Hausdorff convergence the $C_j$ would be contained (for $j$ large enough) in an arbitrarily flat piece of slab around the $y_j$ (thanks to the regularity of $C_\infty$ at $y$); by the improvement of flatness result of Theorem \ref{improvflatRn}, the $y_j$ would be regular points as well, a contradiction. By arbitrariness of the finite open cover of $\sing(C)\cap\bar\B_1$, the assumption that $\mathcal{H}^{t}(\sing(\widetilde C_\infty))=0$ and the definition of $\mathcal{H}^t_\infty$ lead us to deduce that $\mathcal{H}^t_\infty (\sing(C_j)\cap \B_1)$ converges to zero. Scaling back (recall that $C_j=\frac{1}{r_j}(C-x)$), we find that for some $j$ large enough $\mathcal{H}^{t}_\infty(\sing(C)\cap \B_{r_j}(x))\leq \frac{1}{2}r_j^t$, a contradiction with how $x$ was chosen.
    \end{proof}
    
            With the claim at hand, the proof now continues as follows.\\
            In the case $n>n_s^*$, since we assumed $t>n-n_s^*$, the claim can be easily further iterated up to $(n-n_s^*)$ times. This leads to the existence, in the class $\mathcal{A}_m^{Blow-up}$, of a cone of the form $\widetilde C\times \R^{n-n_s^*}$ with $\widetilde C\subset \R^{n_s^*}$ and $\mathcal{H}^{t-(n-n_s^*)}(\sing(\widetilde C))>0$. In particular, $\widetilde C$ is not smooth outside the origin.\\
            In the case $n=n_s^*$, we are already assuming by contradiction that $\widetilde C:=C\subset \R^{n_s^*}$ is not smooth outside the origin.\\
            The rest of the proof is now common for both cases. Let $y\in\R^{n_s^*}$ be such that $y\in\sing(\widetilde C)\cap\partial \B_1$. Blowing up around $(y,0)\in\R^n$, we obtain a new cone that is translation invariant with respect to an additional orthogonal direction. Moreover, this new cone will not be smooth outside the origin either, since otherwise the definition of $n_s^*$ would imply that it is a half-space, and then the Hausdorff convergence and the improvement of flatness result in Theorem \ref{improvflatRn} would give that $\widetilde C$ is smooth around $y$. Iterating this argument, we obtain in the end a cone in $\mathcal{A}_m^{Blow-up}$ which is translation invariant with respect to $n-2$ directions and which is not a half-space by the improvement of flatness argument we have repeatedly been using. Lemma \ref{2Dcone} then gives a contradiction, concluding the proof.
\end{proof}
\begin{proof}[Proof of Theorem \ref{dimred}]
    Let $\partial E\in\mathcal{A}_m(M)$, with $M$ of dimension $n\geq 3$. We distinguish between the three cases depending on $n_s^*$:
    \begin{itemize}
        \item Assume $n_s^*>n$. At every $p\in \partial E$, the flatness assumptions ${\rm FA}_3(M,g,R_0,p,\varphi_p)$ will be satisfied for some $R_0>0$ (recall (d) in Remark \ref{flatscalingrmk}), so that we can apply Theorem \ref{UnifReg} after scaling and conclude the $C^{1,\alpha}$ regularity (in fact, with quantitative estimates) of $\partial E$ around $p$.
        \item Assume $n_s^* < n$. Fix any $t>n-n_s^*$; by Theorem \ref{BlowUpConv} and the arguments in Lemma \ref{singcone}, given any $q\in \sing(E)$ we can blow up around $q$ and find a cone $C_q$. Applying Lemma \ref{singcone2}, we deduce that $\mathcal{H}^t( \sing(C_q))=0$.\\
        Now, assume for contradiction that $\mathcal{H}^t(\sing(E))>0$. We can then apply the same argument as in the Claim in the proof of Lemma \ref{singcone2}, but with $\sing(E)$ instead of $\sing(C)\cap\partial\B_1$; this shows the existence of a point $q\in \sing(E)$ such that, blowing-up around $q$, we would find a cone with $\mathcal{H}^t( \sing(C_q))>0$, thus reaching a contradiction.
        \item Assume $n_s^*=n$. Suppose that $q\in \sing(E)$ is an accumulation point. By Theorem \ref{BlowUpConv} and the arguments in Lemma \ref{singcone}, we can blow up around $q$ and find a cone $C_q$ which is not smooth outside the origin. Lemma \ref{singcone2} then gives a contradiction.
    \end{itemize}

    This proves that $E$ is $C^{1,\alpha}$ outside of a set of the desired size. The fact that $C^{1,\alpha}$ $s$-minimal surfaces are smooth ($C^\infty$) is proved in \cite{FrancS}.
\end{proof}

\subsection{The De Giorgi and Bernstein conjectures in the finite Morse index case -- proof of Theorems \ref{DeGiorgi} and \ref{BernsteinIntro}}\label{DGBSection}

We will now first prove Theorem \ref{DeGiorgi}. We will need the following result, which is a consequence of an improvement of flatness theory for phase transitions in the ``genuinely nonlocal'' regime, meaning that the order $s$ of the operator is strictly less than 1.
\begin{theorem}[Theorem 1.2 in \cite{dPSV}]\label{AsympFlat}
Let $n\ge 2$, $s\in(0,1)$, and $W(u)=\frac 1 4 (1-u^2)^2$.  Let $u:\R^n\rightarrow (-1,1)$ be a  solution of $(-\Delta)^{s/2}u + W'(u)=0$ in $\R^n$.

Assume that there exists a function $a:(1,\infty) \rightarrow (0,1]$ such that  $a(R)\to 0$ as $R\to +\infty$ and
such that, for all $R>0$, we have
\begin{equation}\label{eq: 1D improvement of flatness hypothesis}
\{ e_R\cdot x\le -a(R)R\} \subset \big\{u\le -{\textstyle \frac 4 5} \big\}\subset  \big\{u\le {\textstyle \frac 4 5}\big \} \subset \{e_R\cdot x\le a(R)R\} \quad \mbox{in } \B_{R} \,,
\end{equation}
for some $e_R\in \Sp^{n-1}$
which may depend on $R$.

\vsp
Then, $u(x)=\phi(e\cdot x)$ for some direction $e\in \Sp^{n-1}$ and an increasing function $\phi:\R\to (-1,1)$.
\end{theorem}

\begin{proof}[Proof of Theorem \ref{DeGiorgi}]
    Let \(u:\R^{n}\to(-1,1)\) be a finite Morse index solution of the Allen-Cahn equation with parameter $\ep =1$. For every $R>0$ we introduce the blow-down rescalings
\begin{equation*}
    u_{R}(x):=u(Rx). 
\end{equation*}
These are solutions of the Allen-Cahn equation with parameter $\ep =1/R$ and the same Morse index as $u$. 

By the strong convergence result of Theorem \ref{StrongConv} (whose proof on \(\R^{n}\) is identical to the closed-manifold case, plus a diagonal argument to get convergence in each of the balls $\B_{k}$ for $k\in \N$), there exists a sequence $ R_{j}\to\infty $ and an $s$-minimal surface $E \subset \R^n$ such that
\begin{equation*}
u_{R_{j}} \longrightarrow  u_{E}:=\chi_{E}-\chi_{E^{c}}
\quad\text{in }L^{1}_{\mathrm{loc}}(\R^{n}).
\end{equation*}
In particular, $E$ belongs to the class $\mathcal{A}_m(\R^n)$ (see Definition \ref{AClimits}). The properties in the hypotheses of Theorem \ref{A-CBernstein} can be proved for the class $\mathcal{A}_m(\R^n)$, exactly as they were proved for the class $\mathcal{A}_m^{Blow-up}$ as recorded in Corollary \ref{adfhgsfhd}. In fact, all necessary results have been stated with local assumptions, other than the use of the kernel $K_s(x,y)$, which becomes $\alpha_{n,s} |x-y|^{-n-s}$ on $\R^n$, and in fact, several proofs could be simplified due to working on $\R^n$.

Applying Theorem \ref{A-CBernstein} to the class \(\mathcal{A}_{m}(\R^{n})\) we deduce that $E$ must be a half-space. Moreover, the local convergence in the Hausdorff distance of the level sets of $u_{R_j}$ to $\partial E$---see Theorem \ref{StrongConv}---shows that both $\{u_{R_j} \le -4/5 \}$ and $\{u_{R_j} \le 4/5 \}$ converge (in Hausdorff distance) in $\B_1$ to a half-plane. Rescaling back to $u$ gives that \eqref{eq: 1D improvement of flatness hypothesis} is eventually satisfied in $\B_{R_j}$. Then, Theorem \ref{AsympFlat} gives that $u$ is a one-dimensional solution.  
\end{proof}

We turn now to the proof of Theorem \ref{BernsteinIntro}, the finite Morse index analog for $s$-minimal surfaces of class $C^2$ of the Bernstein conjecture. We recall that this result is false for classical minimal surfaces since for example, the catenoid in $\R^3$ is a complete minimal surface with Morse index $1$, and that even under the assumption of stability, this result is only known up to dimension $4$ despite stable classical minimal cones being known to be hyperplanes up to dimension $7$. See the Introduction for more details.\\
To prove Theorem \ref{BernsteinIntro}, we will again apply the classification result of Theorem \ref{A-CBernstein}. For that reason, we introduce the following definition.

\begin{definition}
We say that a set $E\subset \R^n$ belongs to the class $\mathcal A'_m(\R^n)$ if 
there exists a sequence of sets $E_j\subset \R^n$ with $E_j \to E$ in $L^1_{\rm loc}(\R^n)$ such that:
\begin{itemize}
 \item[\textit{(i)}] the boundaries $\partial E_j$ are $(n-1)$-dimensional manifolds of class $C^2$;
 \item[\textit{(ii)}] $E_j$ are critical sets for the $s$-perimeter in $\R^n$ with Morse index $\leq m$ in the weak sense.
\end{itemize}
\end{definition}

\begin{proposition}\label{good2}
Let $n\geq 3$. Then the family $\mathcal{A}'_m(\R^n)$ satisfies the properties in the hypotheses of Theorem \ref{A-CBernstein}.
\end{proposition}
\begin{proof}
For $E$ an $s$-minimal surface of class $C^2$ and Morse index at most $m$ in $\R^n$:
\begin{itemize}
    \item A uniform BV estimate holds, which one can see by considering the proof for stable $s$-minimal surfaces in \cite{CSV} and adapting it to the finite Morse index case using the ideas in the present paper, or by directly interpreting the arguments in the present paper for $s$-minimal surfaces of class $C^2$ and Morse index at most $m$ in $\R^n$ instead of for solutions of the Allen-Cahn equation. More precisely, the $C^2$ assumption allows one to use the geometric formula on $\R^n$ for the second variation of the fractional perimeter (see Theorem 6.1 and Remark 6.2 in \cite{FFMMM}) instead of the second-variation formula for the Allen-Cahn energy, with which one first proves an almost-stability inequality like the one in Lemma \ref{asineq}. Afterwards, one shows the BV estimate arguing as in Section \ref{BVSection}. See \cite[Theorem 5.4]{Florit} for a full proof.
    
    \item The surface $E$ is a viscosity solution of the NMS equation, since it is stationary and of class $C^2$ (in particular, its nonlocal mean curvature is well defined and equal to $0$ at every boundary point).

    \item The almost-stability inequality (4) is proved exactly like the one in Lemma \ref{improvineq1} considering the formula for the second variation of the fractional perimeter and test functions $\xi_i=X_i\cdot\nu_{\partial E}$ instead, where $\nu_{\partial E}$ denotes the outer normal vector. 
\end{itemize}
If we then consider $E$ to be any element of $\mathcal{A}_m'(\R^n)$, not necessarily of class $C^2$, by definition we can approximate it with $E_i$ satisfying the above. This readily shows that $E$ inherits the properties of the $E_i$, which proves (1)-(5) for the class $\mathcal{A}_m'(\R^n)$. Regarding property (6), concerning the classification of cones in $\mathcal{A}_m'(\R^n)$ with $n-2$ directions of translation invariance, it is proved as in Lemma 7.7 in \cite{Stable}, the latter dealing with limits of Allen-Cahn solutions in the stable case and in $\R^n$. Instead of the inequality (7.17) in \cite{Stable}, which is stated for solutions of Allen-Cahn, one considers the second variation formula for the fractional perimeter. The almost-stability inequality (4) together with the assumption of having at least one direction of translation invariance (recall that we are assuming $n\geq 3$) then results in an inequality analogous to (7.17) in \cite{Stable}, with an additional term coming from the assumption of almost-stability (instead of stability) but which is immediately seen to be bounded thanks to the BV estimate.
\end{proof}

\begin{proof} [Proof of Theorem \ref{BernsteinIntro}]
Thanks to Proposition \ref{good2}, we can apply Theorem \ref{A-CBernstein} to the family $\mathcal{A}_m'(\R^n)$ and conclude the result.
\end{proof}

\appendix

\section{Hölder estimates for nonlocal equations}

\begin{lemma}\label{fracholest} Let $v: \R^n \to \R$ be Lipschitz function with $\| v\|_{L^\infty (\R^n)}\le C_\circ$ satisfying $|L v(x)| \le C_\circ $ for every $x \in \B_1(0)$, where $L$ is an integro-differential operator of order $s\in (0,1)$ of the integral form 
\begin{equation*}
    Lu(x) = \int_{\R^n} (u(x)-u(y))K(x, y) \, dxdy \,,
\end{equation*}
and $K$ is a nonnegative kernel comparable to the one of the fractional $s$-Laplacian, that is satisfying
\begin{equation}\label{cdcdcd}
     \frac{c}{|y|^{n+s}} \le K(x, x-y) \le \frac{C}{|y|^{n+s}}  \s \forall \, x,y\in \R^n ,
\end{equation}
for some constants $c,C>0$. Then

\begin{equation}\label{78uijky7}
    [v]_{C^{\alpha}(\B_{1/2}(0))} \le C(n,s) C_\circ \,,
\end{equation}
for some small positive $\alpha=\alpha(n,s)$. 
    
\end{lemma}
\begin{proof}
The result is a standard consequence of  \cite[Theorem 5.1]{silreg}. Let us point out that Theorem 5.1 in \cite{silreg} would seem to require assumption \cite[(2.2)]{silreg} to hold for all $r>0$. However, it is clear from its (very short) proof that \eqref{78uijky7} only requires assumption \cite[(2.2)]{silreg} to be verified at ``small'' scales $r\in (0,1)$ 
(and in our setting this can be easily verified using \eqref{cdcdcd}).
\end{proof}

\subsection*{Acknowledgements} 
M.C. would like to thank the FIM (Institute for Mathematical Research) at ETH Z\"{u}rich for the wonderful hospitality during his many stays in 2022-2023. J.S. is supported by the European Research Council under Grant Agreement No 948029. E.F-S has been partially supported
by the CFIS Mobility Grant, the MOBINT-MIF Scholarship from AGAUR, and the support of a fellowship from ”la Caixa” Foundation (ID 100010434)”. The
fellowship code is “LCF/BQ/EU22/11930072”.


\end{document}